\documentclass[11pt]{article}
\usepackage{amsfonts}
\usepackage{lineno,hyperref}
\usepackage{bbm}
\usepackage{amsmath}
\usepackage{mathrsfs}
\usepackage{amssymb}
\usepackage{amsmath,color}
\usepackage{amsthm}
\usepackage{amstext}
\usepackage{amsopn}
\usepackage{texdraw}
\usepackage{graphicx}
\usepackage{multirow}
\usepackage{epsfig,subfig,color}
\usepackage{lscape}
\usepackage{rotating}
\usepackage{float}
\usepackage{cite}
\usepackage{srcltx}
\usepackage[square, comma, sort&compress, numbers]{natbib}
\usepackage{appendix}
\newtheorem{theorem}{Theorem}[section]

\newtheorem{proposition}[theorem]{Proposition}

\theoremstyle{definition}

\theoremstyle{remark}

\numberwithin{equation}{section}

\graphicspath{{figequilibrium/}}
\date{}
\begin{document}
\title{ \bf\large{Double Hopf bifurcation analysis in the  memory-based diffusion system}\footnote{Partially supported by the National Natural Science Foundation of  China (Nos.11971143, 12071105), and  Natural Science Foundation of Zhejiang Province of China  (No.LY19A010010).}}
%-----------------------------------------------------------------------------------------------
\author{Yongli Song\textsuperscript{1}, \  \ Yahong Peng\textsuperscript{2}\footnote{Corresponding author, Email: pengyahong@dhu.edu.cn },  \  \ Tonghua Zhang\textsuperscript{3}  \\
{\small \textsuperscript{1} School of Mathematics, Hangzhou Normal University, Hangzhou 311121, China\hfill {\ }}\\
{\small \textsuperscript{2} Department of Mathematics, Donghua University, Shanghai,
201620, China\hfill{\ }}\\
{\small \textsuperscript{3} Department of Mathematics, Swinburne University of Technology, Hawthorn, VIC3122,  Australia\hfill {\ }}\\
}
%---------------------------------------------------------------------------------------
\maketitle
\begin{abstract}
In this paper,  we derive the algorithm for calculating the normal form of the double Hopf bifurcation that appears in a memory-based diffusion system via taking memory-based diffusion coefficient and the memory delay as the perturbation parameters. Using the obtained theoretical results, we study  the dynamical classification near the double Hopf bifurcation point  in a predator-prey system with Holling type II functional response. We show the existence of different kinds of stable spatially inhomogeneous periodic solutions, the transition from one kind to the other as well as the coexistence of two types of periodic solutions  with different spatial profiles by varying the memory-based diffusion coefficient and the memory delay.
\end{abstract}


%\baselineskip=16pt

%=========================================================
\begin{section}
{Introduction}
 \end{section}
%=========================================================

 The  complex internal mechanism of memory-driven movement is still poorly understood although it has been known  by many biologists that the spatial  memory has an important influence on the animal movement,  and  results in  complex mathematical and
computational challenges \cite{FaganLewis-13EL,Fagan-PNAS2019}. In this regard, mathematical models may provide deep insights into the theoretical mechanism behind the biological phenomenon.  For example, considering the spatial memories decay over time and the fact that animal movement is affected by the population at the past  time \cite{FaganLewis-13EL},  Shi {\it et al.} \cite{Shi-WWY-JDDE2020} introduce a time delay (also known as “memory delay”)  into the advection term of the classic reaction-diffusion-advection equation and  propose  a memory-based diffusion equation to model the dynamics of animal movement  with memory.   From the theoretical analysis in \cite{Shi-WWY-JDDE2020} authors found that the stability of a spatially homogeneous steady state depends on the reaction term and the relationship of the coefficients  of random diffusion and the directional diffusion,  but not the memory delay. Since then many researchers have shown their interests in modeling and investigating the dynamics  for the single-species model with the memory  \cite{Shi-WW-Nonlinearity2019,Song-Wu-Wang-JDE2019,An-Wang-Wang-DCDS2020,Wang-FW-JDDE2021,Shi-Shi-Wang-JMB2021,Song-WuWang-AMC2021,Oliveira-Kar-Ber-MB2020}.
    
 More recently,  Song {\it et al.} \cite{Song-SW-SAPM2021}  consider a two species model, where authors assume that the prey, such as plants or ``drunk" animals are considered as resource so that they have no or negligible memory or cognition. Thus they only introduce a spatial memory into the predator, and propose the following predator-prey system 
\begin{equation}
\label{PPMODEL}
  \begin{cases}\frac{\partial u(x,t)}{\partial t}=d_{11}u_{xx}(x,t)+f(u(x,t),v(x,t)), ~0<x<\ell\pi, t>0, \vspace{0.1cm}\\
\frac{\partial v(x,t)}{\partial t}=d_{22}v_{xx}(x,t)-d_{21}(v(x,t)u_{x}(x,t-\tau))_x+g(u(x,t),v(x,t)), ~0<x<\ell\pi, t>0,\vspace{0.1cm}\\
u_x(0,t)=u_x(\ell\pi,t)=v_x(0,t)=v_x(\ell\pi,t)=0, ~t\geq0,
  \end{cases}
 \end{equation}
where $u(x, t)$ and $v(x, t)$ are the density of the prey and predator, respectively, at the space $x$ and time $t$,  
 $d_{11}$ and $d_{22}$ are the random diffusion coefficients of the prey and predator respectively,   $d_{21}$ is the memory-based diffusion coefficient of the predator,  $\tau$ is the time delay representing the averaged memory period of  the predator, $f$ and $g$ are biological birth/death of prey and predator respectively.  The effects of  the memory-based diffusion coefficient $d_{21}$  and  time delay  $\tau$  on the stability of the positive constant equilibrium of system \eqref{PPMODEL}  has been investigated in \cite{Song-SW-SAPM2021} and it has been shown that unlike the classic prey-taxis model, memory-based prey-taxis destabilizes  the positive constant equilibrium, which is a new mechanism for spatiotemporal pattern formation.   The spatially inhomogeneous Hopf bifurcation, double Hopf bifurcation and stability
switches are also found in  \cite{Song-SW-SAPM2021}.  The algorithm for computing the normal form to investigate the spatially inhomogeneous Hopf bifurcation are developed in  \cite{Song-PZ-JDE2021}.  

In this paper, we are interested in the complex dynamics due to the interaction of two spatially inhomogeneous Hopf bifurcations.  For this purpose, we shall first develop the algorithm for computing the normal form of double Hopf bifurcation for system \eqref{PPMODEL}. Recall that the standard Hopf bifurcation happens when the equilibrium loses its stability with a pair of purely imaginary eigenvalues at the bifurcation value.  For this Hopf bifurcation,  there exists a family of the periodic solutions with small amplitudes  near the neighbourhood of the equilibrium when the bifurcation parameter is taken in the  unilateral neighbourhood of the bifurcation value. However,  the interaction of  Hopf bifurcations may result in more complex dynamics like quasi-periodic solution or invariant torus \cite{KUZ-BOOK1998}.  The  interaction  of two Hopf bifurcations, which have a pair of purely
imaginary eigenvalues $\pm i \omega_1$ and $\pm i\omega_2$, $\omega_j>0, j=1, 2$, respectively, is interpreted in the framework of the double Hopf bifurcation also known as Hopf-Hopf bifurcation.   When the ratio of  $\omega_1/\omega_2$  is irrational, the bifurcation is said to be non-resonant, otherwise resonant.  The resonant double Hopf bifurcation is distinguished into two cases:  weakly and strongly resonant double Hopf bifurcations.  The bifurcation is said to be strongly  resonant if there exit  two positive integers $m_1,  m_2$ so that
\begin{equation}
\label{CforResonant}
m_1\omega_1  = m_2\omega_2,  ~m_1+m_2\leq 4,
\end{equation}
and to be weakly resonant if  there are no $m_1, m_2$ satisfying $m_1+m_2\leq 4$ such that the condition \eqref{CforResonant} holds.   The weakly resonant double Hopf bifurcation is often codimension-two, but the strongly resonant  case is often codimension-three and it is  more difficult to analyze the related dynamics \cite{Luongo-AD-ND2003,Luongo-DA-CS2004}.  

 The two most popular approaches used to investigate the bifurcations are the rigorous centre manifold reduction and normal form theory  \cite{HALE-1977BOOK,Hassard-book1981,Chow-Hale-book82,KUZ-BOOK1998,Faria-95JDE,Faria-00TAMS,Yu-DJ-IJBC2014} and the method of multiple scales  \cite{Nayfeh-book1991,Luongo-AD-ND2003,Luongo-DA-CS2004,Nayfeh-ND2008}.  For other  methods to investigate the bifurcations such as the method of small parameters or the theory of averaging, please refer to \cite{Hale-Lunel-JIEA1990,San-book2007}  and references therein. The complex dynamics arising form  double Hopf bifurcation has  been  recently studied by many authors   for various dynamical systems,  refering to  \cite{Kno-P-PRSLA1988,Van-KL-Nonlinearity1990,Li-JDE2016, Rev-AM-SIADS2015,Zhang-ZS-CNSNS2021} for ordinary differential equations, 
 to \cite{Cam-Le-JDDE1995, Ding-JiangYu-AMC2013,Ding-CaoJiang-ND2016,Wang-XS-IJBC2013,Jiang-Song-15AMC,Mon-DIS-PRSCA2017,Gen-IM-ND2018,Pei-Wang-AMC2019,Bos-JK-SIADS2020,Wang-LW-MCS2021} for delay differential equations. More recently,  based on the theory of  normal forms for partial functional differential equations  developed by Faria \cite{Faria-00TAMS}, the double Hopf bifurcation in the reaction-diffusion system with delay has attracted the attention of the researchers \cite{Du-Niu-Guo-Wei-JDDE2020,Duan-NiuWei-CSF2019,Chen-WLC-AMM2020,Liu-Wei-NAMC2021}. The solutions of these systems involve not only the time but also the space, thus the investigation of the bifurcation phenomenon is more difficult \cite{zhang2018a,peng2021a}. The idea in  \cite{Faria-00TAMS} has also been successfully used to calculate the normal form of the Turing-Hopf bifurcation in the reaction-diffusion system with/without delay \cite{Song-ZP-16CNSNS,Song-Jiang-Yuan-JAAC2019,Jiang-An-Shi-JDE2019}. Unfortunately,  the procedure of calculating the normal forms of the double Hopf bifurcation for the classical reaction-diffusion system with delay can not apply to the system \eqref{PPMODEL}, where the delay is involved the diffusion term not in the reaction term and the diffusion terms are not linear.

 %%%%%%%%%

Motivated by the recent results, particularly the aforementioned, on the double Hopf bifurcation for the reaction-diffusion systems with delays,  in this paper, we investigate the dynamics associated with the double Hopf bifurcation arising from \eqref{PPMODEL} and our results can be summarized as follows: 
\begin{enumerate}
\item We derive the algorithm for computing the normal form of the double Hopf bifurcation induced by the memory-based diffusion coefficient and memory delay  for the memory-based diffusion system \eqref{PPMODEL}.  The explicit relationship between the second and third terms of the normal form and those in \eqref{PPMODEL} near the positive equilibrium are completely established; 
 \item  For system \eqref{PPMODEL} with Holling-type II functional response,  the dynamical classification near the double Hopf bifurcation point are determined for two cases: (i) the interaction of two Hopf bifurcations with the same spatial mode-$2$; and (ii)  the interaction of two Hopf bifurcations with the  spatial mode-$1$ and mode-$2$;
 \item We find different kinds of stable spatially inhomogeneous periodic solutions and the transition from one to another near the neighbourhood of the double Hopf bifurcation point. Especially, for case 2(i), we find the  stable quasi-periodic solution  like a “bird” with the spatial mode-1, and for case 2(ii)  we find the bistability region of two kinds of stable periodic solutions with spatial mode-$1$ and mode-$2$. 
 \end{enumerate}

%%%%%%%%%%%%%%%%%%%%%%%%%%%%%%%%%%%%%%%%%%%%%%%%%%%5

%======================================
The the rest of the paper is organized as follows. In Section 2, we derive the algorithm for computing the normal form associated with the double Hopf bifurcation for  \eqref{PPMODEL}. In Section 3, we study the dynamical classification near the double Hopf bifurcation point for  \eqref{PPMODEL}  with  Holling type-II functional response by employing the theoretical results developed in Section 2.  We conclude our study with a short discussion in Section 4. Finally, we give all detailed calculations used in Section 2 in the Appendices. Throughout the paper, $\mathbb{N}$ represents the set of all positive integers, and $\mathbb{N}_0=\mathbb{N}\cup {0}$ represents the set of all nonnegative integers.

%-------------------------------------------------------------------------------------------
\begin{section}
   {Normal forms  for the double Hopf bifurcation and Hopf bifurcation}
 \end{section}
 %-------------------------------------------------------------------------------------------
\begin{subsection}
   {Basic assumptions}
 \end{subsection}
  %-------------------------------------------------------------------------------------------
    We   assume that system \eqref{PPMODEL} has a   positive constant equilibrium $E_*(u_*, v_*)$ and $f, g\in \mathbb{C}^3$  near the neighbourhood of $E^*$  for the calculation of the normal form. Then the characteristic equation of  the  linearized system of  \eqref{PPMODEL} at the positive equilibrium  $E_*(u_*, v_*)$  is
\begin{equation}
\label{CE0}
\prod\limits_{n\in \mathbb{N}_0} \Gamma_n(\lambda)=0,
\end{equation}
where $\Gamma_n(\lambda)=\mathrm{det}\left(\mathcal{M}_n(\lambda)\right)$ with characteristic matrix
\begin{equation}
\label{CEN}
\mathcal{M}_n(\lambda)=\lambda I_2+(n/\ell)^2D_1+(n/\ell)^2e^{-\lambda\tau}D_2-A,
\end{equation}
 where $I_2$ is a $2\times 2$ identity matrix, 
 \begin{equation}
\label{DJA}
D_1=\left(\begin{array}{cc}
d_{11} & 0\vspace{0.1cm}\\
0 & d_{22}
\end{array}\right),~D_2=\left(\begin{array}{cc}
0 & 0\vspace{0.1cm}\\
-d_{21}v_*    & 0
\end{array}\right),  ~A=(a_{ij})_{2\times 2},
\end{equation}
with
 $$
a_{11}=\frac{\partial f(u_*, v_*)}{\partial u }, a_{12}=\frac{\partial f(u_*, v_*)}{\partial v }, a_{21}=\frac{\partial g(u_*, v_*)}{\partial u }, a_{22}=\frac{\partial g(u_*, v_*)}{\partial v }
$$

  From \eqref{CEN}, we have
\begin{equation}
\label{CE}
\Gamma_n(\lambda)=\lambda ^2-T_n\lambda+\widetilde{J}_n(\tau)=0,
\end{equation}
where
\begin{equation}
\label{TK}
T_n=Tr(A)-Tr(D_1)(n/\ell)^2,
\end{equation}
$$
\widetilde{J}_n(\tau)=d_{11}d_{22}(n/\ell)^4-\left(d_{11}a_{22}+d_{22}a_{11}+d_{21}v_*a_{12}e^{-\lambda\tau}\right)(n/\ell)^2+Det(A).$$

Assume that  at $(\tau, d_{21})=(\tau_c, d_{21}^c)$, Eq.\eqref{CE0} has two pairs of purely imaginary  roots $\pm i\omega_1$ and $\pm i\omega_2$, respectively, for $n=n_1$ and $n=n_2$ with $n_1\leq n_2$ and  the corresponding transversality condition holds,   and  all other eigenvalues  have negative real parts. We are interested in the cases of weak resonance  and   non-resonance, which are codimension-two bifurcation problem, i.e. the strong resonance condition  \eqref{CforResonant}  doses not hold for these two $\omega_1$ and $\omega_2$.  
 
%-----------------------------------------
 \vspace{0.4cm}
 \begin{subsection}
{Normal form for the double Hopf bifurcation}  \label{Sec-NFDhopf}
\end{subsection}
%----------------------------------------

  In what follows, we set   $\tau=\tau_c+\mu_1,  d_{21}=d_{21}^c+\mu_2$ such that $\left(\mu_1, \mu_2\right)=(0, 0)$ is the double Hopf bifurcation value for Eq.\eqref{PPMODEL}. In this section, $D_1$ is defined by \eqref{DJA} and 
$$
D_2^c=\left(\begin{array}{cc}
0 & 0\vspace{0.1cm}\\
-d_{21}^cv_*    & 0
\end{array}\right).
$$

Define the real-valued Sobolev
space
\begin{equation*}
\mathscr{X}=\left\{U=\left(U_1, U_2\right)^T\in \left(W^{2,2}(0, \ell\pi)\right)^2,\frac{\partial U_1}{\partial
x}=0, \frac{\partial U_2}{\partial
x}=0,~ x=0, \ell\pi\right\},
\end{equation*}
and let  $\mathscr{C}:=C\left([-1,0]; \mathscr{X}\right)$ be the Banach space of continuous mappings from $[-1, 0]$ to $\mathscr{X}$ .

Translating $E_*$ to the origin by setting
$$U(x,t)=\left(U_1(x,t),  U_2(x,t)\right)^T=(u(x,t), v(x,t))^T-(u_*, v_*)^T,$$
normalizing the delay by the time-scaling $t\rightarrow t/\tau$, and then for simplification of notation,  writing $U(t)$ for $U(x, t)$ and $U_t\in \mathscr{C}$ for $U_t(\theta)=U(x, t+\theta), -1\leq \theta\leq 0$,  \eqref{PPMODEL} becomes
\begin{equation}
\label{NFEQ}
\frac{dU(t)}{dt}=d(\mu)\Delta (U_t)+L(\mu)(U_t)+F(U_t, \mu),
\end{equation}
where   for $\varphi=(\varphi^{(1)},\varphi^{(2)})^T\in \mathscr{C}$,  $d(\mu)\Delta, L(\mu): \mathscr{C}\rightarrow \mathscr{X}$, $F: \mathscr{C}\times \mathbb{R}^2\rightarrow \mathscr{X}$  are given, respectively, by
$$d(\mu)\Delta (\varphi)=d_0\Delta (\varphi)+F^d (\varphi, \mu), ~L(\mu) (\varphi)=(\tau_c+\mu_1) A\varphi(0),$$
 and
\begin{equation}
\label{FPM}
  F(\varphi, \mu) =(\tau_c+\mu_1) \left(\begin{array}{c}
f\left(\varphi^{(1)}(0)+u_*, \varphi^{(2)}0)+v_*\right)\vspace{0.2cm}\\
g\left(\varphi^{(1)}(0)+u_*, \varphi^{(2)}(0)+v_*\right)\end{array}\right)-L(\mu)(\varphi),
\end{equation}
where
 $$
 d_0\Delta (\varphi)=\tau_cD_1\varphi_{xx}(0)+\tau_cD_2^c\varphi_{xx}(-1),
$$
%------------------------------------------------------------------------------------------------------------------------------------------------------------------------
\begin{equation}
\label{FD}
\begin{array}{lll}
&&F^d (\varphi, \mu)\vspace{0.2cm}\\
&=&
\tau_c\left(\begin{array}{c}
0\vspace{0.2cm}\\
-d_{21}^c\left(\varphi^{(1)}_{x}(-1)\varphi^{(2)}_{x}(0)+\varphi^{(1)}_{xx}(-1)\varphi^{(2)}(0)\right)\end{array}\right)\vspace{0.3cm}\\
&&
+\mu_1\left(\begin{array}{c}
d_{11}\varphi^{(1)}_{xx}(0)\vspace{0.2cm}\\
-d_{21}^c v_*\varphi^{(1)}_{xx}(-1) +d_{22}\varphi^{(2)}_{xx}(0)\end{array}\right)
-\tau_c\mu_2\left(\begin{array}{c}
0\vspace{0.2cm}\\
v_*\varphi^{(1)}_{xx}(-1)\end{array}\right)
\vspace{0.3cm}\\
&&
-\left(d_{21}^c \mu_1+\tau_c\mu_2\right)\left(\begin{array}{c}
0\vspace{0.2cm}\\
\varphi^{(1)}_{x}(-1)\varphi^{(2)}_{x}(0)+\varphi^{(1)}_{xx}(-1)\varphi^{(2)}(0)\end{array}\right) \vspace{0.3cm}\\
&&-\mu_1\mu_2\left(\begin{array}{c}
0\vspace{0.2cm}\\
v_*\varphi^{(1)}_{xx}(-1)\end{array}\right)-\mu_1\mu_2\left(\begin{array}{c}
0\vspace{0.2cm}\\
\varphi^{(1)}_{x}(-1)\varphi^{(2)}_{x}(0)+\varphi^{(1)}_{xx}(-1)\varphi^{(2)}(0)\end{array}\right).
\end{array}
 \end{equation}
%------------------------------------------------------------------------------------------------------------------------------------------------------------------------

  Noticing that $\mu_1, \mu_2$ are perturbation parameters and treated as variables in the calculation of normal forms,  we denote
  $L_0(\varphi)=\tau_cA\varphi(0)$
 and rewrite  \eqref{NFEQ}   as  the following linear form from nonlinear terms
  \begin{equation}\label{NFEQ1}
\frac{dU(t)}{dt}=d_0\Delta (U_t)+L_0(U_t)+\widetilde{F}(U_{t},  \mu_1, \mu_2),
\end{equation}
where for $\varphi=(\varphi_{1},\varphi_{2})^T\in \mathscr{C}$,
\begin{equation}
\label{NLTERM}
\widetilde{F}(\varphi, \mu_1, \mu_2)=\mu_1A\varphi(0)+F(\varphi, \mu)+F^d(\varphi, \mu).
\end{equation}

The characteristic equation for the linearized system
 \begin{equation}\label{LNFEQ1}
\frac{dU(t)}{dt}=d_0\Delta (U_t)+L_0(U_t)
\end{equation}
is
\begin{equation}
\label{NFCE}
\prod\limits_{n\in \mathbb{N}_0} \widetilde{\Gamma}_n(\lambda)=0,
\end{equation}
where  $\widetilde{\Gamma}_n(\lambda)=\mathrm{det}\left(\widetilde{\mathcal{M}}_n(\lambda)\right)$ with
\begin{equation}
\label{NFMN}
\widetilde{\mathcal{M}}_n(\lambda)=\lambda I_2+\tau_c(n/\ell)^2D_1+\tau_c(n/\ell)^2e^{-\lambda}D_2^c-\tau_cA.
\end{equation}

Comparing   \eqref{NFMN} with \eqref{CEN}, we know that
Eq.\eqref{NFCE}  has two pairs of purely imaginary  roots
$\pm i\omega_{1c}$ and $\pm i\omega_{2c}$  for $n=n_1$ and $n=n_2$, respectively,  and  all other eigenvalues  have negative real parts, where
$\omega_{jc}=\tau_c\omega_{j}, j=1,2$.

 It is well known that the eigenvalue problem
 $$
-\gamma''=\mu \gamma,~~x\in(0, \ell\pi);~~\gamma'(0)=\gamma'(\ell\pi)=0
$$
\ has eigenvalues $\mu_{n}=(n/\ell)^{2}, n\in \mathbb{N}_0 $,
with corresponding normalized eigenfunctions
$$\gamma_n(x)=\frac{\cos\left(\frac{nx}{\ell}\right)}{\|\cos\left(\frac{nx}{\ell}\right)\|_{2,2}}=\left\{
\begin{array}{ll}
  \frac{1}{\sqrt{\ell \pi}}, & ~\mathrm{for}~n=0,\vspace{0.2cm} \\
  \frac{\sqrt{2}}{\sqrt{\ell \pi}}\cos\left(\frac{nx}{\ell}\right), & ~\mathrm{for}~n\neq 0,
\end{array}
\right.$$
where the norm $\|\cdot\|_{2,2}$ is induced by the inner product $\left[\cdot, \cdot\right]$ as follows
\begin{equation*}
\left[ u,v \right]=\int_0^{\ell\pi}u^T v
dx, ~\mbox{for}~  u, v \in \mathscr{X}.
\end{equation*}

Let $\beta_n^{(j)}= \gamma_n(x)e_j, j=1, 2$,
where $e_j$ is the unit coordinate vector of $\mathbb{R}^2$,
and $\mathscr{B}_n=\mbox{span}\left\{\left[v(\cdot), \beta_n^{(j)}\right]\beta_n^{(j)}|~ v\in \mathscr{C}, j=1, 2\right\}$. Then it is easy to verify that
$$L_0(\mathscr{B}_n)\subset \mbox{span}\left\{\beta_n^{(1)}, \beta_n^{(2)}\right\}, n\in \mathbb{N}_0.$$ 
Assume that $z_t(\theta)\in C=C\left([-1,0], \mathbb{R}^2\right)$ and
 $$z_t^T(\theta) \left(\begin{array}{c}
                             \beta_n^{(1)} \vspace{0.1cm}\\
                             \beta_n^{(2)}
                           \end{array}
\right)\in \mathscr{B}_n.$$
Then, on $\mathscr{B}_n$, the linearized equation  \eqref{LNFEQ1} is equivalent to the following  functional differential equation (FDE) in $C$:
\begin{equation}
\label{FDE}
\dot{z}(t)=L_0^d\left(z_t(\theta) \right)+ L_0(z_t(\theta)),
\end{equation}
where
$$
L_0^d\left(z_t(\theta) \right) =\tau_c\left(\begin{array}{cc}
                  -d_{11} (n/\ell)^2 & 0 \\
                   0 &   -d_{22}(n/\ell)^2
                 \end{array}
\right) z_t(0)+\tau_c\left(\begin{array}{cc}
                  0 & 0 \\
                d_{21}^cv_*(n/\ell)^2 &  0
                 \end{array}
\right) z_t(-1).
$$
The characteristic equation of  linear system  \eqref{FDE} is the same as  given in  \eqref{NFCE}.

Define $\eta_n(\theta)\in BV([-1, 0], \mathbb{R}^2)$
such that
$$
\int^0_{-1}d\eta_n(\theta) \varphi(\theta)=L_0^d(\varphi(\theta)) + L_0(\varphi(\theta)), ~~\varphi\in C,
 $$
and use the adjoint bilinear form on $C^*\times C$, $C^*=C([0, 1], \mathbb{R}^{2*})$, where $\mathbb{R}^{2*}$ is the 2-dimensional
space of row vectors,   as follows
$$
\langle\psi(s),
\varphi(\theta)\rangle_n=\psi(0)\varphi(0)-\int^0_{-1}\int^{\theta}_0
\psi(\xi-\theta)d\eta_n(\theta)\varphi(\xi)d\xi,\ \ \
\textmd{for}\ \psi\in C^*, \varphi\in C.
$$

Let
$\Lambda=\{i\omega_{1c}, -i\omega_{1c},i\omega_{2c}, -i\omega_{2c}\}$. Denote the generalized eigenspace of \eqref{FDE} associated with $\Lambda$  by $P_{n_j}$ and  the corresponding adjoint space by
$P_{n_j}^*$. Then, by the adjoint theory of functional differential equation \cite{HALE-1977BOOK}, $C$ can be decomposed as
$C=P_{n_j}\oplus Q_{n_j}$, $j=1,2,$ where $Q_{n_j}=\{\varphi\in C:
\langle\psi,\varphi\rangle=0,\forall\psi\in P_{n_j}^*\}$. Choose the
bases $\Phi_{n_j}(\theta)$ and  $\Psi_{n_j}(s)$ of $P_{n_j}$ and $P_{n_j}^*$, respectively, as follows
$$\Phi_{n_j}(\theta)=\left(\phi_{n_j}(\theta), \overline{\phi} _{n_j}(\theta)\right),
 ~\Psi_{n_j}(s)=\mbox{col}\left(\psi_{n_j}^T(s), \overline{\psi }_{n_j}^T(s)\right),$$
such that $\langle\Psi_{n_j},\Phi_{n_j}\rangle_{n_j}=I_2,$
where
\begin{equation*}
\phi_{n_j}(\theta)=
\left(
\begin{array}{c}
\phi^{(1)}_{n_j}(\theta)\vspace{0.2cm}\\
 \phi^{(2)}_{n_j}(\theta)
\end{array}
\right)=\phi_{n_j}(0)e^{i\omega_{jc}\theta},~~\psi_{n_j}(s)=
\left(
\begin{array}{c}
\psi^{(1)}_{n_j}(s)\vspace{0.2cm}\\
 \psi^{(2)}_{n_j}(s)
\end{array}
\right)=\psi_{n_j}(0)e^{-i\omega_{jc}s},
\end{equation*}
and
\begin{equation*}
\phi_{n_j}(0)=\left(\begin{array}{c} 1 \vspace{0.2cm}\\
\frac{i\omega_j+(n_j/\ell)^2d_{11}-a_{11}}{a_{12}}
\end{array}\right),
~~
\psi_{n_j}(0)=\alpha_j\left(\begin{array}{c} 1 \vspace{0.2cm}\\
\frac{a_{12}}{i\omega_j+(n_j/\ell)^2d_{22}-a_{22}}
\end{array}\right),
\end{equation*}
with
$$\alpha_j=\frac{i\omega_j+(n_j/\ell)^2d_{22}-a_{22}}{i\omega_j+(n_j/\ell)^2d_{11}-a_{11}+i\omega_j+(n_j/\ell)^2d_{22}-a_{22}+\tau_ca_{12}d_{21}^cv_*(n_j/\ell)^2e^{-i\omega_{jc}}}.$$

Using the decomposition $C=P_{n_j}\oplus Q_{n_j}$,  the phase space $\mathscr{C}$ for \eqref{NFEQ}  can be decomposed as
$$ \mathscr{C}=\mathcal{P} \oplus \mathcal{Q}, ~\mathcal{P}=\mbox{Im} \pi, ~\mathcal{Q}=\mbox{Ker} \pi, $$
 where $\pi:  \mathscr{C}\rightarrow \mathcal{P}$ is the  projection operator defined by
$$
\begin{array}{lll}
  \pi (\phi) 
&=&
\Phi_{n_1}(\theta) \left \langle  \Psi_{n_1}(\theta), \left(\begin{array}{c}\left[\phi(\cdot), \beta_{n_1}^{(1)}\right]\vspace{0.2cm}\\
\left[\phi(\cdot),   \beta_{n_1}^{(2)}\right]  
\end{array}\right)  \right\rangle _{n_1}  \gamma_{n_1}(x)\vspace{0.2cm}\\
&&
+\Phi_{n_2}(\theta)   \left \langle \Psi_{n_2}(\theta), \left(\begin{array}{c}\left[\phi(\cdot), \beta_{n_2}^{(1)}\right]\vspace{0.2cm}\\
\left[\phi(\cdot), \beta_{n_2}^{(2)}\right]
\end{array}\right) \right\rangle _{n_2}   \gamma_{n_2}(x).
\end{array}
$$
In the following, for simplification of notations, we use $\phi_{n_j},  \overline{\phi}_{n_j},  \psi_{n_j}$ and  $\overline{\psi}_{n_j}$ for $\phi_{n_j}(\theta),  \overline{\phi}_{n_j}(\theta),  \psi_{n_j}(\theta)$ and $\overline{\psi}_{n_j}(\theta)$, respectively.
In addition, notice that for $n_1=n_2$, $\phi_{n_1}(\theta)$ looks like $\phi_{n_2}(\theta)$ but they are actually different since $\omega_{1c}\neq \omega_{2c}$.
Thus, in this case, one can not replace $\phi_{n_1}(\theta)$ by $\phi_{n_2}(\theta)$ or conversely.

Following \cite{Faria-00TAMS}  and \cite{Song-ZP-16CNSNS}, we
define $\mathscr{C}_0^1=\left\{ \phi\in \mathscr{C}: \dot{\phi}\in \mathscr{C}, \phi (0)\in \mbox{dom}(d\Delta)\right\}$ and let
$$z_x=\left(z_1(t)\gamma_{n_1}(x), z_2(t)\gamma_{n_1}(x), z_3(t)\gamma_{n_2}(x), z_4(t)\gamma_{n_2}(x) \right)^T
$$
and
$$
\Phi(\theta)=\left(\phi_{n_1}(\theta), \overline{\phi}_{n_1}(\theta),\phi_{n_2}(\theta), \overline{\phi} _{n_2}(\theta)\right).
$$
For  $\varphi (\theta)\in \mathscr{C}_0^1$,   we have the following decomposition
$$
\varphi(\theta)=\Phi(\theta) z_x+w,  ~~w=(w^{(1)}, w^{(2)})^T\in \mathscr{C}_0^1\cap \mbox{Ker}\pi:=\mathscr{Q}^1.
$$
For simplicity  of notations,   we write
$$
\left(\begin{array}{c}
\left[\widetilde{F}, \beta_{\nu}^{(1)}\right]\vspace{0.2cm}\\
\left[\widetilde{F}, \beta_{\nu}^{(2)}\right]
\end{array}\right)_{\nu=n_1}^{\nu=n_2} ~~\mathrm{for}~~\mathrm{col} \left(
   \left(\begin{array}{c}
\left[\widetilde{F}, \beta_{n_1}^{(1)}\right]\vspace{0.2cm}\\
\left[\widetilde{F}, \beta_{n_1}^{(2)}\right]
\end{array}\right),
\left(\begin{array}{c}
\left[\widetilde{F}, \beta_{n_2}^{(1)}\right]\vspace{0.2cm}\\
\left[\widetilde{F}, \beta_{n_2}^{(2)}\right]
\end{array}\right)
\right)
$$
and  let $z=(z_1(t) ~ z_2(t) ~z_3(t)  ~z_4(t) )^T.$ Following the notations in \cite{Faria-00TAMS}, we define
$$X_0(\theta)=\left\{
\begin{array}{ll}
0, & -1\leq \theta <0,\vspace{0.2cm}\\
1, & \theta=0.
\end{array}\right.
$$
and then 
\begin{equation}
\label{PIX0}
\begin{array}{lll}
&& \pi\left(X_0(\theta)\widetilde{F}_2\left(\Phi(\theta) z_x,0\right)\right)\vspace{0.2cm}\\
&=&
\Phi_{n_1}(\theta)\Psi_{n_1}(0)\left(\begin{array}{c}\left[\widetilde{F}_2\left(\Phi(\theta) z_x,0\right),\beta_{n_1}^{(1)}\right]\vspace{0.2cm}\\
\left[\widetilde{F}_2\left(\Phi(\theta) z_x,0\right),\beta_{n_1}^{(2)}\right]
\end{array}\right)\gamma_{n_1}(x)\vspace{0.2cm}\\
&&
+\Phi_{n_2}(\theta)\Psi_{n_2}(0)\left(\begin{array}{c}\left[\widetilde{F}_2\left(\Phi(\theta) z_x,0\right),\beta_{n_2}^{(1)}\right]\vspace{0.2cm}\\
\left[\widetilde{F}_2\left(\Phi(\theta) z_x,0\right),\beta_{n_2}^{(2)}\right]
\end{array}\right)\gamma_{n_2}(x).
\end{array}
\end{equation}.

Then system \eqref{NFEQ1} is decomposed as a system of abstract
ODEs on $\mathbb{R}^4\times \mbox{Ker}\pi$:
\begin{equation}
\label{AODE}
\begin{cases}
\dot{z}=Bz+\Psi(0)\left(\begin{array}{c}\left[\widetilde{F}\left(\Phi(\theta) z_x+w,\mu\right),\beta_\nu^{(1)}\right]\vspace{0.1cm}\\
\left[\widetilde{F}\left(\Phi(\theta) z_x+w,\mu\right),\beta_\nu^{(2)}\right]
\end{array}\right)^{\nu=n_2}_{\nu=n_1},\vspace{0.1cm}\\
\dot{w}=A_{\mathcal{Q}^1}w+(I-\pi)X_0(\theta)\widetilde{F}\left(\Phi(\theta) z_x+w,\mu\right),
\end{cases}
\end{equation}
where
$$\Psi(0)=\mbox{diag}\left\{\Psi_{n_1}(0),\Psi_{n_2}(0)\right\}~~B=\mbox{diag}\left\{i\omega_{1c}, -i\omega_{1c},i \omega_{2c},
-i \omega_{2c}\right\},$$
$A_{\mathcal{Q}^1}: \mathcal{Q}^1\rightarrow \mbox{Ker}\pi$ is defined by
$$A_{\mathscr{Q}^1} w=\dot{w}+ X_0(\theta) \left(L_0(w)+L_0^d(w)-\dot{w}(0) \right).$$

Consider the formal Taylor expansion
$$\widetilde{F}(\varphi, \mu)=\sum\limits_{j\geq 2}\frac{1}{j!}\widetilde{F}_j(\varphi, \mu),  ~~ F(\varphi, \mu)=\sum\limits_{j\geq 2}\frac{1}{j!} F_j(\varphi, \mu)$$
and
$$F^d (\varphi, \mu)=\frac{1}{2}F^d_2 (\varphi, \mu)+\frac{1}{3!}F^d_3 (\varphi, \mu)+\frac{1}{4!}F^d_4 (\varphi, \mu).$$
From \eqref{FD}, we have
\begin{equation}
\label{D2MU}
F^d_2 (\varphi, \mu)=F^{d(0,0)}_2 (\varphi)+\mu_1F^{d(1,0)}_2 (\varphi)+\mu_2F^{d(0,1)}_2 (\varphi),
\end{equation}

\begin{equation}
\label{D3MU}
F^d_3 (\varphi, \mu)=\mu_1F^{d(1,0)}_3(\varphi)+\mu_2F^{d(0,1)}_3(\varphi)+\mu_1\mu_2F^{d(1,1)}_3(\varphi),
\end{equation}
with
\begin{equation}
\label{FD23}
\left\{
\begin{array}{l}
F^{d(0,0)}_2 (\varphi)=-2d_{21}^c\tau_c\left(\begin{array}{c}
0\vspace{0.2cm}\\
\varphi^{(1)}_{x}(-1)\varphi^{(2)}_{x}(0)+\varphi^{(1)}_{xx}(-1)\varphi^{(2)}(0) \end{array}\right),\vspace{0.2cm}\\
%--------------------------------
F^{d(1,0)}_2 (\varphi)=2D_1\varphi_{xx}(0)+2D_2^c\varphi_{xx}(-1),\vspace{0.2cm}\\
%------------------------------------
F^{d(0,1)}_2 (\varphi)=\frac{2\tau_c}{d_{21}^c}D_2^c\varphi_{xx}(-1),
\vspace{0.2cm}\\
%------------------------------------
F^{d(1,0)}_3 (\varphi)=-6d_{21}^c\left(\begin{array}{c}
0\vspace{0.2cm}\\
\varphi^{(1)}_{x}(-1)\varphi^{(2)}_{x}(0)+\varphi^{(1)}_{xx}(-1)\varphi^{(2)}(0)\end{array}\right),   \vspace{0.2cm}\\
%------------------------------------
F^{d(0,1)}_3 (\varphi)=-6\tau_c\left(\begin{array}{c}
0\vspace{0.2cm}\\
\varphi^{(1)}_{x}(-1)\varphi^{(2)}_{x}(0)+\varphi^{(1)}_{xx}(-1)\varphi^{(2)}(0)\end{array}\right),   \vspace{0.2cm}\\
%------------------------------------
F^{d(1,1)}_3 (\varphi)=-6\left(\begin{array}{c}
0\vspace{0.2cm}\\
v_*\varphi^{(1)}_{xx}(-1)\end{array}\right),
%------------------------------------
\end{array}
\right.
\end{equation}
and
$$
F^d_4 (\varphi, \mu)
=-24\mu_1\mu_2\left(\begin{array}{c}
0\vspace{0.2cm}\\
\varphi^{(1)}_{x}(-1)\varphi^{(2)}_{x}(0)+\varphi^{(1)}_{xx}(-1)\varphi^{(2)}(0)\end{array}\right).$$

From \eqref{NLTERM}, we have
\begin{equation}
\label{2NDT}
\widetilde{F}_2( \varphi, \mu)=2\mu_1A\varphi(0)+
F_2(\varphi, \mu)+F_2^d(\varphi, \mu),
\end{equation}
and
\begin{equation}
\label{3NDT}
\widetilde{F}_3( \varphi, \mu)=
F_3(\varphi, \mu) +F_3^d(\varphi, \mu).
\end{equation}

Then \eqref{AODE} is written as
$$
\left\{
\begin{array}{l}
  \dot{z}=Bz+\sum\limits_{j\geq 2} \frac{1}{j!}f^1_j(z, w, \mu),\vspace{0.2cm}\\
 \dot{w}=A_{\mathscr{Q}^1}w+ \sum\limits_{j\geq 2} \frac{1}{j!}f^2_j(z, w, \mu),
\end{array}
\right.
$$
where
\begin{equation}
\label{FJ1}
 f^1_j(z, w, \mu)=\Psi(0)\left(\begin{array}{c}\left[\widetilde{F}_j\left(\Phi(\theta) z_x+w,\mu\right),\beta_\nu^{(1)}\right]\vspace{0.1cm}\\
\left[\widetilde{F}_j\left(\Phi(\theta) z_x+w,\mu\right),\beta_\nu^{(2)}\right]
\end{array}\right)^{\nu=n_2}_{\nu=n_1},
\end{equation}
\begin{equation}
\label{FJ2}
 f^2_j(z, w, \mu)=(I-\pi)X_0(\theta) \widetilde{F}_j\left(\Phi(\theta) z_x+w, \mu\right).
\end{equation}

In terms of the normal form theory  of partial functional differential equations \cite{Faria-00TAMS}, after a recursive transformation of variables of the form
\begin{equation}
\label{TOV}
(z, w)=(\widetilde{z}, \widetilde{w})+\frac{1}{j!}\left(U_j^1(\widetilde{z}, \mu), U_j^2(\widetilde{z}, \mu)(\theta)\right), j\geq 2,
\end{equation}
where $z, \widetilde{z}\in \mathbb{R}^4, w, \widetilde{w}\in \mathscr{Q}^1$ and $U_j^1:\mathbb{R}^6\rightarrow  \mathbb{R}^4,U_j^2:\mathbb{R}^6\rightarrow  \mathscr{Q}^1$ are homogeneous polynomials of degree $j$ in $\widetilde{z}$ and $\mu,$
 the flow on the local center manifold for \eqref{NFEQ1} is written as
\begin{equation}
\label{NF}
  \dot{z}=Bz+\sum\limits_{j\geq2}\frac{1}{j!}g_j^1(z,0,\mu),
\end{equation}
which is the normal form as in the usual sense for ODEs.

Following  \cite{Li-JW-12JMAA} and \cite{Jiang-Song-15AMC},   we have
$$
g_2^1(z, 0,\mu)=\mbox{Proj}_{\mbox{Ker}(M_2^1)}f_2^1(z,0,\mu),
$$
and
\begin{equation}
\label{G310}
g_3^1(z,0,\mu)=\mbox{Proj}_{\mbox{Ker}(M_3^1)}\widetilde{f}_3^1(z,0,\mu)=\mbox{Proj}_{S}\widetilde{f}_3^1(z,0,0)+O(|\mu|^2|z|),
\end{equation}
where  $\widetilde{f}_3^1(z,0,\mu)$ is  the terms of order $3$ in $(z, \mu)$ obtained  after performing the change of variables \eqref{TOV} of order $2$ and is determined by  \eqref{FTLD31},
\begin{equation*}
\mbox{Ker}\left(M_2^1\right)=\mbox{Span}\left\{\mu_iz_1e_1,\mu_iz_2e_2,\mu_iz_3e_3,\mu_iz_4e_4,i=1,2\right\},
 \end{equation*}
\begin{equation*}
\mbox{Ker}\left(M_3^1\right)=\mbox{Span}\left\{
\begin{array}{l}
\mu_1\mu_2z_1e_1,\mu_i^2z_1e_1,z_1^2z_2e_1,z_1z_3z_4e_1,\mu_1\mu_2z_2e_2,\mu_i^2z_2e_2, \vspace{0.3cm}\\
z_1z_2^2e_2,z_2z_3z_4e_2,\mu_1\mu_2z_3e_3,\mu_i^2z_3e_3,z_3^2z_4e_3,z_1z_2z_3e_3,\vspace{0.3cm}\\
\mu_1\mu_2z_4e_4,\mu_i^2z_4e_4,z_3z_4^2e_4,z_1z_2z_4e_4,~~i=1, 2
\end{array}\right\},
 \end{equation*}
and
$$
S=\mbox{Span}\left\{z_1^2z_2e_1,z_1z_3z_4e_1,z_1z_2^2e_2,z_2z_3z_4e_2,z_3^2z_4e_3,z_1z_2z_3e_3,z_3z_4^2e_4,z_1z_2z_4e_4\right\}.
$$

For convenience, in what follows we set
$$\mathcal{H}\left(\alpha z_1^{q_1}z_2^{q_2}z_3^{q_3}z_4^{q_4} \mu_1^{\iota_1}\mu_2^{\iota_2}\right)=\left(\begin{array}{c}
\alpha z_1^{q_1}z_2^{q_2}z_3^{q_3}z_4^{q_4}\mu_1^{\iota_1}\mu_2^{\iota_2}\vspace{0.2cm}\\
\overline{\alpha}z_1^{q_2}z_2^{q_1}z_3^{q_4}z_4^{q_3}\mu_1^{\iota_1}\mu_2^{\iota_2}
\end{array}\right), ~\alpha\in \mathbb{C}.$$

%-------------------------------------------------------------------------------
\vspace{0.2cm}
\begin{subsubsection}
{Calculation of $\bf{g_2^1(z,0,\mu)}$}
\end{subsubsection}
  %-----------------------------------------------------------------------------
 It follows from  \eqref{FJ1} that
 \begin{equation}
\label{G2F21}
 f^1_2(z, 0,  \mu)=\Psi(0)\left(\begin{array}{c}\left[\widetilde{F}_2\left(\Phi(\theta) z_x,\mu\right),\beta_\nu^{(1)}\right]\vspace{0.1cm}\\
\left[\widetilde{F}_2\left(\Phi(\theta) z_x,\mu\right),\beta_\nu^{(2)}\right]
\end{array}\right)^{\nu=n_2}_{\nu=n_1}.
\end{equation}

 From \eqref{2NDT}, we have
\begin{equation}
\label{TILF2}
\widetilde{F}_2(\Phi(\theta) z_x, \mu)=2\mu_1A\left(\Phi(0) z_x\right)+
F_2(\Phi(\theta) z_x, \mu)+F_2^d(\Phi(\theta) z_x, \mu).
\end{equation}

%Letting
%$$
%\begin{array}{lll}
%\varphi(\theta)&=&\left(\varphi^{(1)}(\theta), \varphi^{(2)}(\theta)\right)^T
%= \Phi(\theta) z_x\vspace{0.3cm}\\
%&=&\gamma_{n_1}(x)\left(z_1\phi_{n_1}(\theta)+z_2\overline{\phi}_{n_1}(\theta)\right) +\gamma_{n_2}(x)\left(z_3\phi_{n_2}(\theta)+z_4\overline{\phi}_{n_2}(\theta)\right)
%\end{array}
%$$
%and
%$$\xi_n(x)=\frac{\sin\left(\frac{nx}{\ell}\right)}{\|\cos\left(\frac{nx}{\ell}\right)\|_{2,2}}=\left\{
%\begin{array}{ll}
% 0, & ~\mathrm{for}~n=0,\vspace{0.2cm} \\
%\frac{\sqrt{2}}{\sqrt{\ell \pi}}\sin\left(\frac{nx}{\ell}\right), & ~\mathrm{for}~n\neq 0,
%\end{array}
%\right.$$
%we have
%\begin{equation}
%\label{PHIX}
%\begin{array}{l}
%\varphi_x(\theta) =-(n_1/\ell)\xi_{n_1}(x) \left(\Phi_{n_1}(\theta) \left(\begin{array}{c}z_1\\z_2\end{array}\right)\right) -(n_2/\ell)\xi_{n_2}(x)\left( \Phi_{n_2}(\theta) \left(\begin{array}{c}z_3\\z_4\end{array}\right)\right), \vspace{0.3cm}\\
%\varphi_{xx}(\theta) =-(n_1/\ell)^2\gamma_{n_1}(x)\left(\Phi_{n_1}(\theta) \left(\begin{array}{c}z_1\\z_2\end{array}\right)\right)  -(n_2/\ell)^2\gamma_{n_2}(x)\left( \Phi_{n_2}(\theta) \left(\begin{array}{c}z_3\\z_4\end{array}\right)\right).
%\end{array}
%\end{equation}

From \eqref{D2MU}- \eqref{FD23}, \eqref{G2F21} and \eqref{TILF2}, and noticing that
 $$\left[\gamma_{n_i}(x),  \gamma_{n_j}(x)\right] = \left\{\begin{array}{ll}
1, & i=j, \vspace{0.3cm}\\
0,  & i\not=j,
\end{array}
\right. $$
then for $n_1\not=n_2$, we have
\begin{equation}
\label{G211}
\left(\begin{array}{c}\left[2\mu_1A\left(\Phi(0) z_x\right),\beta_\nu^{(1)}\right]\vspace{0.3cm}\\
\left[2\mu_1A\left(\Phi(0) z_x\right),\beta_\nu^{(2)}\right]
\end{array}\right)= \left\{\begin{array}{ll}
2\mu_1A\left( \Phi_{n_1}(0) \left(\begin{array}{c}z_1\\z_2\end{array}\right) \right), & \nu=n_1, \vspace{0.3cm}\\
2\mu_1A\left( \Phi_{n_2}(0) \left(\begin{array}{c}z_3\\z_4\end{array}\right) \right),  & \nu=n_2,
\end{array}
\right.
\end{equation}
%----------------------------------------
\begin{equation}
\label{G212}
\begin{array}{lll}
&&\left(\begin{array}{c}\left[\mu_1F^{d(1,0)}_2 \left(\Phi(\theta) z_x\right),\beta_\nu^{(1)}\right]\vspace{0.3cm}\\
\left[\mu_1F^{d(1,0)}_2 \left(\Phi(\theta) z_x\right),\beta_\nu^{(2)}\right]
\end{array}\right)\vspace{0.3cm}\\
&=& \left\{\begin{array}{ll}
-2(n_1/\ell)^2\mu_1\left(D_1\left( \Phi_{n_1}(0) \left(\begin{array}{c}z_1\\z_2\end{array}\right) \right)+D_2^c\left( \Phi_{n_1}(-1) \left(\begin{array}{c}z_1\\z_2\end{array}\right) \right)\right),& \nu=n_1, \vspace{0.3cm}\\
-2(n_2/\ell)^2\mu_1\left(D_1\left( \Phi_{n_2}(0) \left(\begin{array}{c}z_3\\z_4\end{array}\right) \right)+D_2^c\left( \Phi_{n_2}(-1) \left(\begin{array}{c}z_3\\z_4\end{array}\right) \right)\right), & \nu=n_2,
\end{array}
\right.
\end{array}
\end{equation}
and
%----------------------------------------
\begin{equation}
\label{FG213}
\begin{array}{lll}
&&\left(\begin{array}{c}\left[\mu_2F^{d(0,1)}_2 \left(\Phi(\theta) z_x\right),\beta_\nu^{(1)}\right]\vspace{0.3cm}\\
\left[\mu_2F^{d(0,1)}_2 \left(\Phi(\theta) z_x\right),\beta_\nu^{(2)}\right]
\end{array}\right)\vspace{0.3cm}\\
&=& \left\{\begin{array}{ll}
-\frac{2(n_1/\ell)^2\tau_c}{d_{21}^c}\mu_2D_2^c\left( \Phi_{n_1}(-1) \left(\begin{array}{c}z_1\\z_2\end{array}\right) \right), & \nu=n_1, \vspace{0.3cm}\\
-\frac{2(n_2/\ell)^2\tau_c}{d_{21}^c}\mu_2D_2^c\left( \Phi_{n_2}(-1) \left(\begin{array}{c}z_3\\z_4\end{array}\right) \right),  & \nu=n_2.
\end{array}
\right.
\end{array}
\end{equation}

By \eqref{FPM}, it is easy to verify that  for all $\mu\in \mathbb{R}^2$,
\begin{equation}
\label{F2PHI}
F_2( \Phi(\theta) z_x,\mu)=F_2( \Phi(\theta) z_x, 0).
\end{equation}
   This,  together with  \eqref{G211}, \eqref{G212} and \eqref{FG213}, yields to
\begin{equation}
\label{G21}
\begin{array}{lll}
g^1_2(z,0,\mu)&=&\mbox{Proj}_{\tiny{\mbox{Ker}(M_2^1)}}f_2^1(z,0,\mu)\vspace{0.2cm}\\
&=&\left(\begin{array}{c}
\mathcal{H}\left(\left(B^{(1)}_1\mu_1+B^{(2)}_1\mu_2\right)z_1\right)\vspace{0.2cm}\\
\mathcal{H}\left(\left(B^{(1)}_3\mu_1+B^{(2)}_3\mu_2\right)z_3\right)
\end{array}
\right),
\end{array}
\end{equation}
 where
$$
\begin{array}{l}
B^{(1)}_1=2\psi_{n_1}^{T}(0)\left(A \phi_{n_1}(0) -(n_1/\ell)^2   \left(D_1\phi_{n_1}(0)+D_2^c\phi_{n_1}(-1)\right)\right)=2i\omega_1\psi_{n_1}^{T}(0)\phi_{n_1}(0), \vspace{0.2cm}\\

B^{(1)}_3=2\psi_{n_2}^{T}(0)\left(A \phi_{n_2}(0) -(n_2/\ell)^2   \left(D_1\phi_{n_2}(0)+D_2^c\phi_{n_2}(-1)\right)\right)=2i\omega_2\psi_{n_2}^{T}(0)\phi_{n_2}(0),\vspace{0.2cm}\\

B^{(2)}_1=-\frac{2(n_1/\ell)^2\tau_c}{d_{21}^c}\psi_{n_1}^T(0)\left(D_2^c \phi_{n_1} (-1)\right),
~~~~B^{(2)}_3=-\frac{2(n_2/\ell)^2\tau_c}{d_{21}^c}\psi_{n_2}^T(0)\left(D_2^c \phi_{n_2} (-1)\right).
\end{array}
$$

For $n_1=n_2$, noticing  the fact that
$\left[  \Phi(\theta) z_x, \gamma_{n_j}(x)\right]=\Phi(\theta) z, j=1,2,$ and using the similar calculations  as above, it is easy to obtain the same $g^1_2(z,0,\mu)$ as in   \eqref{G21}.
%-------------------------------------------------------------------------------
\vspace{0.2cm}
\begin{subsubsection}
{Calculation of $\bf{g_3^1(z,0,\mu)}$}
\end{subsubsection}
%-------------------------------------------------------------------------------

In this subsection, we  calculate the
third term $g^1_3(z,0,0)$ in terms of \eqref{G310}. Notice that
  $\frac{1}{3!}\widetilde{f}_3^1$ in \eqref{G310} is the term of order 3
obtained after the changes of variables in previous step.
Denote
\begin{equation}
\label{F211W}
 f^{(1,1)}_2(z, w, 0)=\Psi(0)\left(\begin{array}{c}\left[F_2\left(\Phi(\theta) z_x+w, 0\right),\beta_\nu^{(1)}\right]\vspace{0.2cm}\\
\left[F_2\left(\Phi(\theta) z_x+w, 0\right),\beta_\nu^{(2)}\right]
\end{array}\right)^{\nu=n_2}_{\nu=n_1},
\end{equation}
\begin{equation}
\label{F212W}
 f^{(1,2)}_2(z, w, 0)=\Psi(0)\left(\begin{array}{c}\left[F^d_2\left(\Phi(\theta) z_x+w, 0\right),\beta_\nu^{(1)}\right]\vspace{0.2cm}\\
\left[F^d_2\left(\Phi(\theta) z_x+w, 0\right),\beta_\nu^{(2)}\right]
\end{array}\right)^{\nu=n_2}_{\nu=n_1}.
\end{equation}

In addition, it follows from \eqref{G21}   that $g_2^1(z,0,0)=(0,0,0,0)^T$. Then  $\widetilde{f}_3^1(z,0,0)$ is determined  by
\begin{equation}
\label{FTLD31}
\begin{array}{lll}
&&\widetilde{f}_3^1(z,0,0)\vspace{0.3cm}\\
&=&f_3^1(z,0,0)+\frac{3}{2}\left[\left(D_zf_2^1(z,0,0)\right)U_2^1(z,0)+\left(D_wf_2^{(1,1)}(z,0,0)\right)U_2^2(z,0)(\theta)\right.\vspace{0.3cm}\\
&&\left.+\left(D_{w,w_x,w_{xx}}f^{(1,2)}_2(z,0,0) \right)U_2^{(2,d)}(z, 0)(\theta)\right],
\end{array}
 \end{equation}
 where $f_2^1(z,0,0)=f_2^{(1,1)}(z,0,0)+f_2^{(1,2)}(z,0,0)$, 
 $$D_{w,w_x,w_{xx}}f^{(1,2)}_2(z,0,0)=\left(D_wf^{(1,2)}_2(z,0,0), D_{w_x}f^{(1,2)}_2(z,0,0), D_{w_{xx}}f^{(1,2)}_2(z,0,0) \right),$$
  \begin{equation*}
U_2^1(z,0)=\left(M_2^1\right)^{-1}\mbox{Proj}_{Im\left(M_2^1\right)}f_2^1(z,0,0),~~U_2^2(z,0)(\theta)=\left(M_2^2\right)^{-1}f_2^2(z,0,0),
 \end{equation*}
 and
\begin{equation}
\label{U2D}
 U_2^{(2,d)}(z, 0)(\theta)=\mbox{col}\left(U_2^2(z, 0)(\theta), U_{2x}^2(z, 0)(\theta), U_{2xx}^2(z, 0)(\theta)\right).
 \end{equation}

 Next, we compute  $\mbox{Proj}_{S}\widetilde{f}_3^1(z,0,0)$ step by step according to \eqref{FTLD31}. The calculation is divided into the following four steps.

 %---------------------------------------------
\vspace{0.4cm}
\noindent{\bf
 {Step 1: The calculation of $\bf{\mbox{Proj}_{S}f_3^1(z,0,0)}$ }}\\
 %---------------------------------------------
From  \eqref{D3MU} and \eqref{3NDT},  we have
$\widetilde{F}_3(\Phi(\theta)z_x,0)= F_3(\Phi(\theta)
z_x,0)$, which can be written as
\begin{equation}
\label{F3TILDM0}
\widetilde{F}_3(\Phi(\theta) z_x,0)=F_3(\Phi(\theta)
z_x,0)=\sum_{q_1+q_2+q_3+q_4=3}A_{q_1q_2q_3q_4}\gamma_{n_1}^{q_1+q_2}(x)\gamma_{n_2}^{q_3+q_4}(x)z_1^{q_1}z_2^{q_2}z_3^{q_3}z_4^{q_4},
\end{equation}
where $q_1,q_2,q_3,q_4\in \mathbb{N}_0$, and
\begin{equation}
\label{AQ1}
A_{q_1q_2q_3q_4}=\left(A^{(1)}_{q_1q_2q_3q_4}, A^{(2)}_{q_1q_2q_3q_4}\right)^T.
\end{equation}

It follows from \eqref{FJ1} and \eqref{F3TILDM0}  that
$$
f_3^1(z, 0, 0)=\Psi (0)\left(
 \begin{array}{c}
\sum\limits_{q_1+q_2+q_3+q_4=3}A_{q_1q_2q_3q_4}\int_0^{\ell\pi}\gamma_{n_1}^{q_1+q_2+1}(x)\gamma^{q_3+q_4}_{n_2}(x)dxz_1^{q_1}z_2^{q_2}z_3^{q_3}z_4^{q_4}\vspace{0.3cm} \\
\sum\limits_{q_1+q_2+q_3+q_4=3}A_{q_1q_2q_3q_4}\int_0^{\ell\pi}\gamma_{n_1}^{q_1+q_2}(x)\gamma^{q_3+q_4+1}_{n_2}(x)dxz_1^{q_1}z_2^{q_2}z_3^{q_3}z_4^{q_4}
 \end{array}\right),
$$
which, together with
the fact that
$$ \int_0^{\ell\pi}\gamma_{n_1}^2(x)\gamma_{n_2}^2(x) dx=\left\{\begin{array}{cc}
\frac{3}{2\ell\pi}, &  n_2=n_1,\\
\frac{1}{\ell\pi}, & n_2 \not= n_1,
\end{array}
\right. $$
implies that
\begin{equation}\label{PF31}
\mbox{Proj}_{S}f_3^1(z,0,0)=\left(\begin{array}{c}
\mathcal{H} \left(C_{11}z_1^2z_2+C_{12}z_1z_3z_4\right)\vspace{0.2cm}\\
\mathcal{H} \left( C_{31}z_3^2z_4+C_{32}z_1z_2z_3 \right)
\end{array}\right),
\end{equation}
where
$$
\begin{array}{l}
C_{11}=\frac{3}{2\ell\pi}  \psi_{n_1}^{T}(0)A_{2100}, ~~C_{31}=\frac{3}{2\ell\pi} \psi_{n_2}^{T}(0)A_{0021},\vspace{0.2cm}\\
%-----------
C_{12}= \left\{\begin{array}{cc}
\frac{3}{2\ell\pi} \psi_{n_1}^{T}(0)A_{1011}, &  n_2=n_1,\\
\frac{1}{\ell\pi}\psi_{n_1}^{T}(0)A_{1011} , & n_2 \not= n_1,
\end{array}
\right.~~
 C_{32}= \left\{\begin{array}{cc}
\frac{3}{2\ell\pi}  \psi_{n_2}^{T}(0)A_{1110}, &  n_2=n_1,\\
\frac{1}{\ell\pi} \psi_{n_2}^{T}(0)A_{1110} , & n_2 \not= n_1.
\end{array}
\right.
\end{array}
$$

%----------------------------------------
\vspace{0.4cm}
\noindent{\bf Step 2: The calculation of $\bf{\mbox{Proj}_{S}\left(\left(D_zf_2^1\right)(z,0,0)U_2^1(z,0)\right)}$}
%------------------------------------------

From \eqref{2NDT}, we have
\begin{equation}
\label{TLDF2}
\widetilde{F}_2( \Phi(\theta) z_x,0)=F_2( \Phi(\theta) z_x,0)+F^{d(0,0)}_2\left (\Phi(\theta) z_x\right) 
\end{equation}

By \eqref{F2PHI}, we  write
\begin{equation}
\label{F2EXP}
\begin{array}{lll}
&&F_2(\Phi(\theta) z_x+w,\mu) =F_2(\Phi(\theta)
z_x+w,0)   \vspace{0.2cm}\\
&=&\sum\limits_{q_1+q_2+q_3+q_4=2}A_{q_1q_2q_3q_4}\gamma_{n_1}^{q_1+q_2}(x)\gamma_{n_2}^{q_3+q_4}(x)z_1^{q_1}z_2^{q_2}z_3^{q_3}z_4^{q_4}  \vspace{0.2cm}\\
&&+\mathcal{S}_2(\Phi(\theta)
z_x,w)+O\left(|w|^2\right),
\end{array}
\end{equation}
where $q_1,q_2,q_3,q_4\in \mathbb{N}_0$ and
 $ 
 \mathcal{S}_2(\Phi(\theta)
z_x,w)$  is the second cross terms of $\Phi z_x$ and $w$.
In addition,  by  \eqref{D2MU},   we write
\begin{equation}
\label{F2DEXP}
\begin{array}{lll}
&&F_2^d\left(\Phi(\theta) z_x,0\right) =F^{d(0,0)}_2\left (\Phi(\theta) z_x\right)\vspace{0.2cm}\\
&=&\sum_{q_1+q_2+q_3+q_4=2}A^{(d,1)}_{q_1q_2q_3q_4} \left((-n_1/\ell)\xi_{n_1}(x)\right)^{q_1+q_2}  \left((-n_2/\ell)\xi_{n_2}(x)\right)^{q_3+q_4}  z_1^{q_1}z_2^{q_2} z_3^{q_3}z_4^{q_4} \vspace{0.2cm}\\
&&-(n_1/\ell)^2A^{(d,2)}_{2000} \gamma_{n_1}^2(x)z_1^2-(n_1/\ell)^2A^{(d,2)}_{0200}  \gamma^2_{n_1}(x) z_2^2-(n_2/\ell)^2A^{(d,2)}_{0020} \gamma_{n_2}^2(x)z_3^2
\vspace{0.2cm}\\
&&-(n_2/\ell)^2A^{(d,2)}_{0002}  \gamma_{n_2}^2(x) z_4^2-(n_1/\ell)^2A^{(d,2)}_{1100} \gamma_{n_1}^2(x)z_1z_2-(n_2/\ell)^2A^{(d,2)}_{0011}  \gamma_{n_2}^2(x) z_3z_4
\vspace{0.2cm}\\
&&-\gamma_{n_1}(x)\gamma_{n_2}(x)\left( \left((n_1/\ell)^2A^{(d,2)}_{1010} +(n_2/\ell)^2A^{(d,3)}_{1010}     \right)z_1z_3 + \left((n_1/\ell)^2A^{(d,2)}_{1001} +(n_2/\ell)^2A^{(d,3)}_{1001}\right)z_1z_4\right.\vspace{0.2cm}\\
&&+  \left((n_1/\ell)^2A^{(d,2)}_{0110} +(n_2/\ell)^2A^{(d,3)}_{0110}\right) z_2z_3 + \left((n_1/\ell)^2A^{(d,2)}_{0101} +(n_2/\ell)^2A^{(d,3)}_{0101}\right) z_2z_4,
\end{array}
\end{equation}
where $\xi_{n_j}=\frac{\sqrt{2}}{\sqrt{\ell\pi}}\sin\left(\frac{n_jx}{\ell}\right)$  with $j=1,2$, and $A_{i_1i_2i_3i_4}^{(d,j)}$ with $j=1,2,3$ are given in Appendix A.

It is easy to verify that
$$
\int_0^{\ell\pi}\gamma_{n_1}^3(x)dx=\int_0^{\ell\pi}\gamma_{n_2}^3(x)dx=\int_0^{\ell\pi}\gamma_{n_1}(x)\gamma^2_{n_2}(x)dx=0,
$$
$$
\int_0^{\ell\pi}\xi_{n_j}^2(x)\gamma_{n_j}(x)dx=0, j=1,2,
\int_0^{\ell\pi}\xi_{n_2}^2(x)\gamma_{n_1}(x)dx=\int_0^{\ell\pi}\xi_{n_1}(x)\xi_{n_2}(x)\gamma_{n_2}(x)dx=0,
$$

$$
  \int_0^{\ell\pi}\gamma_{n_1}^2(x)\gamma_{n_2}(x)dx=\int_0^{\ell\pi}\xi_{n_1}(x)\xi_{n_2}(x)\gamma_{n_1}(x)dx=\left\{\begin{array}{cc}
\frac{1}{\sqrt{2\ell\pi}}, & n_2 = 2n_1,\\
0, &  n_2\not=2n_1,
\end{array}
\right.
$$
and
$$
\int_0^{\ell\pi}\xi_{n_1}^2(x)\gamma_{n_2}(x)dx=\left\{\begin{array}{cc}
-\frac{1}{\sqrt{2\ell\pi}}, & n_2 = 2n_1,\\
0, &  n_2\not=2n_1.
\end{array}
\right.
$$

 Then, by \eqref{TLDF2} and  a direct calculation, we have
$$
\begin{array}{lll}
&& f^1_2(z, 0, 0) = \Psi(0)\left(\begin{array}{c}\left[\widetilde{F}_2\left(\Phi(\theta) z_x, 0\right),\beta_\nu^{(1)}\right]\vspace{0.1cm}\\
\left[\widetilde{F}_2\left(\Phi(\theta) z_x, 0\right),\beta_\nu^{(2)}\right]
\end{array}\right)^{\nu=n_2}_{\nu=n_1}  \vspace{0.3cm} \\
 &=& \left\{
 \begin{array}{cc}
\frac{1}{\sqrt{2\ell\pi}}\Psi(0) \left(
 \begin{array}{c}
\widetilde{A}_{1010}z_1z_3+\widetilde{A}_{1001}z_1z_4+\widetilde{A}_{0110}z_2z_3+\widetilde{A}_{0101}z_2z_4\vspace{0.2cm} \\
\widetilde{A}_{2000}z_1^2+\widetilde{A}_{0200}z_2^2+\widetilde{A}_{1100}z_1z_2
 \end{array}\right), &  n_2=2n_1,\vspace{0.2cm} \\
  (0,  0,  0, 0)^T, &  n_2\not=2n_1,\vspace{0.2cm} \\
 \end{array}
 \right.
 \end{array}
$$
where
%----------------------------------------
\begin{equation}
\label{WLDAJ1}
\left\{\begin{array}{l}
\widetilde{A}_{j_1j_2j_3j_4}=A_{j_1j_2j_3j_4}+\frac{n_1n_2}{\ell^2}A_{j_1j_2j_3j_4}^{(d,1)}-\frac{n_1^2}{\ell^2}A_{j_1j_2j_3j_4}^{(d,2)}-\frac{n_2^2}{\ell^2}A_{j_1j_2j_3j_4}^{(d,3)},\vspace{0.2cm}\\
j_1, j_2, j_3, j_4=0, 1, \quad j_1+j_2=1, \quad  j_3+j_4=1,
\end{array}\right.
\end{equation}
%--------------------------------------
\begin{equation}
\label{WLDAJ2}
\left\{\begin{array}{l}
\widetilde{A}_{j_1j_2j_3j_4}=A_{j_1j_2j_3j_4}-\frac{n_1^2}{\ell^2}\left(A_{j_1j_2j_3j_4}^{(d,1)}+ A_{j_1j_2j_3j_4}^{(d,2)}\right),\vspace{0.2cm}\\
j_1, j_2=0, 1, 2, \quad j_1+j_2=2, \quad   j_3=j_4=0.
\end{array}\right.
\end{equation}
%---------------------------------------
Then, for $n_2\not=2n_1$,  $U_2^1(z,0)=(0, 0, 0, 0)^T$, and for $n_2=2n_1$,
\begin{equation*}
\begin{array}{lll}
&&U_2^1(z,0)\vspace{0.2cm}\\
&=&\left(M_2^1\right)^{-1}\mbox{Proj}_{\mbox{Im} M_2^1}f_2^1(z,0,0)\vspace{0.2cm}\\
&=&\frac{1}{i\omega_{1c}\sqrt{2\ell\pi}}\left(\begin{array}{c}
\psi_{n_1}^T(0)\left(\frac{\omega_{1c}}{\omega_{2c}}\widetilde{A}_{1010}z_1z_3-\frac{\omega_{1c}}{\omega_{2c}}\widetilde{A}_{1001}z_1z_4+\frac{\omega_{1c}}{\omega_{2c}-2\omega_{1c}}\widetilde{A}_{0110}z_2z_3-\frac{\omega_{1c}}{\omega_{2c}+2\omega_{1c}}\widetilde{A}_{0101}z_2z_4\right)\vspace{0.1cm}\\
\overline{\psi}_{n_1}^T(0)\left(\frac{\omega_{1c}}{\omega_{2c}+2\omega_{1c}}\widetilde{A}_{1010}z_1z_3-\frac{\omega_{1c}}{\omega_{2c}-2\omega_{1c}}\widetilde{A}_{1001}z_1z_4+\frac{\omega_{1c}}{\omega_{2c}}\widetilde{A}_{0110}z_2z_3-\frac{\omega_{1c}}{\omega_{2c}}\widetilde{A}_{0101}z_2z_4\right)\vspace{0.4cm}\\
\psi_{n_2}^T(0)\left(-\frac{\omega_{1c}}{\omega_{2c}-2\omega_{1c}}\widetilde{A}_{2000}z_1^2-\frac{\omega_{1c}}{\omega_{2c}+2\omega_{1c}}\widetilde{A}_{0200}z_2^2-\frac{\omega_{1c}}{\omega_{2c}}\widetilde{A}_{1100}z_1z_2\right)\vspace{0.1cm}\\
\overline{\psi}_{n_2}^T(0)\left(\frac{\omega_{1c}}{\omega_{2c}+2\omega_{1c}}\widetilde{A}_{2000}z_1^2+\frac{\omega_{1c}}{\omega_{2c}-2\omega_{1c}}\widetilde{A}_{0200}z_2^2+\frac{\omega_{1c}}{\omega_{2c}}\widetilde{A}_{1100}z_1z_2\right)
\end{array}\right).
\end{array}
\end{equation*}

Hence,
\begin{equation}
\label{PDZF21}
\mbox{Proj}_{S}\left[\left(D_zf_2^1\right)(z,0,0)U_2^1(z,0)\right]=\left(\begin{array}{c}
\mathcal{H} \left(D_{11}z_1^2z_2+D_{12}z_1z_3z_4\right)\vspace{0.1cm}\\
\mathcal{H} \left(D_{31}z_3^2z_4+D_{32}z_1z_2z_3\right)\vspace{0.1cm}\\
\end{array}\right),
\end{equation}
where  for $n_2\not= 2n_1$,
$
D_{11}=D_{12}=D_{31}=D_{32}=0 $, and for $n_2=2n_1$, $D_{31}=0$,
$$
\begin{array}{lll}
D_{11}&=&
\frac{1}{2i\ell \pi}\left(-\frac{1}{\omega_{2c}}\left(\psi_{n_1}^T(0)\widetilde{A}_{1010}\right)\left(\psi_{n_2}^T(0)\widetilde{A}_{1100}\right)
+\frac{1}{\omega_{2c}}\left(\psi_{n_1}^T(0)\widetilde{A}_{1001}\right)\left(\overline{\psi}_{n_2}^T(0)\widetilde{A}_{1100}\right)
\right.\vspace{0.2cm}\\
&&\left.-\frac{1}{\omega_{2c}-2\omega_{1c}}\left(\psi_{n_1}^T(0)\widetilde{A}_{0110}\right)\left(\psi_{n_2}^T(0)\widetilde{A}_{2000}\right)
+\frac{1}{\omega_{2c}+2\omega_{1c}}\left(\psi_{n_1}^T(0)\widetilde{A}_{0101}\right)\left(\overline{\psi}_{n_2}^T(0)\widetilde{A}_{2000}\right)\right),
\end{array}
$$
$$
\begin{array}{lll}
D_{12}&=&
\frac{1}{2i\ell \pi}\left(-\frac{1}{\omega_{2c}}\left(\psi_{n_1}^T(0)\widetilde{A}_{1010}\right)\left(\psi_{n_1}^T(0)\widetilde{A}_{1001}\right)
+\frac{1}{\omega_{2c}}\left(\psi_{n_1}^T(0)\widetilde{A}_{1001}\right)\left(\psi_{n_1}^T(0)\widetilde{A}_{1010}\right)
\right.\vspace{0.2cm}\\
&&\left.-\frac{1}{\omega_{2c}-2\omega_{1c}}\left(\psi_{n_1}^T(0)\widetilde{A}_{0110}\right)\left(\overline{\psi}_{n_1}^T(0)\widetilde{A}_{1001}\right)
+\frac{1}{\omega_{2c}+2\omega_{1c}}\left(\psi_{n_1}^T(0)\widetilde{A}_{0101}\right)\left(\overline{\psi}_{n_1}^T(0)\widetilde{A}_{1010}\right)\right),
\end{array}
$$
$$
\begin{array}{lll}
D_{32}&=&
\frac{1}{2i\ell \pi}\left(\frac{2}{\omega_{2c}-2\omega_{1c}}\left(\psi_{n_2}^T(0)\widetilde{A}_{2000}\right)\left(\psi_{n_1}^T(0)\widetilde{A}_{0110}\right)
+\frac{1}{\omega_{2c}}\left(\psi_{n_2}^T(0)\widetilde{A}_{1100}\right)\left(\psi_{n_1}^T(0)\widetilde{A}_{1010}\right)
\right.\vspace{0.2cm}\\
&&\left.+\frac{1}{\omega_{2c}}\left(\psi_{n_2}^T(0)\widetilde{A}_{1100}\right)\left(\overline{\psi}_{n_1}^T(0)\widetilde{A}_{0110}\right)
+\frac{2}{\omega_{2c}+2\omega_{1c}}\left(\psi_{n_2}^T(0)\widetilde{A}_{0200}\right)\left(\overline{\psi}_{n_1}^T(0)\widetilde{A}_{1010}\right)\right).
\end{array}
$$
%----------------------------------------------------------------
\vspace{0.5cm}
\noindent{\bf
 Step 3: The calculation of $\bf{\mbox{Proj}_{S}\left(\left(D_wf_2^{(1,1)}(z,0,0)\right)U_2^2(z,0)(\theta)\right)}$}
%----------------------------------------------------------------

 Let
\begin{equation}
\label{U22}
 U_2^2(z, 0)(\theta)\triangleq h(\theta, z)=\sum\limits_{n\in \mathbb{N}_0}  h_n(\theta,z) \gamma_n(x),
 \end{equation}
where
$$  h_n(\theta,z) =\sum\limits_{q_1+q_2+q_3+q_4=2} h_{n,q_1q_2q_3q_4}(\theta)z_1^{q_1}z_2^{q_2}z_3^{q_3}z_4^{q_4}$$
with
$$h_{n,q_1q_2q_3q_4}(\theta)=\left( h_{n,q_1q_2q_3q_4}^{(1)}(\theta)\quad h_{n,q_1q_2q_3q_4}^{(2)}(\theta)\right) ^T.$$

By \eqref{U22}, we have
\begin{equation}
\label{U2DX}
\left\{\begin{array}{l}
U_{2x}^2(z, 0)(\theta)=h_{x}(\theta, z)=- \sum\limits_{n\in \mathbb{N}_0} (n/\ell) h_n(\theta,z) \xi_n(x),\vspace{0.2cm}\\
U_{2xx}^2(z, 0)(\theta)=h_{xx}(\theta, z)=-\sum\limits_{n\in \mathbb{N}_0} (n/\ell)^2 h_n(\theta,z) \gamma_n(x).
\end{array}\right.
\end{equation}

Then, from \eqref{F211W} and \eqref{U22},  we obtain
 $$
\begin{array}{lll}
&& \left(D_wf^{(1,1)}_2(z,0,0)\right)U_2^2(z,0)\vspace{0.2cm}\\
&=&\Psi(0)
   \left(\begin{array}{c}
\left[\left.D_w F_2\left(\Phi(\theta) z_x+w, 0\right)\right|_{w=0}\left(\sum\limits_{n\in \mathbb{N}_0}  h_n(\theta,z) \gamma_n(x)\right), \beta_{\nu}^{(1)}\right]\vspace{0.2cm}\\
\left[\left.D_w F_2\left(\Phi(\theta) z_x+w, 0\right)\right|_{w=0}\left(\sum\limits_{n\in \mathbb{N}_0}  h_n(\theta,z) \gamma_n(x)\right), \beta_{\nu}^{(2)}\right]
\end{array}\right)_{\nu=n_1}^{\nu=n_2}.\\
\end{array}
 $$

By \eqref{F2EXP} and a  straightforward  computation, we obtain
$$\left.D_w F_2\left(\Phi(\theta) z_x+w, 0\right)\right|_{w=0}\left(\sum\limits_{n\in \mathbb{N}_0}  h_n(\theta,z) \gamma_n(x)\right)= \mathcal{S}_2\left(\Phi(\theta) z_x, \sum\limits_{n\in \mathbb{N}_0}  h_n(\theta,z) \gamma_n(x)\right) $$
and
$$
\begin{array}{lll}
&&\left(\begin{array}{c}
                        \left[\mathcal{S}_2\left(\Phi(\theta) z_x, \sum\limits_{n\in \mathbb{N}_0}  h_n(\theta,z) \gamma_n(x)\right), \beta_{\nu}^{(1)}\right] \vspace{0.2cm}\\
  \left[\mathcal{S}_2\left(\Phi(\theta)  z_x, \sum\limits_{n\in \mathbb{N}_0}  h_n(\theta,z) \gamma_n(x)\right), \beta_{\nu}^{(2)}\right]
                         \end{array}
\right)\vspace{0.2cm}\\
&=&\sum\limits_{n\in \mathbb{N}_0} b_{n_1,n,\nu}\left(\mathcal{S}_2\left(\phi_{n_1}(\theta) z_1,h_n(\theta, z)\right)+\mathcal{S}_2\left(\overline{\phi}_{n_1}(\theta) z_2,h_n(\theta, z)\right)\right)\vspace{0.2cm}\\
&&
+\sum\limits_{n\in \mathbb{N}_0} b_{n_2,n,\nu}\left(\mathcal{S}_2\left(\phi_{n_2}(\theta) z_3,h_n(\theta, z)\right)+\mathcal{S}_2\left(\overline{\phi}_{n_2}(\theta) z_4,h_n(\theta, z)\right)\right),
\end{array}
$$
where $n=0, 1, 2, \cdots, \nu=n_1, n_2,$
\begin{equation}
\label{BKJM}
b_{n_j,n,\nu}=\int_0^{\ell\pi}\gamma_{n_j}(x)\gamma_n(x)\gamma_{\nu}(x)dx=\left\{\begin{array}{ll}
\frac{1}{\sqrt{\ell\pi}}, & n=0, \nu= n_j,    \vspace{0.2cm}\\
\frac{1}{\sqrt{2\ell\pi}}, & n=2n_j,  \nu=n_j,\vspace{0.2cm}\\
\frac{1}{\sqrt{2\ell\pi}}, & n=n_1+n_2,  \nu=n_{j+(-1)^{j+1}},\vspace{0.2cm}\\
\frac{1}{\sqrt{2\ell\pi}}, & n=n_2-n_1,  \nu=n_{j+(-1)^{j+1}},  n_1<n_2,\vspace{0.2cm}\\
 0, & \mbox{otherwise}.
 \end{array}
\right.
\end{equation}

Hence,
$$
\begin{array}{lll}
&&\left(D_wf_2^{(1,1)}(z, 0,
0)\right)U^2_2(z,0)(\theta)\vspace{0.2cm}\\
&=&\Psi(0)\left(
\begin{array}{l}
\sum\limits_{n=0, 2n_1} b_{n_1,n,n_1}\left(\mathcal{S}_2\left(\phi_{n_1}(\theta) z_1,h_n(\theta, z)\right)+\mathcal{S}_2\left(\overline{\phi}_{n_1}(\theta) z_2,h_n(\theta, z)\right)\right)\vspace{0.2cm}\\
\quad\quad+\sum\limits_{n=n_1+n_2, n_2-n_1} b_{n_2,n,n_1}\left(\mathcal{S}_2\left(\phi_{n_2}(\theta) z_3,h_n(\theta, z)\right)+\mathcal{S}_2\left(\overline{\phi}_{n_2}(\theta) z_4,h_n(\theta, z)\right)\right)\vspace{0.4cm}\\
\sum\limits_{n=0, 2n_2} b_{n_2,n,n_2}\left(\mathcal{S}_2\left(\phi_{n_2}(\theta) z_3,h_n(\theta, z)\right)+\mathcal{S}_2\left(\overline{\phi}_{n_2}(\theta) z_4,h_n(\theta, z)\right)\right)\vspace{0.2cm}\\
\quad\quad+\sum\limits_{n=n_1+n_2, n_2-n_1} b_{n_1,n,n_2}\left(\mathcal{S}_2\left(\phi_{n_1}(\theta) z_1,h_n(\theta, z)\right)+\mathcal{S}_2\left(\overline{\phi}_{n_1}(\theta) z_2,h_n(\theta, z)\right)\right)
\end{array}
\right).
\end{array}
$$

Then, we have
\begin{equation}
\label{PDWF221}
\textmd{Proj}_{S}\left(D_wf_2^{(1,1)}(z, 0,
0)U^2_2(z,0)\right)=\left(\begin{array}{c}
\mathcal{H} \left(E_{11}z_1^2z_2+E_{12}z_1z_3z_4\right)\vspace{0.1cm}\\
\mathcal{H} \left(E_{31}z_3^2z_4+E_{32}z_1z_2z_3\right)\vspace{0.1cm}\\
\end{array}\right),
 \end{equation}
 %------------------------
\begin{equation}
\label{EIJ}
\left\{
\begin{array}{l}
\begin{array}{lll}
E_{11}&=&\frac{1}{\sqrt{\ell\pi}}\psi_{n_1}^T(0)\left(\mathcal{S}_2(\phi_{n_1}(\theta) ,h_{0,1100}(\theta))+\mathcal{S}_2(\overline{\phi}_{n_1}(\theta) ,h_{0,2000}(\theta) ) \right) \vspace{0.2cm}\\
&&+\frac{1}{\sqrt{2\ell\pi}}\psi_{n_1}^T(0)\left(\mathcal{S}_2(\phi_{n_1}(\theta) ,h_{2n_1,1100}(\theta))+\mathcal{S}_2(\overline{\phi}_{n_1}(\theta) ,h_{2n_1,2000}(\theta) ) \right),
\end{array}
 \vspace{0.3cm}\\
%------------------------
\begin{array}{lll}
E_{12}&=&\frac{1}{\sqrt{\ell\pi}}\psi_{n_1}^T(0)\mathcal{S}_2(\phi_{n_1}(\theta) ,h_{0,0011}(\theta))+\frac{1}{\sqrt{2\ell\pi}}\psi_{n_1}^T(0)\mathcal{S}_2(\phi_{n_1}(\theta) ,h_{2n_1,0011}(\theta))
\vspace{0.2cm}\\
&&+\frac{1}{\sqrt{2\ell\pi}} \psi_{n_1}^T(0)\left(\mathcal{S}_2(\phi_{n_2}(\theta) ,h_{n_1+n_2,1001}(\theta))+\mathcal{S}_2(\overline{\phi}_{n_2}(\theta) ,h_{n_1+n_2,1010}(\theta) ) \right),
\vspace{0.2cm}\\
&&+\delta_{n_1n_2}\psi_{n_1}^T(0)\left(\mathcal{S}_2(\phi_{n_2}(\theta),h_{n_2-n_1,1001}(\theta))+\mathcal{S}_2(\overline{\phi}_{n_2}(\theta) ,h_{n_2-n_1,1010}(\theta) ) \right)),
\end{array}
 \vspace{0.3cm}\\
%------------------------
\begin{array}{lll}
E_{31}&=& \frac{1}{\sqrt{\ell\pi}}\psi_{n_2}^T(0)\left(\mathcal{S}_2(\phi_{n_2}(\theta) ,h_{0,0011}(\theta))+\mathcal{S}_2(\overline{\phi}_{n_2}(\theta) ,h_{0,0020}(\theta) ) \right) \vspace{0.2cm}\\
&&+\frac{1}{\sqrt{2\ell\pi}}\psi_{n_2}^T(0)\left(\mathcal{S}_2(\phi_{n_2}(\theta) ,h_{2n_2,0011}(\theta))+\mathcal{S}_2(\overline{\phi}_{n_2}(\theta),h_{2n_2,0020}(\theta) ) \right)),
\end{array}
  \vspace{0.3cm}\\
%------------------------
\begin{array}{lll}
E_{32}&=&\frac{1}{\sqrt{\ell\pi}}\psi_{n_2}^T(0)\mathcal{S}_2(\phi_{n_2}(\theta),h_{0,1100}(\theta))+\frac{1}{\sqrt{2\ell\pi}}\psi_{n_2}^T(0)\mathcal{S}_2(\phi_{n_2}(\theta) ,h_{2n_2,1100}(\theta))\vspace{0.2cm}\\
&&+\frac{1}{\sqrt{2\ell\pi}} \psi_{n_2}^T(0)\left(\mathcal{S}_2(\phi_{n_1}(\theta) ,h_{n_1+n_2,0110}(\theta))+\mathcal{S}_2(\overline{\phi}_{n_1}(\theta) _{n_1+n_2,1010}(\theta) ) \right))\vspace{0.2cm}\\
&&+\delta_{n_1n_2}\psi_{n_2}^T(0)\left(\mathcal{S}_2(\phi_{n_1}(\theta) ,h_{n_2-n_1,0110}(\theta))+\mathcal{S}_2(\overline{\phi}_{n_1}(\theta) ,h_{n_2-n_1,1010}(\theta) ) \right)),
\end{array}
\end{array}
\right.
\end{equation}
where
$$
\delta_{n_1n_2}=\left\{\begin{array}{ll}
 \frac{1}{\sqrt{2\ell\pi}} ,&   n_1<n_2, \vspace{0.2cm}\\
 \frac{1}{\sqrt{\ell\pi}} , &    n_1=n_2.
\end{array}\right.
$$

%----------------------------------------------------------------
\vspace{0.4cm}
\noindent
{\bf Step 4: The calculation of $\bf{ \mbox{Proj}_{S}\left(\left(D_{w,w_x,w_{xx}}f^{(1,2)}_2(z,0,0)\right)U_2^{(2,d)}(z, 0)(\theta)\right)}$}
%----------------------------------------------------------------

 The calculation of $\mbox{Proj}_{S}\left(\left(D_{w,w_x,w_{xx}}f^{(1,2)}_2(z,0,0)\right)U_2^{(2,d)}(z, 0)(\theta)\right)$ is similar to that in Step 3 but is more tedious. We leave the calculation to Appendix B.

   \begin{equation}
\label{PDWF22}
\textmd{Proj}_{S}\left(\left(D_{w,w_x,w_{xx}}f^{(1,2)}_2(z,0,0)\right)U_2^{(2,d)}(z, 0)(\theta)\right)=\left(\begin{array}{c}
\mathcal{H} \left(E_{11}^dz_1^2z_2+E_{12}^dz_1z_3z_4\right)\vspace{0.1cm}\\
\mathcal{H} \left(E_{31}^dz_3^2z_4+E_{32}^dz_1z_2z_3\right)\vspace{0.1cm}\\
\end{array}\right),
 \end{equation}
where
\begin{equation}
\label{EIJD1}
\left\{
 \begin{array}{l}
\begin{array}{lll}
 E_{11}^d &=& -\frac{1}{\sqrt{\ell\pi}} \frac{n_1^2}{\ell ^2} \psi_{n_1}^T(0)\left(\mathcal{S}_2^{(d,1)}\left(\phi_{n_1}(\theta),h_{0,1100}(\theta)\right)+\mathcal{S}_2^{(d,1)}\left(\overline{\phi}_{n_1}(\theta), h_{0,2000}(\theta)\right)\right)
 \vspace{0.2cm}\\
&&+\frac{1}{\sqrt{2\ell\pi}}\psi_{n_1}^T{(0)} \sum\limits_{j=1,2,3}  b^{(1,j)}_{2n_1} \mathcal{S}_2^{(d,j)}\left(\phi_{n_1}(\theta), h_{2n_1,1100}(\theta)\right) \vspace{0.2cm}\\
&&+\frac{1}{\sqrt{2\ell\pi}}\psi_{n_1}^T{(0)} \sum\limits_{j=1,2,3}  b^{(1,j)}_{2n_1} \mathcal{S}_2^{(d,j)}\left(\overline{\phi}_{n_1}(\theta), h_{2n_1,2000}(\theta)\right),
\end{array}
%--------------------------------------------------------
\vspace{0.3cm}\\
\begin{array}{lll}
 E_{12}^d &=&-\frac{1}{\sqrt{\ell\pi}}\frac{n_1^2}{\ell ^2} \psi_{n_1}^T(0)\mathcal{S}_2^{(d,1)}\left(\phi_{n_1}(\theta), h_{0,0011}(\theta)\right)
 \vspace{0.2cm}\\
&&+\frac{1}{\sqrt{2\ell\pi}}\psi_{n_1}^T{(0)} \sum\limits_{j=1,2,3}  b^{(1,j)}_{2n_1}\mathcal{S}_2^{(d,j)}\left(\phi_{n_1}(\theta), h_{2n_1,0011}(\theta)\right) \vspace{0.2cm}\\
&&+\frac{1}{\sqrt{2\ell\pi}}\psi_{n_1}^T{(0)} \sum\limits_{j=1,2,3}  b^{(2,j)}_{n_2+n_1}  \mathcal{S}_2^{(d,j)}\left(\phi_{n_2}(\theta), h_{(n_2+n_1),1001}(\theta)\right) 
 \vspace{0.2cm}\\
&&+\frac{1}{\sqrt{2\ell\pi}}\psi_{n_1}^T{(0)} \sum\limits_{j=1,2,3}  b^{(2,j)}_{n_2+n_1} \mathcal{S}_2^{(d,j)}\left(\overline{\phi}_{n_2}(\theta), h_{(n_2+n_1),1010}(\theta)\right) \vspace{0.2cm}\\
&&+\delta_{n_1n_2}\psi_{n_1}^T{(0)} \sum\limits_{j=1,2,3}  b^{(2,j)}_{n_2-n_1} \mathcal{S}_2^{(d,j)}\left(\phi_{n_2}(\theta), h_{(n_2-n_1),1001}(\theta)\right)
\vspace{0.2cm}\\
&&+\delta_{n_1n_2}\psi_{n_1}^T{(0)} \sum\limits_{j=1,2,3}  b^{(2,j)}_{n_2-n_1} \mathcal{S}_2^{(d,j)}\left(\overline{\phi}_{n_2}(\theta), h_{(n_2-n_1),1010}(\theta)\right),
\end{array}
\end{array}
\right.
\end{equation}
and
\begin{equation}
\label{EIJD2}
\left\{
 \begin{array}{l}
\begin{array}{lll}
E_{31}^d&=& -\frac{1}{\sqrt{\ell\pi}} \frac{n_2^2}{\ell ^2} \psi_{n_2}^T(0)\left(\mathcal{S}_2^{(d,1)}\left(\phi_{n_2}(\theta), h_{0,0011}(\theta)\right)+\mathcal{S}_2^{(d,1)}\left(\overline{\phi}_{n_2}(\theta), h_{0,0020}(\theta)\right)\right)\vspace{0.2cm}\\
&&+\frac{1}{\sqrt{2\ell\pi}}\psi_{n_2}^T{(0)} \sum\limits_{j=1,2,3}  b^{(2,j)}_{2n_2} \mathcal{S}_2^{(d,j)}\left(\phi_{n_2}(\theta), h_{2n_2,0011}(\theta)\right)
 \vspace{0.2cm}\\
&&+\frac{1}{\sqrt{2\ell\pi}}\psi_{n_2}^T{(0)} \sum\limits_{j=1,2,3}  b^{(2,j)}_{2n_2} \mathcal{S}_2^{(d,j)}\left(\overline{\phi}_{n_2}(\theta), h_{2n_2,0020}(\theta)\right),
\end{array}
%--------------------------------------------------------
\vspace{0.3cm}\\
\begin{array}{lll}
E_{32}^d&=&-\frac{1}{\sqrt{\ell\pi}} \frac{n_2^2}{\ell ^2} \psi_{n_2}^T(0)\mathcal{S}_2^{(d,1)}\left(\phi_{n_2}(\theta), h_{0,1100}(\theta)\right)
 \vspace{0.2cm}\\
&&+\frac{1}{\sqrt{2\ell\pi}}\psi_{n_2}^T{(0)} \sum\limits_{j=1,2,3} b^{(2,j)}_{2n_2}\mathcal{S}_2^{(d,j)}\left(\phi_{n_2}(\theta), h_{2n_2,1100}(\theta)\right) \vspace{0.2cm}\\
&&+\frac{1}{\sqrt{2\ell\pi}}\psi_{n_2}^T{(0)} \sum\limits_{j=1,2,3}  b^{(1,j)}_{n_2+n_1} \mathcal{S}_2^{(d,j)}\left(\phi_{n_1}(\theta), h_{(n_2+n_1),0110}(\theta)\right)\vspace{0.2cm}\\
&&+\frac{1}{\sqrt{2\ell\pi}}\psi_{n_2}^T{(0)} \sum\limits_{j=1,2,3}  b^{(1,j)}_{n_2+n_1}\mathcal{S}_2^{(d,j)}\left(\overline{\phi}_{n_1}(\theta), h_{(n_2+n_1),1010}(\theta)\right)\vspace{0.2cm}\\
&&+\delta_{n_1n_2}\psi_{n_2}^T{(0)} \sum\limits_{j=1,2,3} b^{(1,j)}_{n_2-n_1} \mathcal{S}_2^{(d,j)}\left(\phi_{n_1}(\theta), h_{(n_2-n_1),0110}(\theta)\right)\vspace{0.2cm}\\
&&+\delta_{n_1n_2}\psi_{n_2}^T{(0)} \sum\limits_{j=1,2,3} b^{(1,j)}_{n_2-n_1}\mathcal{S}_2^{(d,j)}\left(\overline{\phi}_{n_1}(\theta), h_{(n_2-n_1),1010}(\theta)\right),
\end{array}
\end{array}
\right.
\end{equation}
with
$$
\begin{array}{l}
b^{(1,1)}_k=-\frac{n_1^2}{\ell ^2}, ~~b^{(1,3)}_k=-\frac{k^2}{\ell^2}, ~~k=2n_1,  n_2+n_1,  n_2-n_1,~~b^{(1,2)}_k=\left\{\begin{array}{ll}
\frac{n_1k}{\ell^2}, & k=2n_1,  n_2+n_1, \vspace{0.2cm}\\
-\frac{n_1k}{\ell^2},&  k=n_2-n_1,
\end{array}
\right.
%--------------------------------------------------------
\vspace{0.3cm}\\
b^{(2,1)}_k=-\frac{n_2^2}{\ell ^2}, ~~ b^{(2,2)}_k=\frac{n_2k}{\ell^2},~~b^{(2,3)}_k=-\frac{k^2}{\ell^2}, ~~k= 2n_2, n_2+n_1, n_2-n_1.
\end{array}
$$
Clearly, we still need to compute  $h_{n,1010}(\theta)$. We leave this tedious calculation to Appendix C.
%-------------------------------------------------------------------------------
\vspace{0.3cm}
\begin{subsubsection}
{The normal form of double Hopf bifurcation truncated to third terms}
\end{subsubsection}
 %-----------------------------------------------------------------------------
   Let
$$
\begin{array}{l}
B_{11}=C_{11}+\frac{3}{2}\left(D_{11}+E_{11}+E_{11}^d\right),
~B_{12}=C_{12}+\frac{3}{2}\left(D_{12}+E_{12}+E_{12}^d\right),
\vspace{0.3cm}\\
%---------------------------------------------
B_{31}=C_{31}+\frac{3}{2}\left(D_{31}+E_{31}+E_{31}^d\right),
~B_{32}=C_{32}+\frac{3}{2}\left(D_{32}+E_{32}+E_{32}^d\right).
\end{array}
$$

From \eqref{NF}, \eqref{G21}, \eqref{PF31}, \eqref{PDZF21},  \eqref{PDWF221} and \eqref{PDWF22}, we have the following normal form of double Hopf bifurcation truncated to third terms:
\begin{equation}
\label{NF-DHOPF}
\dot{z}=Bz+\frac{1}{2!}\left(\begin{array}{c}
\mathcal{H} \left(\left(B^{(1)}_1\mu_1+B^{(2)}_1\mu_2\right)z_1\right)\vspace{0.2cm}\\
\mathcal{H} \left(\left(B^{(1)}_3\mu_1+B^{(2)}_3\mu_2\right)z_3\right)
\end{array}\right)+\frac{1}{3!}\left(\begin{array}{c}
\mathcal{H} \left(B_{11}z_1^2z_2+B_{12}z_1z_3z_4\right)\vspace{0.2cm}\\
\mathcal{H} \left(B_{31}z_3^2z_4+B_{32}z_1z_2z_3\right)
\end{array}\right).
 \end{equation}
With the polar coordinates
 $z_1=r_1e^{-i\Theta_1}, z_2=r_1e^{i\Theta_1}, z_3=r_2e^{-i\Theta_2}, z_4=r_2e^{i\Theta_2}$,  we have  the  following
 amplitude equations for \eqref{NF-DHOPF}:
\begin{equation}
\label{NF-DHOPF1}
   \begin{cases}\frac{dr_1}{dt}=\left(\delta_1+p_{11}r_1^2+p_{12}r_2^2\right)r_1,\vspace{0.1cm}\\
    \frac{dr_2}{dt}=\left(\delta_2+p_{21}r_1^2+p_{22}r_2^2\right)r_2,\end{cases}
 \end{equation}
where
$$
\begin{array}{l}
 \delta_1=\frac{1}{2}\mbox{Re} \left(B^{(1)}_1\mu_1+B^{(2)}_1\mu_2\right),~~\delta_2=\frac{1}{2}\mbox{Re} \left(B^{(1)}_3\mu_1+B^{(2)}_3\mu_2\right),\vspace{0.2cm}\\
p_{11}=\frac{1}{3!}\mbox{Re} \left(B_{11}\right),~p_{12}=\frac{1}{3!}\mbox{Re}\left(B_{12}\right),~p_{21}=\frac{1}{3!}\mbox{Re} \left(B_{32}\right),
~p_{22}=\frac{1}{3!}\mbox{Re} \left(B_{31}\right).
 \end{array}
$$
It follows from \cite{KUZ-BOOK1998} that depending on whether $p_{11}$
and $p_{22}$ have the same or opposite signs, there are two
essential bifurcation cases: \emph{simple} case: $p_{11}p_{22}>0$,
and \emph{difficult} case: $p_{11}p_{22}<0$.

For the \emph{simple} case or some subcases of the \emph{difficult} case, it is
sufficient to consider the normal form  truncated to three-order terms.  However, for some subcases of the \emph{difficult} case,  we have to calculate the normal form up to fifth-order terms to determine the
dynamics near the bifurcation point.

%------------------------------------------------------
%%------------------------------------------------------
% \vspace{0.2cm}
%\begin{subsection}
%   {Normal form  of the Hopf bifurcation}
% \end{subsection}
%%------------------------------------------------------
%
%For the Hopf bifurcation, we assume that  at $\tau=\tau_c$, Eq.\eqref{CE} has a pair of purely imaginary  roots $\pm i\omega_1$  for $n=n_1$  and all other eigenvalues  have negative real parts.  Take the perturbation of $\tau_c$ by setting $\tau=\tau_c+\mu$. 
%-------------------------------------------------------------------------------------------
\begin{section}
   {Examples}
 \end{section}
 %-------------------------------------------------------------------------------------------
 In this section, taking
$$
f(u,v)=u\left(1-\frac{u}{a}\right)-\frac{buv}{1+u},~~g(u,v)=\frac{buv}{1+u}-c v,
$$
then  \eqref{PPMODEL}  becomes the following  predator-prey model with Holling type II functional response:
\begin{equation}
\label{Exa-PP}
\left\{  \begin{array}{ll}
  \frac{\partial u(x,t)}{\partial t}=d_{11}u_{xx}(x,t)+u\left(1-\frac{u}{a}\right)-\frac{buv}{1+u}, & 0<x<\ell\pi, t>0, \vspace{0.2cm}\\
\frac{\partial v(x,t)}{\partial t}=-d_{21}(v(x,t)u_{x}(x,t-\tau))_x+d_{22}v_{xx}(x,t)-cv+\frac{buv}{1+u}, & 0<x<\ell\pi, t>0,\vspace{0.2cm}\\
u_x(0,t)=u_x(\ell\pi,t)=v_x(0,t)=v_x(\ell\pi,t)=0, & t\geq 0.
  \end{array}
  \right.
 \end{equation}

System \eqref{Exa-PP} has the  positive constant steady state $E_*(\gamma, v_\gamma)$, where
$$\gamma=\frac{c}{b-c},~~v_\gamma=\frac{(a-\gamma)(1+\gamma)}{ab},$$
provided that $b>\frac{c(1+a)}{a}$ (or equivalently, $0<\gamma<a$) holds.  For $d_{21}=0$,   $E_*(\gamma, v_\gamma)$ is asymptotically  stable for   $d_{11}\geq 0$ and $d_{22}\geq 0$ provided that $\frac{a-1}{2}< \gamma< a$. For $d_{21}>0$,  the stability and Hopf bifurcation for system \eqref{Exa-PP}  has been detailedly investigated in \cite{Song-SW-SAPM2021, Song-PZ-JDE2021}.  In what follows, we are interested in the double Hopf bifurcation induced by the spatial memory diffusion coefficient $d_{21}$ and the delay $\tau$. For this purpose, we need to investigate the critical values of $d_{21}$ and $\tau$ at which  \eqref{CE}  with \eqref{AIJ} has two pairs of purely imaginary roots. In the following, we use the same notations as in \cite{Song-SW-SAPM2021}  and  simply introduce some results from \cite{Song-SW-SAPM2021, Song-PZ-JDE2021} for the analysis.

 For  $E_*(\gamma, v_\gamma)$, we have
\begin{equation}
\label{AIJ}
\begin{array}{c}
a_{11}=\frac{\gamma(a-1-2\gamma)}{a(1+\gamma)}\left\{\begin{array}{ll}
\leq 0, & \frac{a-1}{2}\leq \gamma< a,\vspace{0.2cm}\\
>0, & 0<\gamma<\frac{a-1}{2},
\end{array}\right.
 \vspace{0.2cm}\\
 a_{12}=-c<0, ~~a_{21}=\frac{a-\gamma}{a(1+\gamma)}>0,   ~~a_{22}=0.
\end{array}
\end{equation}

%
%
%and
%$$
%a_{11}^2+a_{22}^2+2a_{12}a_{21}\left\{\begin{array}{ll}
%\leq 0, &  c\geq c_*,\vspace{0.2cm}\\
%>0, &  c<c_*,
%\end{array}\right.
%$$
%where $c_*=\frac{\gamma^2(a-1-2\gamma)^2}{2a(1+\gamma)(a-\gamma)}$. In what follows, we consider system \eqref{PPMODEL} with \eqref{FG} and $\ell=2$.
%

%----------------------------------------------------------------------------------------------------------
  
%Let  $\lambda=i \omega~(\omega>0)$ be a root of \eqref{CE}.  Substituting it into \eqref{CE} and separating the real and imaginary parts, we have
%
%\begin{equation}
%\label{RIPARTS}
%  \begin{cases}
%J_n-\omega^2= (n/\ell)^2d_{21}v_*a_{12}\cos\left(\omega\tau\right),\vspace{0.1cm}\\
%T_n\omega=(n/\ell)^2d_{21}v_*a_{12}\sin\left(\omega\tau\right),
%  \end{cases}
% \end{equation}
%which yields
Let
\begin{equation}
\label{OMGE}
\omega^4+P_n\omega^2+Q_n=0,
\end{equation}
where
\begin{equation}
\label{PN}
 P_n=\left(d_{11}^2+d_{22}^2\right)(n/\ell)^4-2\left(d_{11}a_{11}+d_{22}a_{22}\right)(n/\ell)^2
+a_{11}^2+a_{22}^2+2a_{12}a_{21},
 \end{equation}
 and
 \begin{equation}
\label{QN}
 Q_n=\widetilde{J}_n(0)\left(J_n+d_{21}v_*a_{12}(n/\ell)^2\right).
\end{equation}
It follows from \cite{Song-SW-SAPM2021}  that
if Eq.\eqref{OMGE} has a positive roots $\omega_n^+$ (or $\omega_n^-$), then 
Eq.\eqref{CE}  has a pair of purely imaginary roots  $\pm\omega_n^+i$ (or $\pm\omega_n^-i$) at $\tau=\tau_{n, j}^+$  (or $\tau=\tau_{n, j}^-$), where
\begin{equation}
\label{OMGN}
\omega_n^\pm=\sqrt{\frac{-P_n\pm \sqrt{\Delta_n}}{2}}.
\end{equation}
and
\begin{equation}
\label{TAUNJ}
\tau_{n, j}^\pm=\frac{1}{\omega_n^\pm}\left\{\arccos \left\{\frac{J_n-\left(\omega_n^\pm\right)^2}{d_{21}v_*a_{12}(n/\ell)^2}\right\}+2j\pi\right\}, ~~j\in\mathbb{N}_0,~n\in\mathbb{N}.
\end{equation}

The number of the positive root of  Eq.\eqref{OMGE} depends on the signs of $P_n, Q_n$ and 
$$\Delta_n=P_n^2-4Q_n=T_n^4-4T_n^2J_n+4d_{21}^2v_*^2a_{12}^2(n/\ell)^4.$$

Define
\begin{equation}
\label{d21n}
d_{21}^{(n)}  =\frac{1}{v_*|a_{12}|}\left( d_{11}d_{22} (n/\ell)^2 +\frac{Det(A)}{(n/\ell)^2}-\left( d_{11}a_{22}+d_{22}a_{11} \right)\right),
\end{equation}
%---------------------------------------------------------------------------------------------------------------------------
 \begin{equation}
\label{d21starN}
d _{21}^{*(n)}=\frac{\sqrt{4T_n^2J_n-T_n^4 }}{2v_*|a_{12}|(n/\ell)^2}.
\end{equation}
%---------------------------------------------------------------------------------------------------------------------------
Then, for fixed $n$, it follows from \cite{Song-SW-SAPM2021}  that
 $Q_n>0$ if and only if $0<d_{21}<d_{21}^{(n)}$, 
and   $ \Delta_n>0$ for\ $4J_n\leq T_n^2$ if and only if $d_{21}>d_{21}^{*(n)} $.
%-------------------------------------------------
\begin{subsection}
{Dynamics near the double Hopf bifurcation point with the same spatial profile}
\end{subsection}
%-------------------------------------------------

 Taking  the parameters as follows
 \begin{equation}
 \label{PMSFORcase1}
 a=1, b=9, c=3, d_{11}=0.6, d_{22}=0.8,~\ell=2,
 \end{equation}
we have $(u_*, v_*)=(1/2,  1/12)$ and  
$$
a_{11} =-\frac{1}{3}, a_{12} =-3, a_{21} =\frac{1}{3},  a_{22}=0.
$$
From \eqref{PN}, we have   
\begin{equation}
\label{NSPN}
P_n=\frac{1}{16}n^4 +\frac{1}{10}n^2 - \frac{17}{9}
 \left\{\begin{array}{ll}
 <0, & n=1,2, \vspace{0.2cm}\\
 >0, & n=3, 4, \cdots.
 \end{array}
 \right.
 \end{equation}
  By \eqref{d21n}, we have 
$$d_{21}^{(n)}=\frac{12n^2}{25}+\frac{16}{n^2}+\frac{16}{15},$$  
 from which it is easy to verify that
$$
d_{21}^{(2)}=\frac{524}{75} < d_{21}^{(3)}=\frac{1612}{225} <d_{21}^{(4)}=\frac{731}{75}<d_{21}^{(1)}=\frac{1316}{75}<d_{21}^{(5)}=\frac{1028}{75}<d_{21}^{(6)}=\frac{4228}{225},
$$
and  
\begin{equation}
\label{Md21star2}
\min\limits_{n\in \mathbb{N}}\left\{d_{21}^{(n)} \right\}=d_{21}^{(2)}=\frac{524}{75} \doteq 6.9867, ~ d_{21}^{(n)}<d_{21}^{(n+1)}  ~\mbox{for any}~  n\geq 5.
\end{equation}

When $0\leq d_{21} < d_{21}^{*(2)}$,   the stability of the positive constant steady state $(u_*, v_*)$  is independent of the delay and we  have the following stability result. 
%--------------------------------------------------------------------------------------
 \begin{proposition}
 \label{Pro-1}
For system \eqref{Exa-PP}  with the parameters $a=1, b=9, c=3, d_{11}=0.6, d_{22}=0.8,~\ell=2$,
when $0\leq d_{21} < d_{21}^{*(2)} =\frac{927}{134}$,  the positive constant steady state $(u_*, v_*)=(1/2,  1/12)$     is locally  asymptotically stable   for any $\tau\geq 0$;
  \end{proposition}
%------------------------------------------------------------------------------------
\begin{proof} 
From \eqref{QN}  and \eqref{Md21star2},  it follows that when 
$d_{21}<d_{21}^{(2)}$,  $Q_n>0$ for any $n\in \mathbb{N}$. This, together with \eqref{NSPN}, implies that 
when $d_{21}<d_{21}^{(2)}$, Eq.\eqref{OMGE} has no positive root for  $n=3, 4, \cdots$.

From \eqref{d21starN}, we have 
$$d_{21}^{*(2)}=\frac{927}{134} \doteq 6.9179<  d_{21}^{*(1)}=\frac{6851}{633}. 
$$
Notice that  
$$d_{21}^{*(2)} =\frac{927}{134}  \doteq 6.9179 <d_{21}^{(2)}=\frac{524}{75}\doteq 6.9867,$$
which implies that 
 for 
 $0< d_{21}<d_{21}^{*(2)}  \doteq 6.9179$, $\Delta_1<0$ and $\Delta_2<0$.  Thus, for $d_{21}<d_{21}^{*(2)}$, Eq.\eqref{OMGE} has no positive root for  $n=1, 2$. 
 
Combining the above discussion, we can conclude that for  $0< d_{21}<d_{21}^{*(2)}$, Eq.\eqref{OMGE} has no positive root for any $n\in\mathbb{N}$  and then  the positive constant steady state $(u_*, v_*)=(1/2,  1/12)$ is asymptotically stable for any $\tau\geq 0$.  
\end{proof}

 When   $d_{21} > d_{21}^{*(2)} \doteq 6.9179$,  the stability of the positive constant steady state $(u_*, v_*)$  is related to the delay.   Fig.\ref{Fig1}(a) illustrates the stability region and Hopf bifurcation curves in the $d_{21}-\tau$ plane for $6.9\leq d_{21}\leq 7 $ and $0\leq\tau\leq 20$.  
 
 When $d_{21}^{*(2)} \doteq 6.9179<d_{21} < d_{21}^{(2)}\doteq 6.9867 $,  Eq.\eqref{OMGE} has two positive roots for  $n=2$ and no positive roots for $n\in\mathbb{N}$ and $n\neq 2$. Thus, the characteristic equation \eqref{CE} has two sequences of purely imaginary roots $\pm i \omega_2^{\pm}$ at $\tau=\tau_{2,j}^\pm$.  System \eqref{Exa-PP}  undergoes Hopf bifurcations at $\tau=\tau_{2, j}^\pm$, as shown in  Fig.\ref{Fig1}(b). Hopf bifurcation curves $ \tau=\tau_{2,1}^+$  and $\tau=\tau_{2,0}^-$ intersect at the point $P_1(6.9618, 13.1290)$, which is the double Hopf bifurcation point with $\omega_2^-\doteq 0.2222$ and $\omega_2^+\doteq 0.6629$.  This double Hopf bifurcation arises from  the interaction of two Hopf bifurcations with the same mode-2. 
 %---------------------
 \begin{figure}[htp]
\centering
{\includegraphics[scale=0.2]{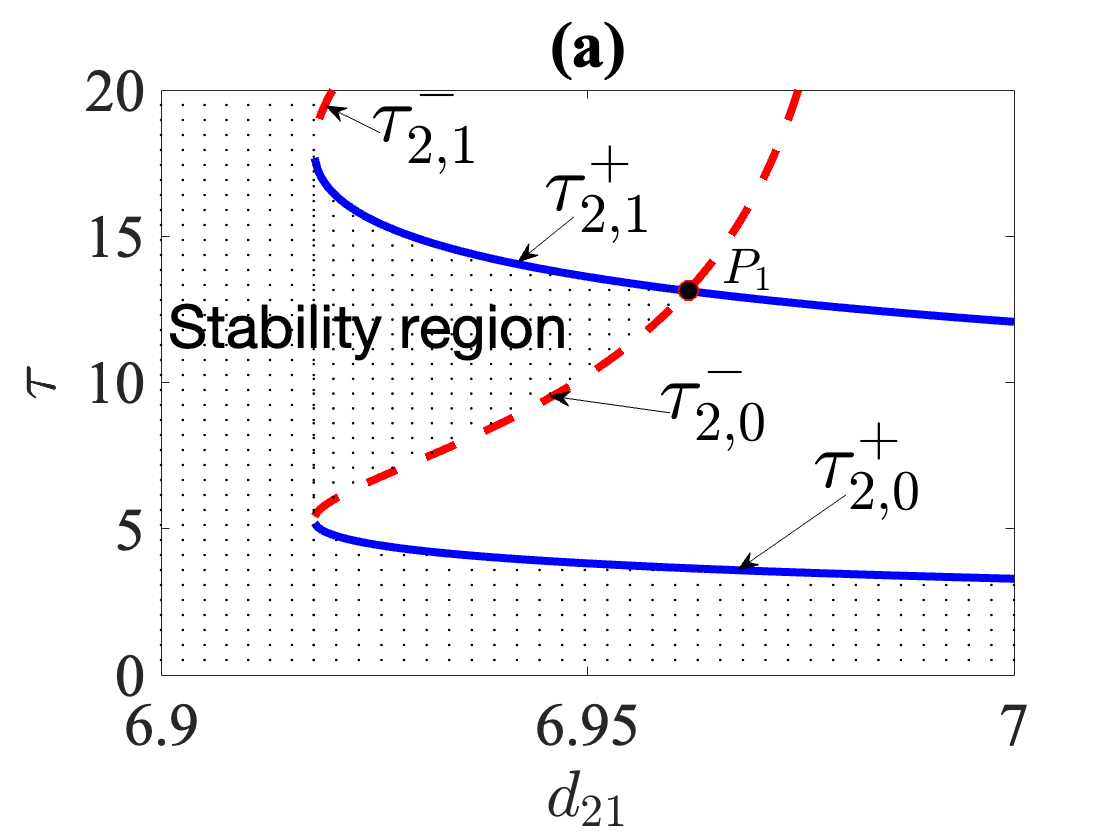}}
{\includegraphics[scale=0.2]{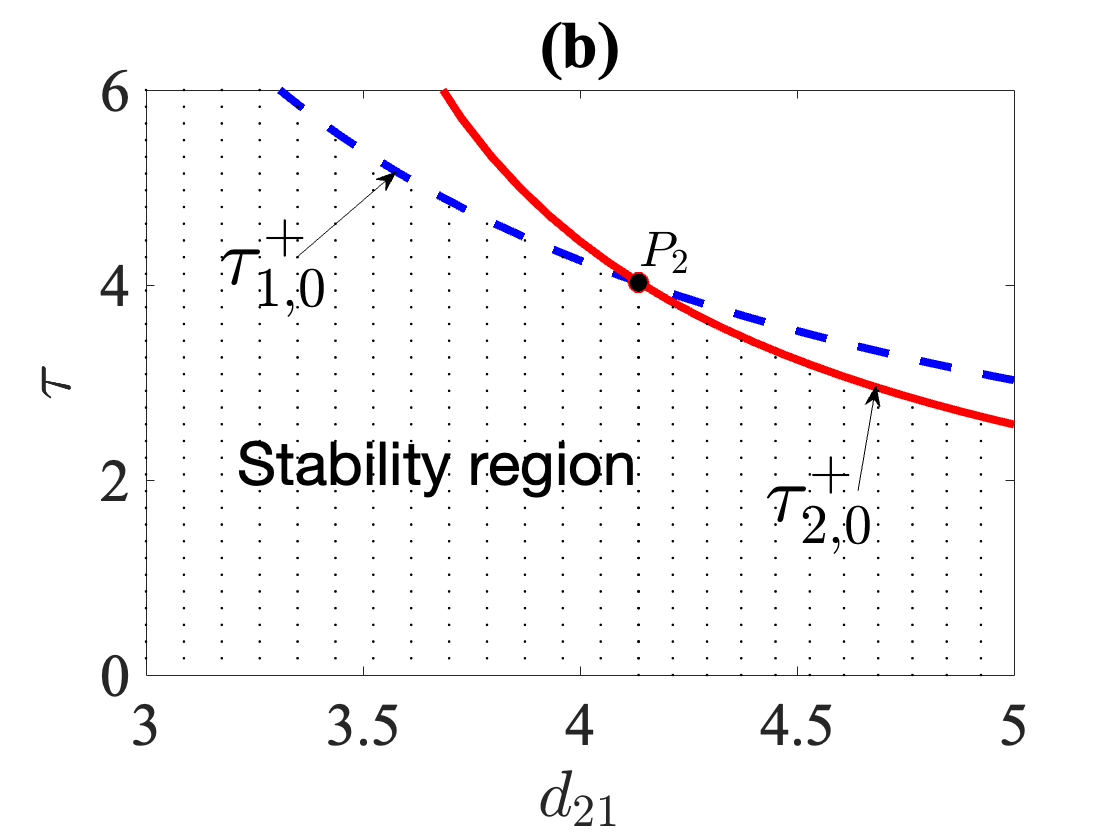}}
 \caption{ Stability and bifurcation curves in $d_{21}$-$\tau$ plane.  (a)  for the  parameters in \eqref{PMSFORcase1};   (b)  for the  parameters in \eqref{PMSFORcase2}.}
 \label{Fig1}
\end{figure}
%----------------------- 
 For this double Hopf bifurcation point $P_1(6.9618, 13.1290)$,  it follows from the normal form theory derived in Section \ref{Sec-NFDhopf}  with $n_1=n_2=2, d_{21}^c=6.9618, \tau_c=13.1290$,   $\omega_1=\omega_2^-$ and $\omega_2=\omega_2^+$.   The normal form truncated to the third order terms is
\begin{equation}\label{NUMNF1}
\begin{cases}
\dot{r}_1=r_1\left(-4.1681\times 10^{-4}  \mu_1+0.1332 \mu_2-0.1428 r_1^2+6.003 r_2^2\right),\vspace{0.1cm}\\
\dot{r}_2=r_2\left(0.0036\mu_1+0.1311\mu_2-5.4981 r_1^2-2.2507r_2^2\right).
\end{cases}
\end{equation}
 Since $r_1, r_2\geq 0$, system
\eqref{NUMNF1} has a zero equilibrium  $E_0(0,0)$ for any
$\mu_1,\mu_2\in \mathbb{R}$,  two boundary equilibria:
$$E_1\left(\sqrt{0.9331\mu_2 - 0.0029\mu_1},~0\right),~~\mu_1<319.6009\mu_2,$$
$$E_2\left(0,~\sqrt{0.0016\mu_1+0.0583\mu_2}\right),~\mu_1>-36.4988\mu_2,$$
and one interior positive equilibrium:
$$E_3\left(\sqrt{6.1905\times 10^{-4}\mu_1 + 0.0326\mu_2},
~\sqrt{0.8416\times 10^{-4} \mu_1 - 0.0214\mu_2}\right)$$
for
$$  \mu_1>\max\left\{-52.6919\mu_2,   254.4824\mu_2\right \}.$$

By analyzing the stability of these equilibria and noticing $\mu_1=\tau-\tau_c, \mu_2=d_{21}-d_{21}^c$, it is easy to obtain the phase portrait and their dynamical topologies as shown in Fig.\ref{Fig2}(a), where the curves $H_1, H_2, L_1, L_2$ are defined by the following:
\begin{equation}
\label{case1-HLJ}
\begin{array}{ll}
H_1:   ~\tau-\tau_c=319.6\left(d_{21}-d_{21}^c\right), &
H_2: ~ \tau-\tau_c=-36.4988\left(d_{21}-d_{21}^c\right),\vspace{0.2cm}\\
L_1: ~\tau-\tau_c=254.4824\left(d_{21}-d_{21}^c\right), d_{21}>d_{21}^c, &
L_2: ~\tau-\tau_c=-52.6919\left(d_{21}-d_{21}^c\right), d_{21}<d_{21}^c.
\end{array}
\end{equation}
 
 The straight lines $H_j, L_j, j=1,2$, divide   the vicinity of the double Hopf bifurcation point $P_1$ into six regions.  For each region, the corresponding phase portrait is plotted in the right side of Fig.\ref{Fig2}.

  %----------------------------------------------
\begin{figure}[htp]
\centering
{\includegraphics[scale=0.2]{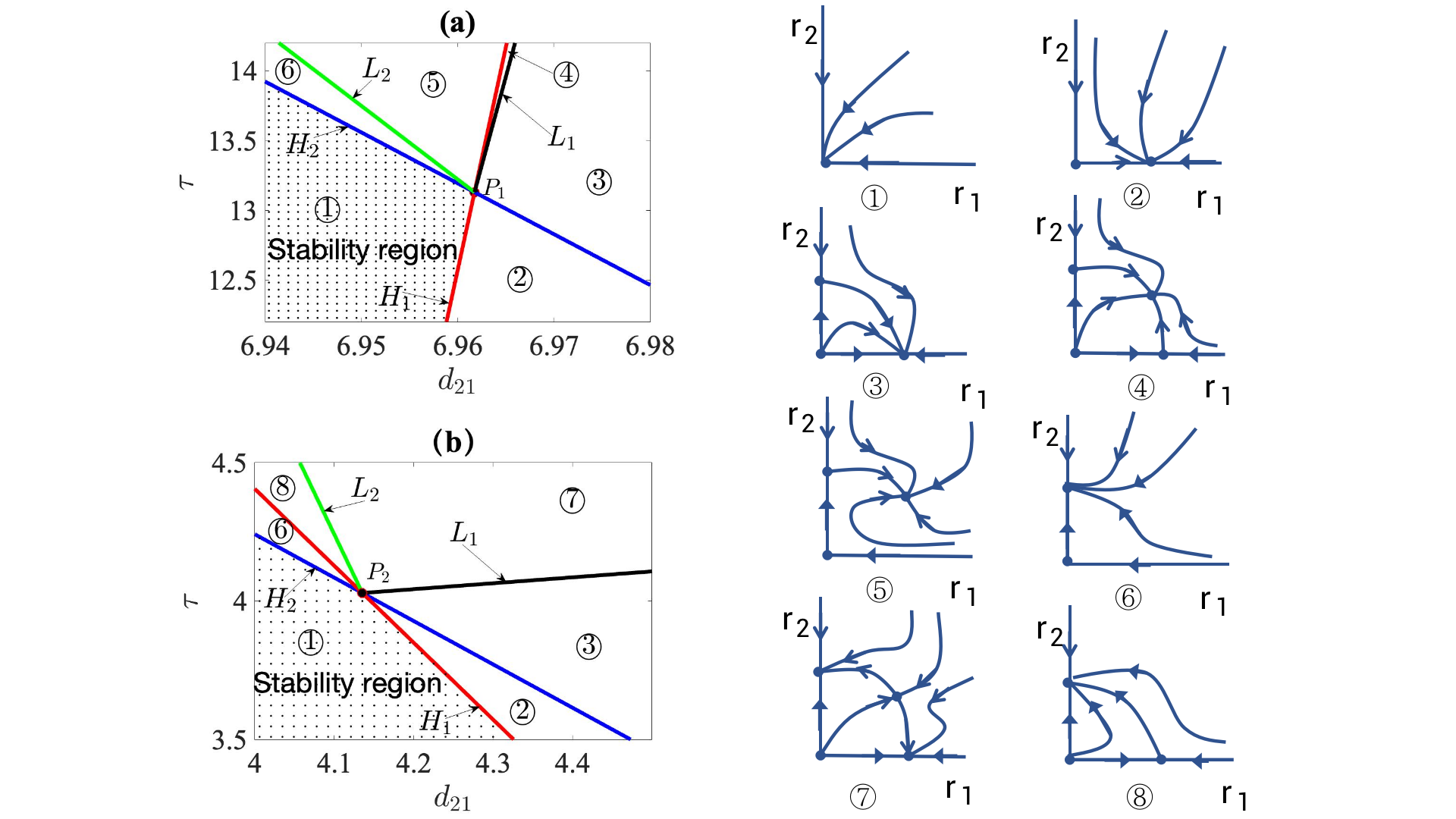}}
 \caption{The left column: partition of the region near the double Hopf bifurcation point:  (a)  partition of the region near the  point $P_1$ in Fig.\ref{Fig1}(a) with $H_j, L_j, j=1,2$,  being defined by \eqref{case1-HLJ} ;     (b) partition of the region near the  point $P_2$ in Fig.\ref{Fig1}(b) with $H_j, L_j, j=1,2$,  being defined by \eqref{case2-HLJ}.  The right column:  phase portraits for different parameter regions in the left figures. }
 \label{Fig2}
\end{figure}
 %-------------------------------------------------
 These phase portraits Fig.\ref{Fig2}  show that  for the zero equilibrium  $E_0$ of\eqref{NUMNF1}, it  is stable in Region $\textcircled{1}$  and unstable in other regions. 
Notice that  the zero equilibrium  $E_0$ of\eqref{NUMNF1} 
 corresponds to the positive equilibrium $E_*$ of
the original  system \eqref{PPMODEL}.    For $(d_{21}, \tau)=(6.96, 12.5)$ in Region  $ \textcircled{1}$,  numerical simulations in Fig.\ref{FigNS-R1}  show the stability of the positive equilibrium $E_*$ of the original  system\eqref{PPMODEL}. 
 %-------------------------------------------------

\begin{figure}[htp]
\begin{center}
\includegraphics[scale=0.2]{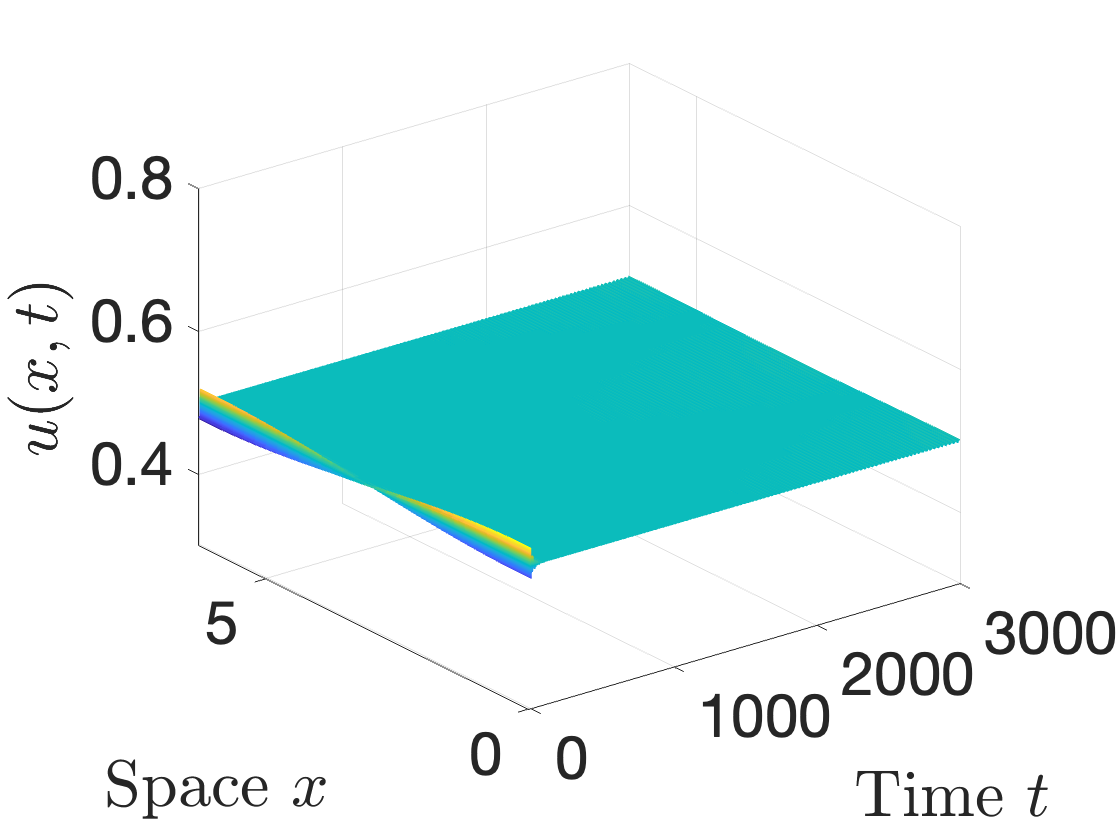}
\includegraphics[scale=0.2]{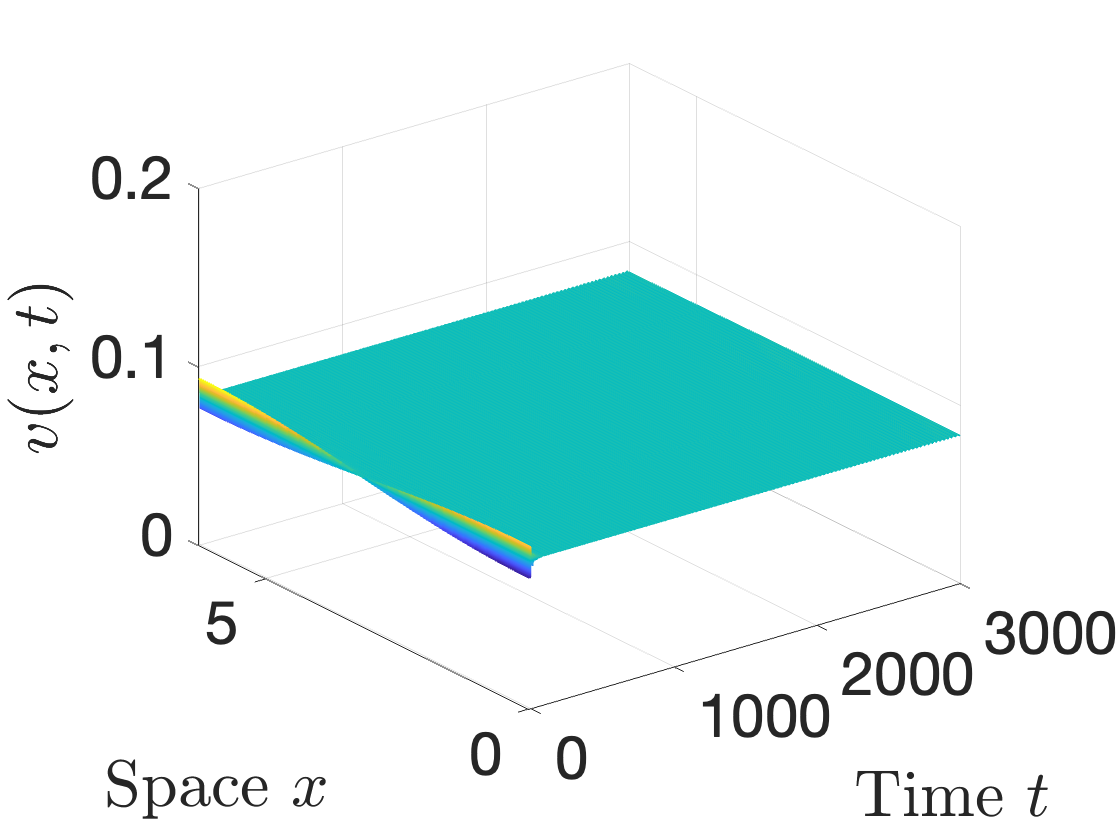}
\end{center}
 \caption{For $(d_{21}, \tau)=(6.95, 12.5)$ in Region  $ \textcircled{1}$,    the positive
equilibrium  $E_*$  of \eqref{PPMODEL}  is asymptotically stable.}
 \label{FigNS-R1}
\end{figure}
 %-------------------------------------------------
The boundary equilibria  $E_1$ and $E_2$ of \label{NUMNF}  correspond to the periodic solution of the original  system \eqref{PPMODEL}. From  the phase portraits \textcircled{2}, \textcircled{3}  and \textcircled{6} in Fig.\ref{Fig2},   it is shown that near the  neighbourhood of the point $P_1$, there exist two types of  stable periodic  solutions, respectively, bifurcating from the Hopf bifurcation at $\tau=\tau_{2,0}^-$ and $\tau=\tau_{2,1}^+$.  Figs.\ref{FigNS-R2} and \ref{FigNS-R6}   
numerically illustrate these two types of periodic solutions \eqref{PPMODEL}  bifurcating from $\tau_{2,0}^-$ and $\tau_{2,1}^+$ for the parameters $(d_{21}, \tau)$ in Regions $\textcircled{2}$ and $\textcircled{6}$, respectively. 
 
  %-------------------------------------------------
\begin{figure}[htp]
\begin{center}
\includegraphics[scale=0.2]{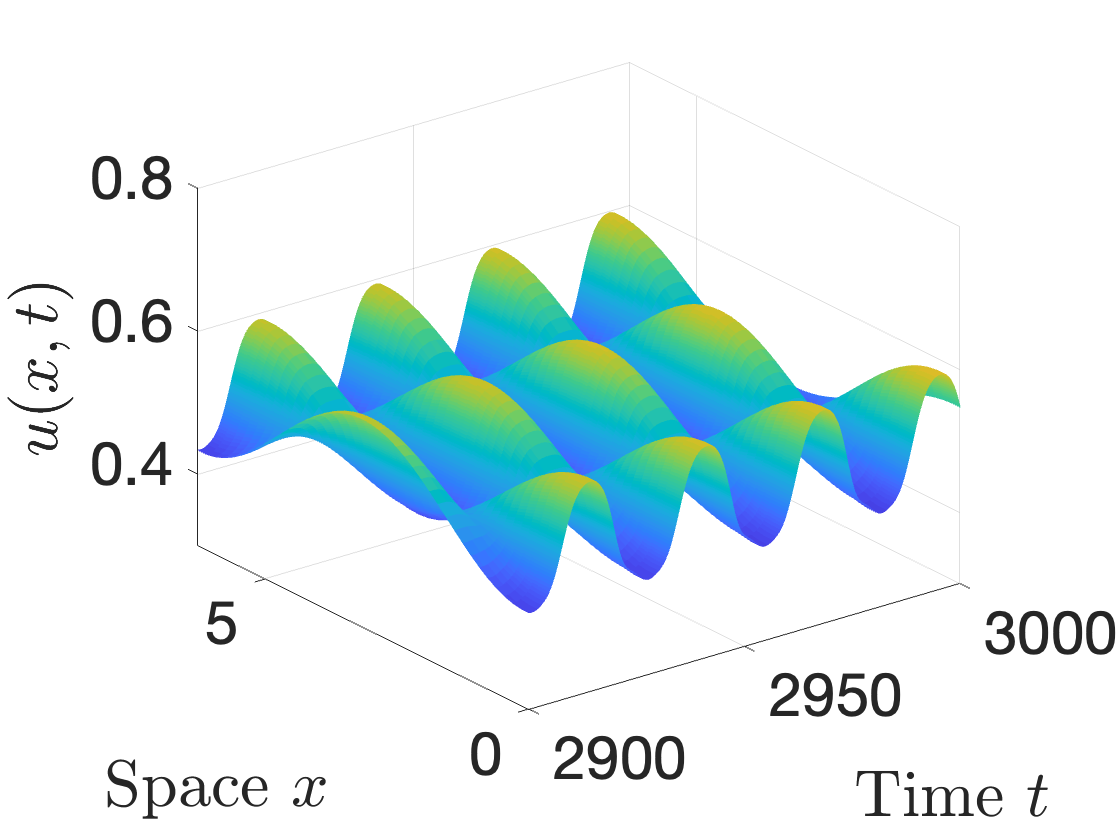}
\includegraphics[scale=0.2]{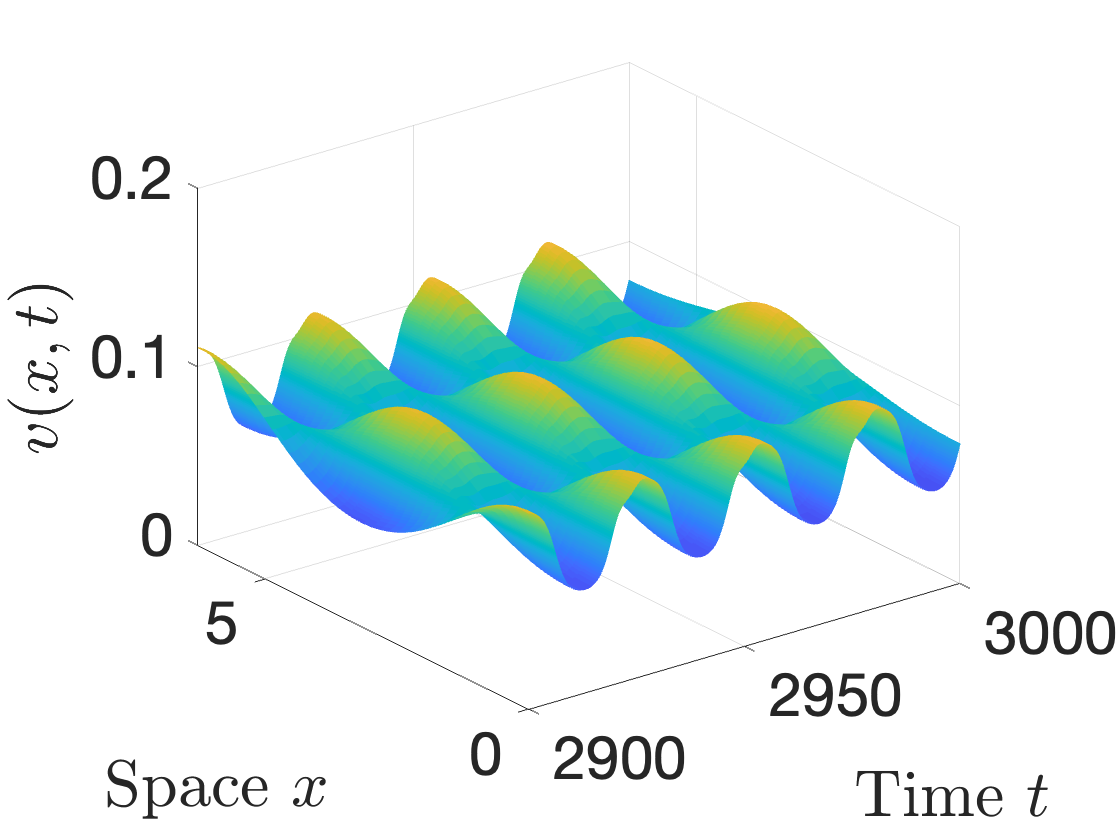}
\end{center}
 \caption{For $(d_{21}, \tau)=(6.96, 12.5)$ in Region  $ \textcircled{2}$,    the positive
equilibrium  $E_*$  is unstable  and the  periodic solution bifurcating from the Hopf bifurcating at $\tau_{2,0}^-$  appears. The initial value is chosen as $u(x, 0)=u_*+0.005\cos(x), v(x,0)=v_*-0.005\cos (x)$. }
 \label{FigNS-R2}
\end{figure}
 %------------------------------------------------- 
    %-------------------------------------------------
\begin{figure}[htp]
\begin{center}
\includegraphics[scale=0.2]{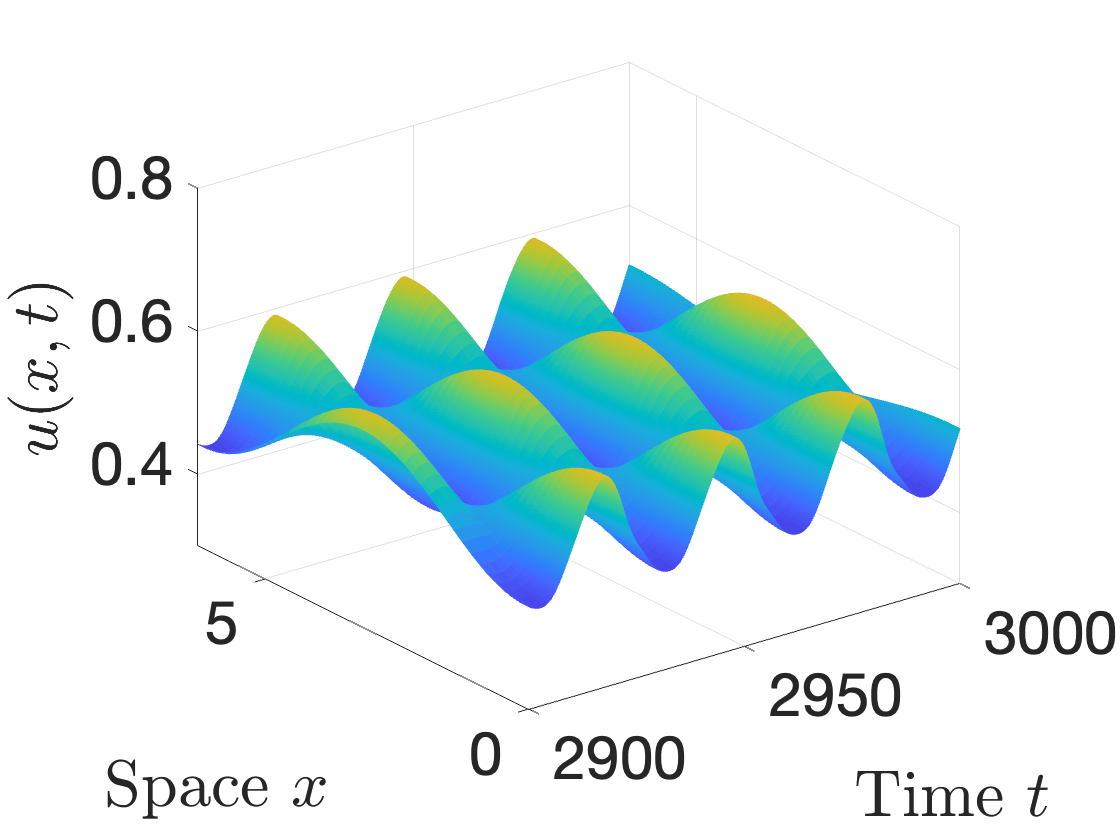}
\includegraphics[scale=0.2]{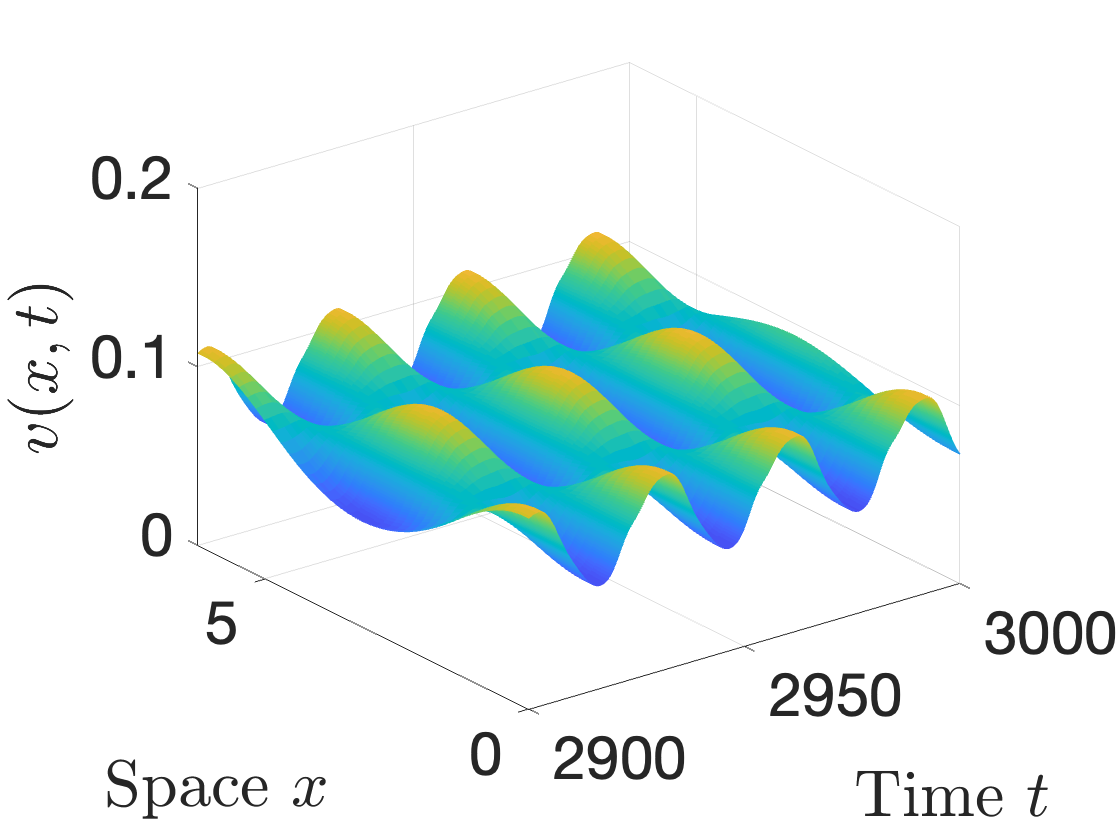}
\end{center}
 \caption{For $(d_{21}, \tau)=(6.945, 13.9)$ in Region  $ \textcircled{6}$,    the positive
equilibrium  $E_*$  is unstable  and the periodic solution bifurcating from the Hopf bifurcating at $\tau_{2,1}^+$ appears. The initial value is chosen as $u(x, 0)=u_*+0.01\cos(x), v(x,0)=v_*-0.01\cos (x)$. }
 \label{FigNS-R6}
\end{figure}

 %------------------------------------------------- 
 The positive equilibrium $E_3$  correspond to the quasi-periodic solution of  the original  system\eqref{PPMODEL}, which exist for  $(d_{21}, \tau)$ in  Regions \textcircled{4}  and \textcircled{5} of Fig.\ref{Fig2}(a).   Figs.\ref{FigNS-R51}(a) and \ref{FigNS-R51}(d) numerically illustrate this quasi-periodic solution for  $(d_{21}, \tau)$ in Region $\textcircled{5}$.  Figs.\ref{FigNS-R51}(b) and \ref{FigNS-R51}(e) are  the truncated curves of Figs.\ref{FigNS-R51}(a) and \ref{FigNS-R51}(d), respectively,  for fixed space $x=\pi/5$, and  Figs.\ref{FigNS-R51}(c) and \ref{FigNS-R51}(f)  are the truncated curves of Figs.\ref{FigNS-R51}(a) and \ref{FigNS-R51}(d),  respectively, for fixed time $t=3000$.  Fig.\ref{FigNS-R52}  illustrates the phase portrait of $u(x, t)$ and $v(x,t)$ in the $u$-$v$ plane for fixed space $x=\pi/5$,  which looks like a “bird”. 
 
   %-------------------------------------------------
\begin{figure} 
\begin{center}
\includegraphics[scale=0.18]{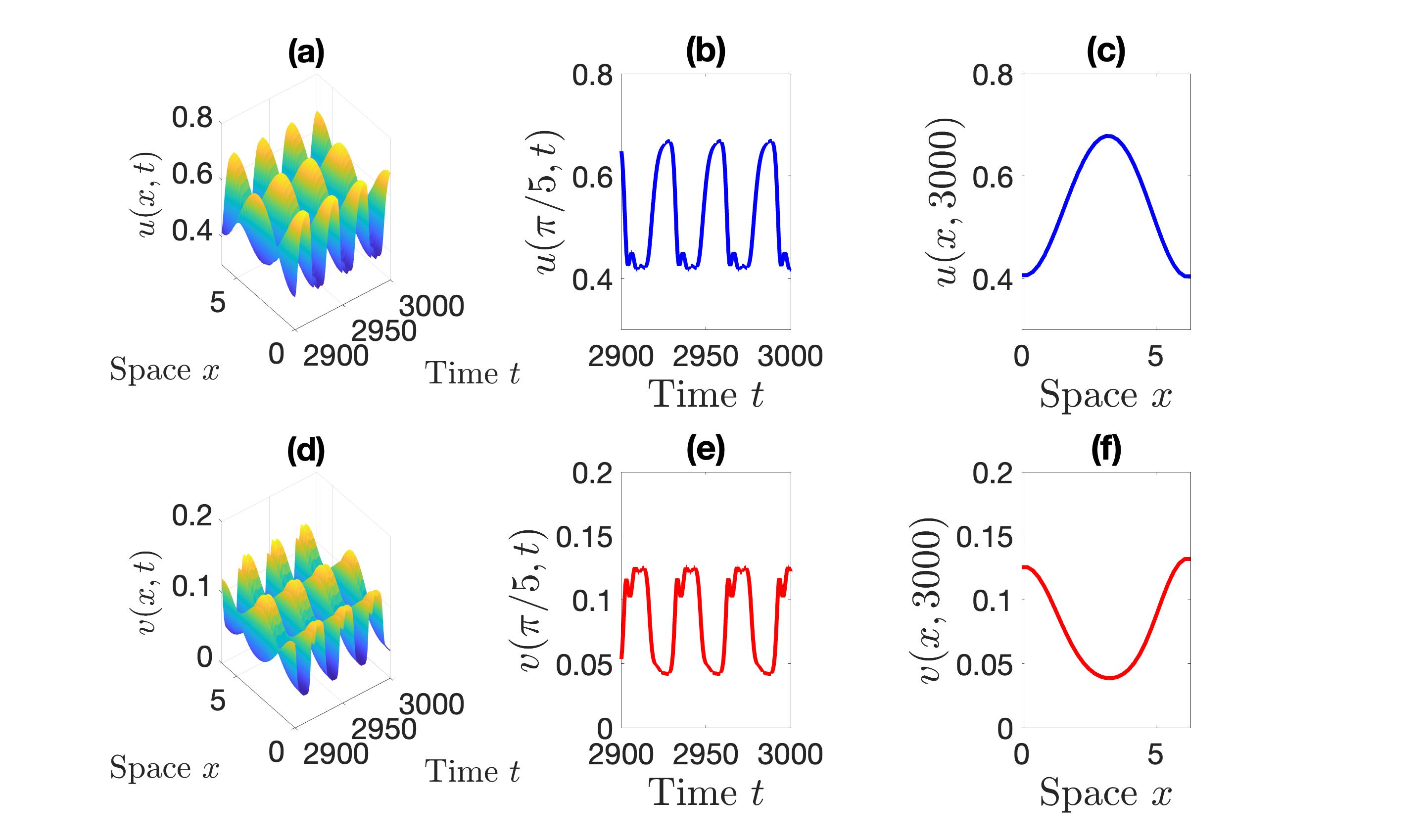}
\end{center}
 \caption{For $(d_{21}, \tau)=(6.95, 14)$ in Region  $ \textcircled{5}$,   the positive
 equilibrium  $E_*$  is unstable and  there exists the stable quasi-periodic solution with   the spatial profile like $\cos(x)$. $(a)$ and $(d)$: the spatiotemporal dynamics of the prey  $u$ and predator $v$; $(b)$ and $(e)$:  the truncated curves of  $(a)$ and $(d)$ for fixed space $x=\pi/5$ showing the evolution of $u$ and $v$ in time; $(c)$ and $(f)$: the truncated curves of  $(a)$ and $(d)$ for fixed time $t=3000$ showing the spatial profiles of $u$ and $v$. The initial value is chosen as $u(x, 0)=u_*+0.02\cos(x), v(x,0)=v_*+0.01\cos (x)$.  }
 \label{FigNS-R51}
\end{figure}

 %------------------------------------------------

\begin{figure}[htp]
\begin{center}
\includegraphics[scale=0.3]{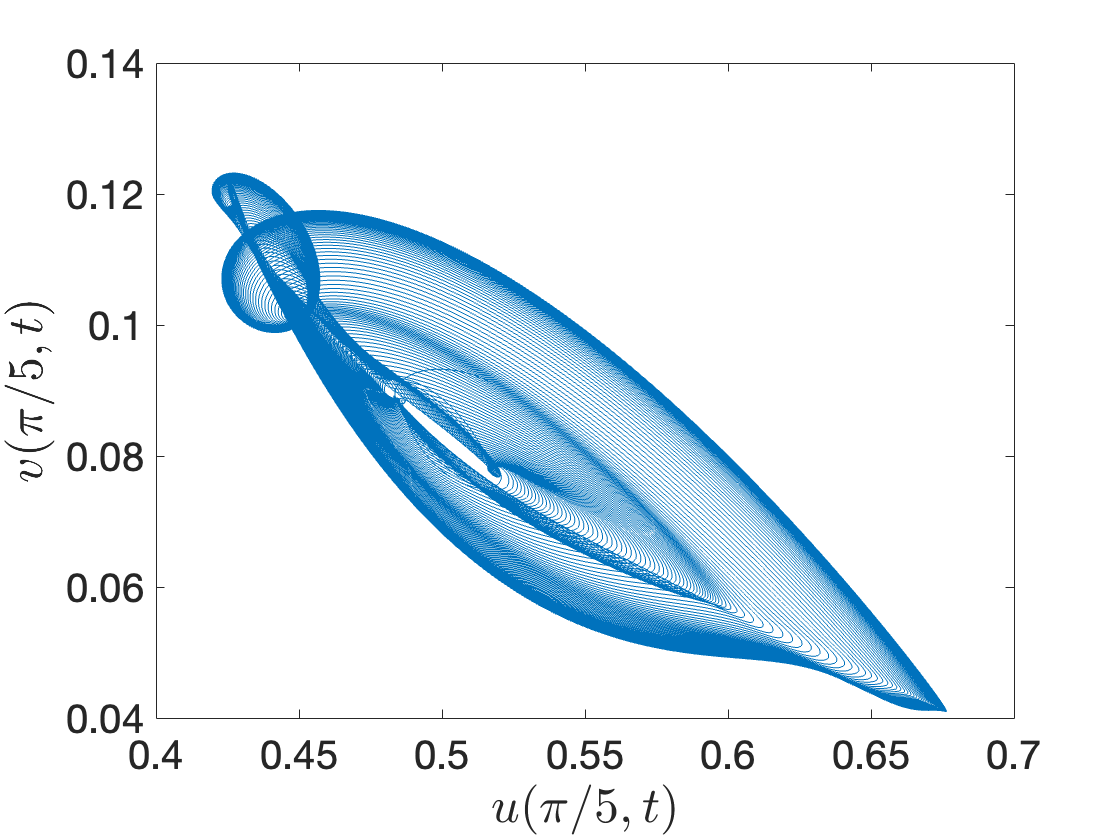}
\end{center}
 \caption{ For $(d_{21}, \tau)=(6.95, 14)$ in Region  $ \textcircled{5}$ and  fixed space $x=\pi/5$,   the evolution of the dynamics of the prey  $u$ and predator $v$  in the $u$-$v$ plane.}
 \label{FigNS-R52}
\end{figure}

%-------------------------------------------------
\begin{subsection}
{Dynamics near the double Hopf bifurcation point with different  spatial modes}
\end{subsection}
%-------------------------------------------------
Taking the same parameters as used in \cite{Song-PZ-JDE2021}: 
\begin{equation}
\label{PMSFORcase2}
a=1, ~b=\frac{3}{10}, ~c=\frac{1}{10}, ~d_{11}=\frac{6}{10}, ~d_{22}=\frac{8}{10},~\ell=2,
\end{equation}
we have $(u_*, v_*)=(1/5, 5/2)$.   For the dynamical classification near the double Hopf bifurcation point, we introduce the results of linear analysis from  \cite{Song-PZ-JDE2021},   as shown in Fig.\ref{Fig1}(b).  In Fig.\ref{Fig1}(b),  Hopf bifurcation curves $\tau=\tau_{1,0}^+$  and $\tau=\tau_{2,0}^+$ intersect at the point $P_2(4.1350,4.0276)$, which is the double Hopf bifurcation point.  This double Hopf bifurcation arises from  the interaction of  spatially inhomogeneous Hopf bifurcations with mode-$1$ and mode-$2$. For this point $P_2$, we can obtain $\omega_1^+ \doteq 0.2671, \omega_2^+\doteq 0.3666 $.

  To investigate the dynamics near this point $P_2$, 
we need to calculate  the corresponding normal form.  Setting $(d_{21}^c, \tau_c)=(13.4531,  0.6811)$,  $\omega_1=\omega_1^+$ and $\omega_2=\omega_2^+$,  and then employing the algorithm developed in Section 2, we have the  following normal form for this  double Hopf bifurcation point $P_2$  \begin{equation}\label{NUMNF2}
\begin{cases}
\dot{r}_1=r_1\left(0.0781\mu_1+0.1224\mu_2-1.4203 r_1^2-4.2174 r_2^2\right),\vspace{0.1cm}\\
\dot{r}_2=r_2\left(0.0595\mu_1+0.1653\mu_2-1.8176 r_1^2-2.3315r_2^2\right).
\end{cases}
\end{equation}
System \eqref{NUMNF2} has the similar dynamics to \eqref{NUMNF1}.  The  dynamical classification of \eqref{NUMNF2} is plotted in  
Fig.\ref{Fig1}(b), where 
\begin{equation}
\label{case2-HLJ}
\begin{array}{ll}
H_1:   ~\tau-\tau_c=-2.7794\left(d_{21}-d_{21}^c\right), &
H_2: ~ \tau-\tau_c=-1.5672\left(d_{21}-d_{21}^c\right),\vspace{0.2cm}\\
L_1: ~\tau-\tau_c=0.2140\left(d_{21}-d_{21}^c\right), d_{21}>d_{21}^c, &
L_2: ~\tau-\tau_c=-5.991\left(d_{21}-d_{21}^c\right), d_{21}<d_{21}^c.
\end{array}
\end{equation}

It follows from Fig.\ref{Fig2}(b) and the corresponding phase portraits that  there exist  stable  periodic solutions with spatial profile like $\cos(x)$ and $\cos(x/2)$, respectively,   for $\left(d_{21}, \tau\right)$ in Region \textcircled{2} and Region \textcircled{6}.   For $\left(d_{21}, \tau\right)$ in Region \textcircled{3}, there exist  a connection orbit from  the  periodic solutions with  spatial profile like $\cos(x)$ to the one spatial profile like $\cos(x/2)$,  but in the reverse direction for $\left(d_{21}, \tau\right)$ in Region \textcircled{8}.

However,  for $\left(d_{21}, \tau\right)$  in Region \textcircled{7} of Fig.\ref{Fig2}(b), there exists  a bistaiblity phenomenon, i.e., the coexistence of  two types of stable  periodic solutions with  spatial profile like $\cos(x)$  and  $\cos(x/2)$.  Taking $(d_{21}, \tau)=(4.4, 4.3)$ in Region  $ \textcircled{4}$,  Figs.\ref{Fig8} and \ref{Fig9} numerically illustrate this  bistaiblity phenomenon. With the same parameter $(d_{21}, \tau)=(4.4, 4.3)$ and different initial values,   Figs.\ref{Fig8}(a) and \ref{Fig8}(b) show that the solution $(u, v)$ finally converges to the  periodic solutions with  spatial profile like $\cos(x/2)$ for the initial values $u(x, 0)=u_*+0.1\cos(x/2), v(x,0)=v_*+0.1\cos (x/2)$ ,   but Figs.\ref{Fig8}(c) and \ref{Fig8}(d)  show that the solution $(u, v)$ finally converges to the  periodic solutions with  spatial profile like $\cos(x)$ for  the initial values  $u(x, 0)=u_*+0.1\cos(x), v(x,0)=v_*+0.1\cos (x)$.
For fixed space $x=\pi/5$,   Fig.\ref{Fig9}  numerically illustrate the  orbit of $(u, v)$ in the $u$-$v$ plane.
 
 %---------------------------------------------------------------------------------------

 %----------------------------------------------------------------------------
    \begin{figure}[htp]
\centering
\subfloat[]{\includegraphics[scale=0.2]{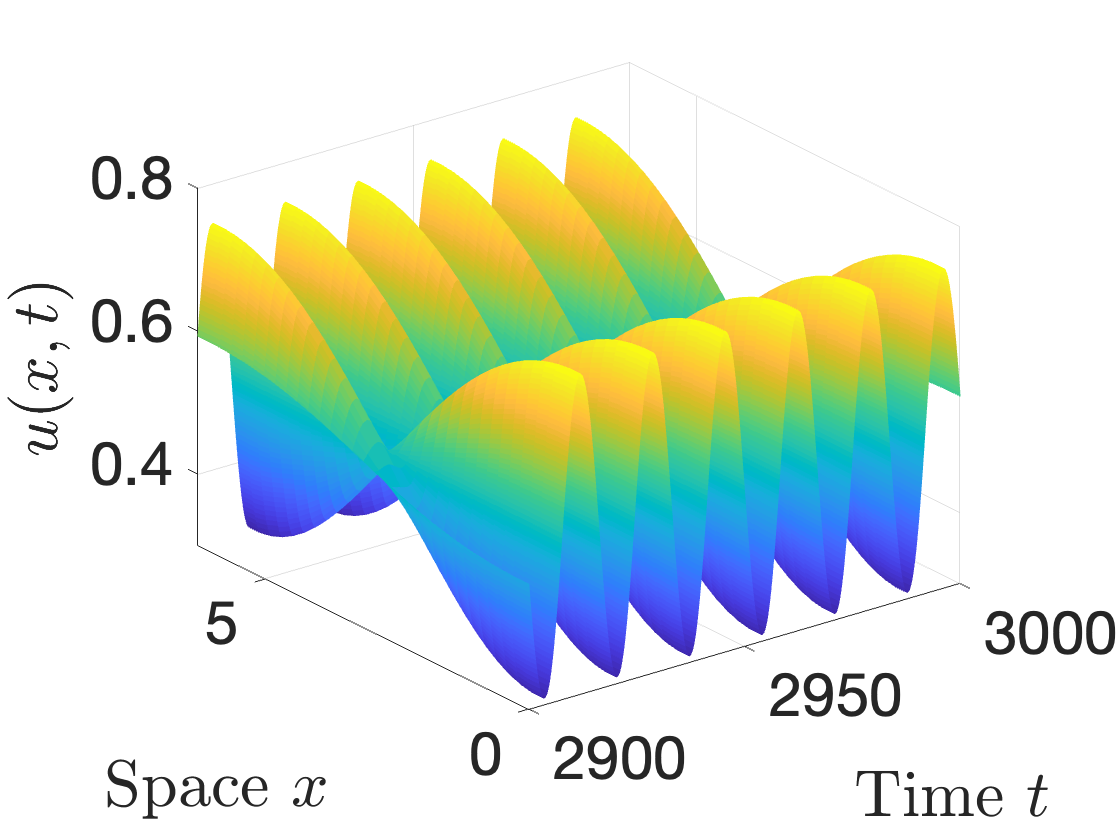}}
\subfloat[]{\includegraphics[scale=0.2]{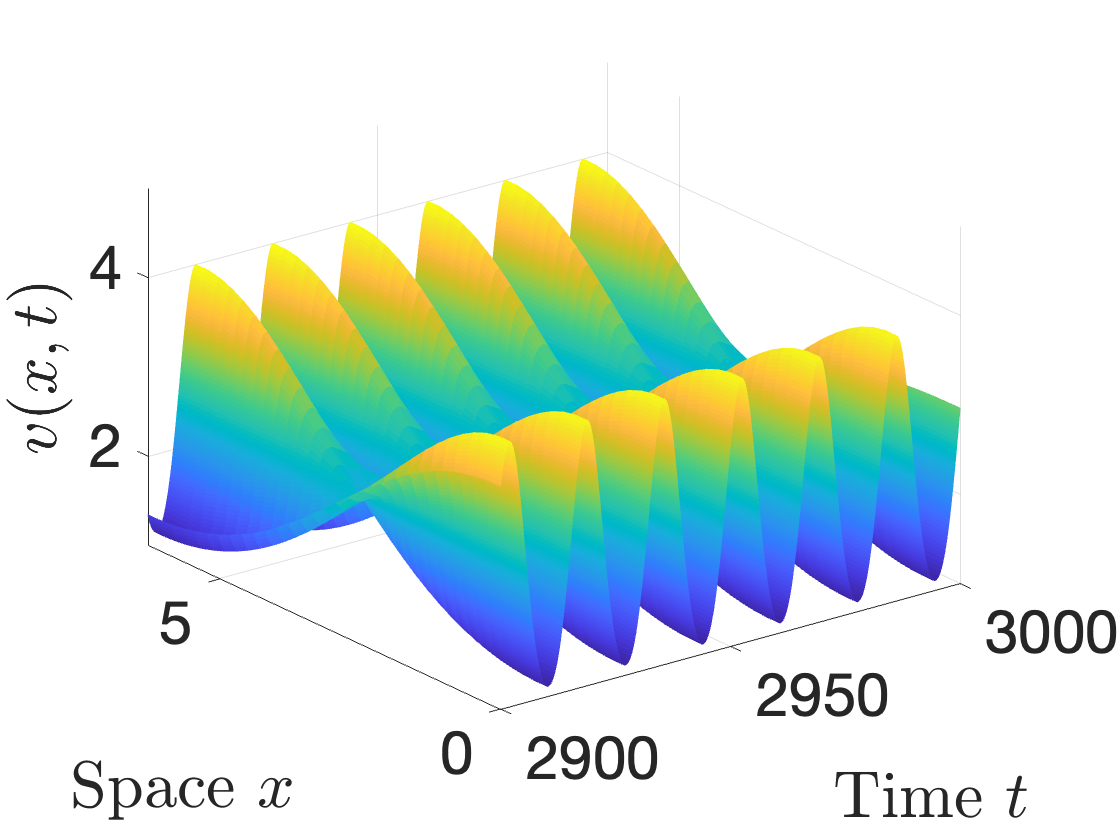}}\quad
\subfloat[]{\includegraphics[scale=0.2]{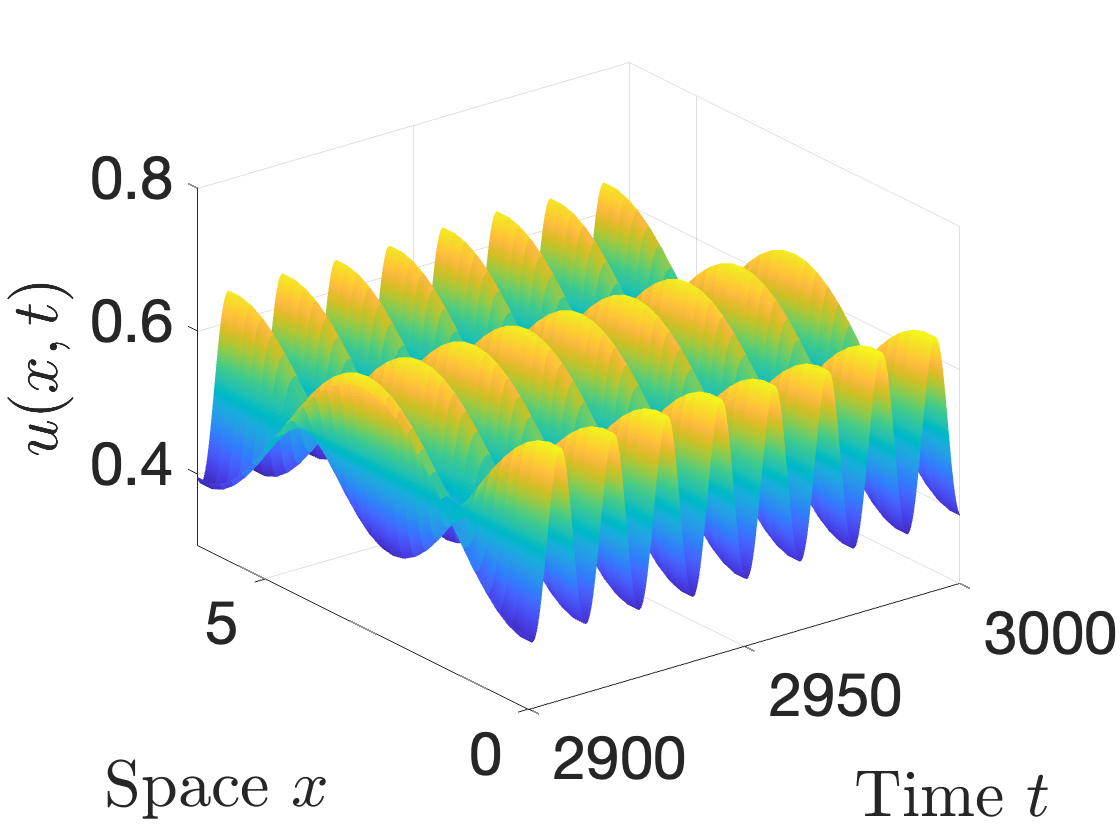}}
\subfloat[]{\includegraphics[scale=0.2]{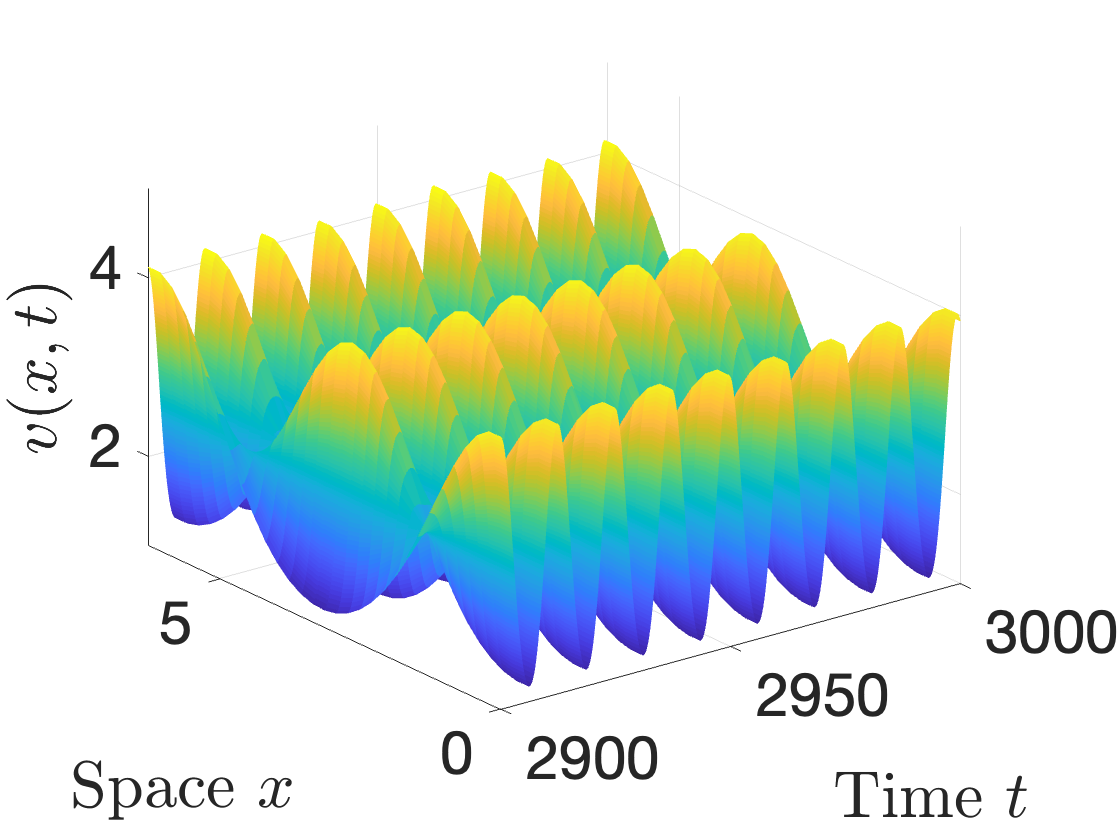}}
  \caption{ For $(d_{21}, \tau)=(4.4, 4.3)$ in Region  $ \textcircled{7}$ of Fig.\ref{Fig2}(b), two types of stable spatially inhomogeneous periodic solutions with different spatial profiles coexist. (a)-(b) The initial value is chosen as $u(x, 0)=u_*+0.1\cos(x/2), v(x,0)=v_*+0.1\cos (x/2)$.; (c)-(d) The initial value is chosen as $u(x, 0)=u_*+0.1\cos(x), v(x,0)=v_*+0.1\cos (x)$.  }
  \label{Fig8}
  \end{figure}
%---------------------------------------------------------------------------

\begin{figure}[htp]
\begin{center}
\includegraphics[scale=0.3]{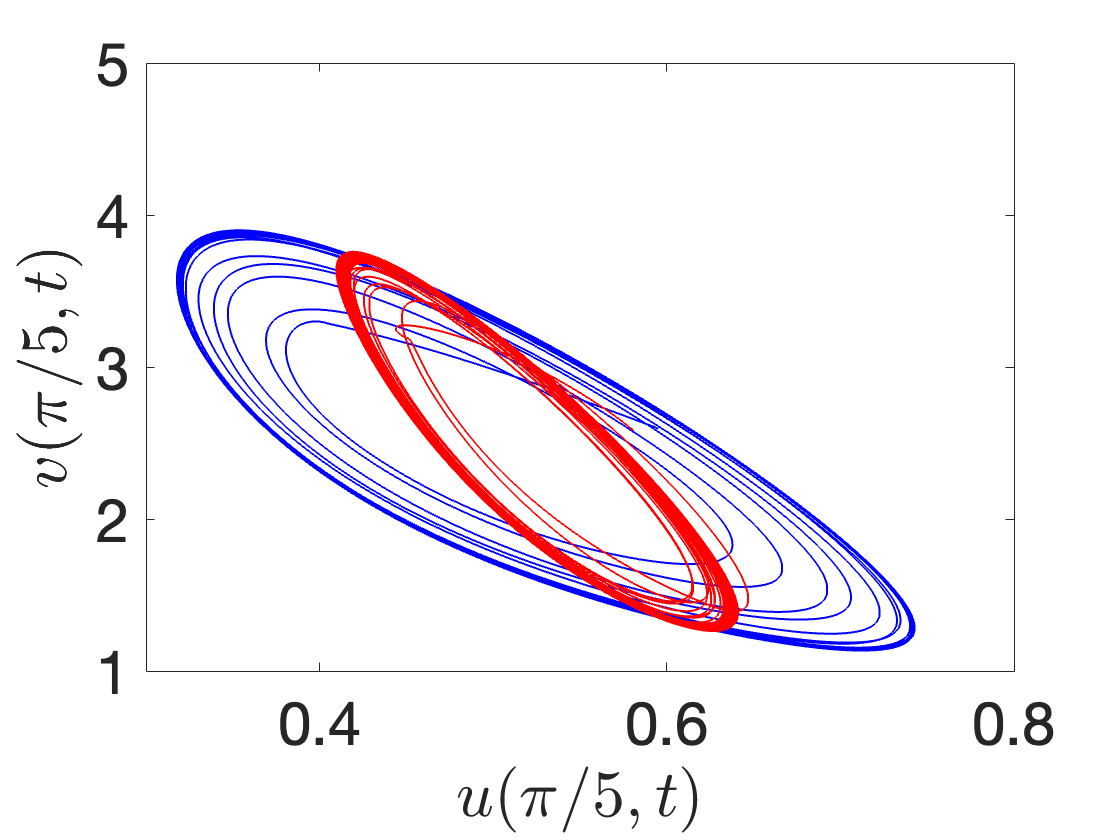}
\end{center}
 \caption{ For $(d_{21}, \tau)=(4.4, 4.3)$ in Region  $ \textcircled{7}$  of Fig.\ref{Fig2}(b) and  fixed space $x=\pi/5$,  two stable periodic orbits in $u$-$v$ plane coexist.}
 \label{Fig9}
\end{figure}

%--------------------------------------------------------------------------------------
\begin{section}
{Discussion}
\end{section}
%-------------------------------------------------------------------------------------- 

In this paper, we have paid our attention on the development of an algorithm for computing the normal form of the double Hopf bifurcation induced by the memory-based diffusion coefficient and memory delay  for the memory-based diffusion system.  The calculating formulae of the second and third terms in the normal form are explicitly derived from those in the original system. This algorithm, which is for the memory-based diffusion system, where the delay appears in the directional diffusion terms and the diffusion terms are nonlinear, is the counterpart of the existing algorithm for the related theory of the classical reaction-diffusion system. By analysing the corresponding normal form, we can determine the dynamical classification near the double Hopf bifurcation point.

Employing the obtained theoretical results to the predator-prey system with Holling-II functional response, we investigate the spatio-temporal dynamics due to  the interaction of  double Hopf bifurcation for two cases: (i) double Hopf bifurcation with the same spatial mode;  (ii) double Hopf bifurcation with the different spatial modes.  For the former, we find two types of stable spatially inhomogeneous periodic solutions with the same spatial mode and similar oscillatory  frequency, and the quasi-periodic solution in some region near the double Hopf bifurcation point. For latter,  we find two types of stable spatially inhomogeneous periodic solutions with different spatial mode and oscillatory  frequency,  and the coexistence of  two stable  periodic solutions with different spatial mode  in some region near the double Hopf bifurcation point.  In both cases,  the pattern  transitions from a unstable periodic solution  to  a stable periodic solution are found. 

We would also like to iterate that the theoretical results in Section 2 are derived for the non-resonant and weakly resonant double Hopf bifurcations. Thus it is not applicable for the strongly resonant double Hopf bifurcation.  The algorithm for computing the normal form for the case of strongly double Hopf bifurcation for the memory-based diffusion  system \eqref{PPMODEL}  still remains open and we leave this for the further investigation.

%
%%----------------------
%\noindent\textbf{Acknowledgement:} 

%-----------------------------------------
 \begin{appendices}
 \appendixpage
 \addappheadtotoc

 %--------------------------------
\section{Expressions of $A_{i_1i_2i_3i_4}^{(d,j)},  j=1,2,3.$}
%---------------------------------------------
$$
 A_{1010}^{(d,1)}=-
2d_{21}^c\tau_c\left(\begin{array}{c}
0\vspace{0.2cm}\\
\phi_{n_1}^{(1)}(-1) \phi_{n_2}^{(2)}(0)+\phi_{n_2}^{(1)}(-1) \phi_{n_1}^{(2)}(0)
\end{array}\right),
$$
$$
 A_{1001}^{(d,1)}=-
2d_{21}^c\tau_c\left(\begin{array}{c}
0\vspace{0.2cm}\\
\phi_{n_1}^{(1)}(-1) \overline{\phi}_{n_2}^{(2)}(0)+\overline{\phi}_{n_2}^{(1)}(-1) \phi_{n_1}^{(2)}(0)
\end{array}\right),
$$

$$
 A_{0110}^{(d,1)}=\overline{A^{(d,1)}_{1001}},\quad\quad
 A_{0101}^{(d,1)}=\overline{A^{(d,1)}_{1010}},
$$

$$
 A_{2000}^{(d,1)}=A_{2000}^{(d,2)}=-
2d_{21}^c\tau_c\left(\begin{array}{c}
0\vspace{0.2cm}\\
\phi_{n_1}^{(1)}(-1) \phi_{n_1}^{(2)}(0)
\end{array}\right),
$$

$$
 A_{0020}^{(d,1)}=A_{0020}^{(d,2)}=-
2d_{21}^c\tau_c\left(\begin{array}{c}
0\vspace{0.2cm}\\
\phi_{n_2}^{(1)}(-1) \phi_{n_2}^{(2)}(0)
\end{array}\right),
$$

$$
 A_{1100}^{(d,1)}=A_{1100}^{(d,2)}=-
2d_{21}^c\tau_c\left(\begin{array}{c}
0\vspace{0.2cm}\\
2\mbox{Re}\left\{\phi_{n_1}^{(1)}(-1) \overline{\phi}_{n_1}^{(2)}(0)\right\}
\end{array}\right),
$$

$$
 A_{0011}^{(d,1)}=A_{0011}^{(d,2)}=-
2d_{21}^c\tau_c\left(\begin{array}{c}
0\vspace{0.2cm}\\
2\mbox{Re}\left\{\phi_{n_2}^{(1)}(-1) \overline{\phi}_{n_2}^{(2)}(0)\right\}
\end{array}\right),
$$

$$
 A_{1010}^{(d,2)}=-
2d_{21}^c\tau_c\left(\begin{array}{c}
0\vspace{0.2cm}\\
\phi_{n_1}^{(1)}(-1) \phi_{n_2}^{(2)}(0)
\end{array}\right), ~~ A_{1010}^{(d,3)}=-
2d_{21}^c\tau_c\left(\begin{array}{c}
0\vspace{0.2cm}\\
\phi_{n_2}^{(1)}(-1) \phi_{n_1}^{(2)}(0)
\end{array}\right),
$$

$$
 A_{1001}^{(d,2)}=-
2d_{21}^c\tau_c\left(\begin{array}{c}
0\vspace{0.2cm}\\
\phi_{n_1}^{(1)}(-1) \overline{\phi}_{n_2}^{(2)}(0)
\end{array}\right), ~~ A_{1001}^{(d,3)}=-
2d_{21}^c\tau_c\left(\begin{array}{c}
0\vspace{0.2cm}\\
\overline{\phi}_{n_2}^{(1)}(-1) \phi_{n_1}^{(2)}(0)
\end{array}\right),
$$

$$
 A_{0200}^{(d,1)}=A_{0200}^{(d,2)} =\overline{A^{(d,1)}_{2000}},\quad\quad
 A_{0002}^{(d,1)}=A_{0002}^{(d,2)}=\overline{A^{(d,1)}_{0020}},
$$

$$
 A_{0110}^{(d,2)}=\overline{A^{(d,2)}_{1001}}, \quad A_{0110}^{(d,3)}=\overline{A^{(d,3)}_{1001}}, \quad
 A_{0101}^{(d,2)}=\overline{A^{(d,2)}_{1010}}, \quad A_{0101}^{(d,3)}=\overline{A^{(d,3)}_{1010}}.
$$

  %---------------------------------------------------
\section{ The calculation of $\bf{ \mbox{Proj}_{S}\left(\left(D_{w,w_x,w_{xx}}f^{(1,2)}_2(z,0,0)\right)U_2^{(2,d)}(z, 0)(\theta)\right)}$}
 %--------------------------------------
Denote $\varphi(\theta)= \Phi(\theta) z_x$ and
$$
\begin{array}{lll}
&&F_2^d( \varphi(\theta),  w, w_x, w_{xx})=F_2^d( \Phi(\theta) z_x+w,0)\vspace{0.3cm}\\
&=&
-2d_{21}^c\tau_c\left(\begin{array}{c}
0\vspace{0.2cm}\\
\left(\varphi^{(1)}_{xx}(-1)+w_{xx}^{(1)}(-1)\right)\left(\varphi^{(2)}(0)+w^{(2)}(0)\right)
\end{array}\right)
\vspace{0.3cm}\\
&&
-2d_{21}^c\tau_c\left(\begin{array}{c}
0\vspace{0.2cm}\\
\left(\varphi^{(1)}_x(-1)+w_{x}^{(1)}(-1)\right)\left(\varphi^{(2)}_{x}(0)+w_{x}^{(2)}(0)\right)
\end{array}\right),
\end{array}
$$
%-------------------------------------------------------------------------------------
$$
\begin{array}{l}
\widetilde{S}_2^{(d,1)}(\varphi (\theta), w)=
-2d_{21}^c\tau_c\left(\begin{array}{c}
0\vspace{0.2cm}\\
\varphi^{(1)}_{xx}(-1)w^{(2)}(0)
\end{array}\right),
%---------------------------------
\vspace{0.3cm}\\
\widetilde{S}_2^{(d,2)}(\varphi (\theta), w_{x} )
=
-2d_{21}^c\tau_c\left(\begin{array}{c}
0\vspace{0.2cm}\\
\varphi^{(1)}_x(-1)w_{x}^{(2)}(0)
\end{array}\right)-2d_{21}^c\tau_c\left(\begin{array}{c}
0\vspace{0.2cm}\\
\varphi^{(2)}_{x}(0)w_{x}^{(1)}(-1)
\end{array}\right),

%---------------------------------
\vspace{0.3cm}\\
\widetilde{S}_2^{(d,3)}(\varphi(\theta), w_{xx} )=
-2d_{21}^c\tau_c\left(\begin{array}{c}
0\vspace{0.2cm}\\
\varphi^{(2)}(0)w_{xx}^{(1)}(-1)
\end{array}\right).
\end{array}
$$
%-------------------------------------------------------------------------------------

Then, we have
 $$
\begin{array}{lll}
&&\left.D_{w, w_x, w_{xx}}F_2^d( \varphi(\theta),  w, w_x, w_{xx})\right|_{w, w_x, w_{xx}=0}U_2^{(2,d)}(z, 0)(\theta)\vspace{0.2cm}\\
&=&\widetilde{S}_2^{(d,1)}\left(\varphi (\theta),  h(\theta,z)\right)
+\widetilde{S}_2^{(d,2)}\left(\varphi (\theta),   h_x(\theta,z)\right)
\vspace{0.2cm}\\
&&
+\widetilde{S}_2^{(d,3)}\left(\varphi(\theta),  h_{xx}(\theta,z) \right)

\end{array}
$$

and
$$
\begin{array}{lll}
&&\left(\begin{array}{c}
                        \left[\widetilde{S}_2^{(d,1)}\left(\varphi (\theta), h(\theta,z) \right), \beta_{\nu}^{(1)}\right] \vspace{0.2cm}\\
  \left[\widetilde{S}_2^{(d,1)}\left(\varphi (\theta),  h(\theta,z) \right), \beta_{\nu}^{(2)}\right]
                         \end{array}
\right)\vspace{0.2cm}\\
&=&-(n_1/\ell)^2\sum\limits_{n\in \mathbb{N}_0} b_{n_1,n,\nu}\left(\mathcal{S}_2^{(d,1)}\left(\phi_{n_1}(\theta)z_1(\theta),h_n(\theta, z)\right)+\mathcal{S}_2^{(d,1)}\left(\overline{\phi}_{n_1}(\theta)z_2,h_n(\theta, z)\right)\right)\vspace{0.2cm}\\
&&
-(n_2/\ell)^2\sum\limits_{n\in \mathbb{N}_0} b_{n_2,n,\nu}\left(\mathcal{S}_2^{(d,1)}\left(\phi_{n_2}(\theta)z_3,h_n(\theta, z)\right)+\mathcal{S}_2^{(d,1)}\left(\overline{\phi}_{n_2}(\theta)z_4,h_n(\theta, z)\right)\right),
\end{array}
$$

$$
\begin{array}{lll}
&&\left(\begin{array}{c}
                        \left[\widetilde{S}_2^{(d,2)}\left(\varphi (\theta), h_x(\theta,z)\right), \beta_{\nu}^{(1)}\right] \vspace{0.2cm}\\
  \left[\widetilde{S}_2^{(d,2)}\left(\varphi (\theta), h_x(\theta,z) \right), \beta_{\nu}^{(2)}\right]
                         \end{array}
\right)\vspace{0.2cm}\\
&=&(n_1/\ell)\sum\limits_{n\in \mathbb{N}_0} (n/\ell)b^s_{n_1,n,\nu}\left(\mathcal{S}_2^{(d,2)}\left(\phi_{n_1}(\theta)z_1,h_n(\theta, z)\right)+\mathcal{S}_2^{(d,2)}\left(\overline{\phi}_{n_1}(\theta)z_2,h_n(\theta, z)\right)\right)\vspace{0.2cm}\\
&&
+(n_2/\ell)\sum\limits_{n\in \mathbb{N}_0} (n/\ell)b^s_{n_2,n,\nu}\left(\mathcal{S}_2^{(d,2)}\left(\phi_{n_2}(\theta)z_3,h_n(\theta, z)\right)+\mathcal{S}_2^{(d,2)}\left(\overline{\phi}_{n_2}(\theta)z_4,h_n(\theta, z)\right)\right),
\end{array}
$$

$$
\begin{array}{lll}
&&\left(\begin{array}{c}
                        \left[\widetilde{S}_2^{(d,3)}\left(\varphi(\theta), h_{xx}(\theta,z) \right), \beta_{\nu}^{(1)}\right] \vspace{0.2cm}\\
  \left[\widetilde{S}_2^{(d,3)}\left(\varphi(\theta), h_{xx}(\theta,z) \right), \beta_{\nu}^{(2)}\right]
                         \end{array}
\right)\vspace{0.2cm}\\
&=&-\sum\limits_{n\in \mathbb{N}_0} (n/\ell)^2b_{n_1,n,\nu}\left(\mathcal{S}_2^{(d,3)}\left(\phi_{n_1}(\theta)z_1,h_n(\theta, z)\right)+\mathcal{S}_2^{(d,3)}\left(\overline{\phi}_{n_1}(\theta)z_2,h_n(\theta, z)\right)\right)\vspace{0.2cm}\\
&&
-\sum\limits_{n\in \mathbb{N}_0} (n/\ell)^2b_{n_2,n,\nu}\left(\mathcal{S}_2^{(d,3)}\left(\phi_{n_2}(\theta)z_3,h_n(\theta, z)\right)+\mathcal{S}_2^{(d,3)}\left(\overline{\phi}_{n_2}(\theta)z_4,h_n(\theta, z)\right)\right),
\end{array}
$$
where, for $\nu=n_1, n_2,$ $b_{n_j,n,\nu}$ is defined as in \eqref{BKJM} and
$$
b^s_{n_j,n,\nu}=\int_0^{\ell\pi}\xi_{n_j}(x)\xi_n(x)\gamma_{\nu}(x)dx=\left\{\begin{array}{ll}
\frac{1}{\sqrt{2\ell\pi}}, & n=2n_j,  \nu=n_j,\vspace{0.2cm}\\
\frac{1}{\sqrt{2\ell\pi}}, & n=n_1+n_2,  \nu=n_{j+(-1)^{j+1}},\vspace{0.2cm}\\
\frac{1}{\sqrt{2\ell\pi}}, & n=n_2-n_1, n_j=n_2, \nu=n_1,  \vspace{0.2cm}\\
-\frac{1}{\sqrt{2\ell\pi}}, & n=n_2-n_1,  n_j=n_1,  \nu=n_2,  \vspace{0.2cm}\\
 0, & \mbox{otherwise},
 \end{array}
\right.
$$
%-------------------------------------------------------------------------------------
and for $\phi(\theta)=\left(\phi^{(1)}(\theta)\quad \phi^{(2)}(\theta)\right )^T, y(\theta)=\left(y^{(1)}(\theta)\quad y^{(2)}(\theta)\right )^T\in C\left([-1, 0], \mathbb{R}^2\right)$,
\begin{equation}
\label{SDJ}
\left\{
\begin{array}{l}
\mathcal{S}_2^{(d,1)}(\phi (\theta), y(\theta))=
-2d_{21}^c\tau_c\left(\begin{array}{c}
0\vspace{0.2cm}\\
\phi^{(1)}(-1)y^{(2)}(0)
\end{array}\right),
%---------------------------------
\vspace{0.3cm}\\
\mathcal{S}_2^{(d,2)}(\phi (\theta), y(\theta) )
=
-2d_{21}^c\tau_c\left(\begin{array}{c}
0\vspace{0.2cm}\\
\phi^{(1)}(-1)y^{(2)}(0)
\end{array}\right)-2d_{21}^c\tau_c\left(\begin{array}{c}
0\vspace{0.2cm}\\
\phi^{(2)}(0)y^{(1)}(-1)
\end{array}\right),

%---------------------------------
\vspace{0.3cm}\\
\mathcal{S}_2^{(d,3)}(\phi(\theta), y(\theta) )=
-2d_{21}^c\tau_c\left(\begin{array}{c}
0\vspace{0.2cm}\\
\phi^{(2)}(0)y^{(1)}(-1)
\end{array}\right).
\end{array}
\right.
\end{equation}
%-------------------------------------------------------------------------------------

Then, from \eqref{F212W},  \eqref{U2D}, \eqref{U22} and \eqref{U2DX},   we have
       $$
\begin{array}{lll}
&& \left(D_{w,w_x,w_{xx}}f^{(1,2)}_2(z,0,0)\right)U_2^{(2,d)}(z, 0)(\theta)\vspace{0.3cm}\\
 &=&\Psi(0)
   \left(\begin{array}{c}
\left[\left.D_{w, w_x, w_{xx}}F_2^d( \varphi(\theta),  w, w_x, w_{xx})\right|_{w, w_x, w_{xx}=0}U_2^{(2,d)}(z, 0)(\theta), \beta_{\nu}^{(1)}\right]\vspace{0.2cm}\\
\left[\left.D_{w, w_x, w_{xx}}F_2^d( \varphi(\theta),  w, w_x, w_{xx})\right|_{w, w_x, w_{xx}=0}U_2^{(2,d)}(z, 0)(\theta), \beta_{\nu}^{(2)}\right]
\end{array}\right)_{\nu=n_1}^{\nu=n_2}
\end{array}
$$
and then we obtian \eqref{PDWF22}.

 %---------------------------------------------------
\section{Calculation of   $h_{n,q_1q_2q_3q_4}(\theta)$. }

From  \cite{Faria-00TAMS} , we have
$$
\begin{array}{lll}
&&M_2^2\left( h_n(\theta,z) \gamma_n(x)\right)\vspace{0.2cm}\\
&=& D_z\left( h_n(\theta,z) \gamma_n(x)\right)Bz -A_{\mathcal{Q}^1}\left( h_n(\theta,z) \gamma_n(x)\right),
\end{array}
$$
which leads to
 \begin{equation}
 \label{M22HJ}
 \begin{array}{lll}
 &&   \left(\begin{array}{c}
\left[M_2^2\left( h_n(\theta,z) \gamma_n(x)\right), \beta_n^{(1)}\right]\vspace{0.2cm}\\
\left[M_2^2\left( h_n(\theta,z) \gamma_n(x)\right), \beta_n^{(2)}\right]
\end{array}\right)\vspace{0.2cm}\\
&=&2i\omega_{1c}\left(h_{n,2000} (\theta)z_1^2- h_{n,0200}(\theta)z_2^2\right)+ 2i\omega_{2c}\left(h_{n,0020} (\theta)z_3^2-h_{n,0002}(\theta)z_4^2\right)\\
&&+i\left(\omega_{1c}+\omega_{2c}\right)h_{n,1010}(\theta)z_1z_3+i\left(\omega_{1c}-\omega_{2c}\right)h_{n,1001}(\theta)z_1z_4 \\
&&-i\left(\omega_{1c}-\omega_{2c}\right)h_{n,0110}(\theta)z_2z_3-i\left(\omega_{1c}+\omega_{2c}\right)h_{n,0101}(\theta)z_2z_4\vspace{0.2cm}\\
&&-\left(\dot{h}_n(\theta, z)+X_0(\theta)\left(\mathscr{L}_0 \left(h_n(\theta, z)\right)- \dot{h}_n(0, z)\right)\right),
\end{array}
\end{equation}
where
$$
\mathscr{L}_0\left( h_n(\theta, z)\right)= -\tau_c(n/\ell)^2\left( D_1 h_n(0, z)+D_2h_n(-1, z) \right)+\tau_cAh_n(\theta, z).
$$

By \eqref{FJ2}, we get
$$
f_2^2(z,0,0)=X_0(\theta)\widetilde{F}_2\left(\Phi(\theta) z_x,0\right) -\pi\left(X_0(\theta) \widetilde{F}_2\left(\Phi(\theta) z_x,0\right)\right).
$$
%----------------------------------------------------------------------------------------
\subsection{Case 1: $n_1\not= n_2$. }
 %----------------------------------------------------------------------------------------------
By \eqref{PIX0},  \eqref{TLDF2}, \eqref{F2EXP} and \eqref{F2DEXP}, we have , for $n_2\not= n_1$,
\begin{equation}\label{EQF22B}
\begin{array}{lll}
 &&
\left(\begin{array}{c}
\left[f_2^2(z,0,0),\beta_n^{(1)}\right]\vspace{0.1cm}\\
\left[f_2^2(z,0,0),\beta_n^{(2)}\right]
\end{array}\right)\vspace{0.2cm}\\
&=&\begin{cases}
\left\{\begin{array}{l}
\frac{1}{\sqrt{\ell\pi}} X_0(\theta)\left(A_{2000}z_1^2+A_{0200}z_2^2+A_{0020}z_3^2\right.\vspace{0.2cm}\\
\left.+A_{0002}z_4^2+A_{1100}z_1z_2+A_{0011}z_3z_4\right),
\end{array}\right.
& n=0,\vspace{0.3cm}\\
%---------------------------------
\frac{1} {\sqrt{2\ell\pi}} X_0(\theta)\left(\widetilde{A}_{2000}z_1^2+\widetilde{A}_{0200}z_2^2+\widetilde{A}_{1100}z_1z_2\right),
&n=2n_1, n_2\not=2n_1, n_2\not=3n_1,\vspace{0.3cm}\\
%---------------------------
\left\{\begin{array}{l}
\frac{1}{\sqrt{2\ell\pi}}\left(X_0(\theta)I_2-\Phi_{n_2}(\theta)\Psi_{n_2}(0)\right)\left(\widetilde{A}_{2000}z_1^2\right.\vspace{0.2cm}\\
\left.+\widetilde{A}_{0200}z_2^2+\widetilde{A}_{1100}z_1z_2\right),
\end{array}\right.
&n=2n_1, n_2=2n_1, \vspace{0.3cm}\\
%---------------------------
\left\{\begin{array}{l}
\frac{1} {\sqrt{2\ell\pi}} X_0(\theta)\left(\widetilde{A}_{2000}z_1^2+\widetilde{A}_{0200}z_2^2+\widetilde{A}_{1100}z_1z_2\right.\vspace{0.2cm}\\
\left.
+\widetilde{A}_{1010}z_1z_3+\widetilde{A}_{1001}z_1z_4+\widetilde{A}_{0110}z_2z_3+\widetilde{A}_{0101}z_2z_4
\right),
\end{array}\right.
&n=2n_1, n_2=3n_1,\vspace{0.3cm}\\
%---------------------------
\frac{1}{\sqrt{2\ell\pi}}X_0(\theta)\left(\widetilde{A}_{0020}z_3^2+\widetilde{A}_{0002}z_4^2+\widetilde{A}_{0011}z_3z_4\right),& n=2n_2,\vspace{0.3cm}\\
%------------------------------------
\left\{\begin{array}{l}
\frac{1}{\sqrt{2\ell\pi}}X_0(\theta)\left(\widehat{A}_{1010}z_1z_3+\widehat{A}_{1001}z_1z_4+\widehat{A}_{0110}z_2z_3\right.\vspace{0.2cm}\\
\left.+\widehat{A}_{0101}z_2z_4\right),
\end{array}\right.
& n=n_1+n_2,\vspace{0.3cm}\\
%-----------------------------
\left\{\begin{array}{l}
\frac{1}{\sqrt{2\ell\pi}}X_0(\theta)\left(\widetilde{A}_{1010}z_1z_3+\widetilde{A}_{1001}z_1z_4\right.\vspace{0.2cm}\\
\left.+\widetilde{A}_{0110}z_2z_3+\widetilde{A}_{0101}z_2z_4\right),
\end{array}\right.
& n=n_2-n_1,    n_2\not=2n_1,   n_2\not=3n_1,\vspace{0.2cm}\\
%-----------------------------
\left\{\begin{array}{l}
\frac{1}{\sqrt{2\ell\pi}}\left(X_0(\theta)I_2- \Phi_{n_1}(\theta)\Psi_{n_1}(0)\right)  \left(\widetilde{A}_{1010}z_1z_3\right.\\
\left.+\widetilde{A}_{1001}z_1z_4+\widetilde{A}_{0110}z_2z_3+\widetilde{A}_{0101}z_2z_4\right),
\end{array}\right.
& n=n_2-n_1,   n_2=2n_1,\vspace{0.2cm}\\
\left\{\begin{array}{l}
\frac{1}{\sqrt{2\ell\pi}}X_0(\theta)\left(\widetilde{A}_{1010}z_1z_3+\widetilde{A}_{1001}z_1z_4+\widetilde{A}_{0110}z_2z_3\right.\vspace{0.2cm}\\
\left.+\widetilde{A}_{0101}z_2z_4 +\widetilde{A}_{2000}z_1^2+\widetilde{A}_{0200}z_2^2+\widetilde{A}_{1100}z_1z_2\right),
\end{array}\right.
& n=n_2-n_1,    n_2=3n_1,
%------------------------------
\end{cases}
\end{array}
\end{equation}
where $\widetilde{A}_{j_1j_2j_3j_4} $ is defined by  \eqref{WLDAJ1}, \eqref{WLDAJ2} and  the following \eqref{WLDAJ3}
\begin{equation}
\label{WLDAJ3}
\left\{\begin{array}{l}
\widetilde{A}_{j_1j_2j_3j_4}=A_{j_1j_2j_3j_4}-\frac{n_2^2}{\ell^2}\left(A_{j_1j_2j_3j_4}^{(d,1)}+ A_{j_1j_2j_3j_4}^{(d,2)}\right),\vspace{0.2cm}\\
j_3, j_4=0, 1, 2, \quad j_3+j_4=2, \quad   j_1=j_2=0,
\end{array}\right.
\end{equation}
and
\begin{equation}
\label{HATAJ}
\left\{\begin{array}{l}
\widehat{A}_{j_1j_2j_3j_4}=A_{j_1j_2j_3j_4}-\frac{n_1n_2}{\ell^2}A_{j_1j_2j_3j_4}^{(d,1)}-\frac{n_1^2}{\ell^2}A_{j_1j_2j_3j_4}^{(d,2)}-\frac{n_2^2}{\ell^2}A_{j_1j_2j_3j_4}^{(d,3)},\vspace{0.2cm}\\
j_1, j_2, j_3, j_4=0, 1, \quad j_1+j_2=1, \quad  j_3+j_4=1.
\end{array}\right.
\end{equation}

Hence, from \eqref{M22HJ}, \eqref{EQF22B} and matching the
coefficients of $z_1^2,z_1z_2,z_1z_3,z_2z_3,z_3^2$, we have
\begin{equation}
\label{H0}
n=0,~\left\{\begin{array}{ll}
z_1^2: &\begin{cases}
\dot{h}_{0,2000}(\theta)-2i\omega_{1c} h_{0,2000}(\theta)=(0\quad 0)^T,\vspace{0.1cm}\\
\dot{h}_{0,2000}(0)-L_0(h_{0,2000}(\theta))=\frac{1}{\sqrt{\ell\pi}}A_{2000},
\end{cases}\vspace{0.2cm}\\
z_3^2:&
\begin{cases}
\dot{h}_{0,0020}(\theta)-2i\omega_{2c} h_{0,0020}(\theta)=(0\quad 0)^T,\vspace{0.1cm}\\
\dot{h}_{0,0020}(0)-L_0\left(h_{0,0020}(\theta)\right)=\frac{1}{\sqrt{\ell\pi}}A_{0020},
\end{cases}
\vspace{0.2cm}\\
z_1z_2:&\begin{cases}
\dot{h}_{0,1100}(\theta)=(0\quad 0)^T,\vspace{0.1cm}\\
\dot{h}_{0,1100}(0)-L_0(h_{0,1100}(\theta))=\frac{1}{\sqrt{\ell\pi}}A_{1100},
\end{cases}\vspace{0.2cm}\\
z_3z_4:&\begin{cases}
\dot{h}_{0,0011}(\theta)=(0\quad 0)^T,\vspace{0.1cm}\\
\dot{h}_{0,0011}(0)-L_0(h_{0,0011}(\theta))=\frac{1}{\sqrt{\ell\pi}}A_{0011},
\end{cases}
\end{array}
\right.
\end{equation}
%----------------------------------------------------------------
\begin{equation}
\label{H2N11}
\begin{array}{l}
n=2n_1,\vspace{0.2cm}\\
n_2\not=2n_1,
\end{array}
~\left\{\begin{array}{ll}
z_1^2: &\begin{cases}
\dot{h}_{2n_1,2000}(\theta)-2i\omega_{1c} h_{2n_1,2000}(\theta)=(0\quad 0)^T,\vspace{0.1cm}\\
\dot{h}_{2n_1,2000}(0)-\mathscr{L}_0(h_{2n_1,2000}(\theta))=\frac{1}{\sqrt{2\ell\pi}}\widetilde{A}_{2000},
\end{cases}\vspace{0.2cm}\\
z_1z_2:&\begin{cases}
\dot{h}_{2n_1,1100}(\theta)=(0\quad 0)^T,\vspace{0.1cm}\\
\dot{h}_{2n_1,1100}(0)-\mathscr{L}_0(h_{2n_1,1100}(\theta))=\frac{1}{\sqrt{2\ell\pi}}\widetilde{A}_{1100},
\end{cases}\vspace{0.2cm}\\
z_3z_4:&\begin{cases}
\dot{h}_{2n_1,0011}(\theta)=(0\quad 0)^T,\vspace{0.1cm}\\
\dot{h}_{2n_1,0011}(0)-\mathscr{L}_0(h_{2n_1,0011}(\theta))=(0\quad 0)^T,
\end{cases}
\end{array}
\right.
\end{equation}
%--------------------------------------------------------------------
\begin{equation}
\label{H2N12}
\begin{array}{l}
n=2n_1,\vspace{0.2cm}\\
n_2=2n_1,
\end{array}
~\left\{\begin{array}{ll}
z_1^2: &\begin{cases}
\dot{h}_{2n_1,2000}(\theta)-2i\omega_{1c} h_{2n_1,2000}(\theta)=\frac{1}{\sqrt{2\ell\pi}}\Phi_{n_2}(\theta)\Psi_{n_2}(0)\widetilde{A}_{2000},\vspace{0.1cm}\\
\dot{h}_{2n_1,2000}(0)-\mathscr{L}_0(h_{2n_1,2000}(\theta))=\frac{1}{\sqrt{2\ell\pi}}\widetilde{A}_{2000},
\end{cases}\vspace{0.2cm}\\
z_1z_2:&\begin{cases}
\dot{h}_{2n_1,1100}(\theta)=\frac{1}{\sqrt{2\ell\pi}}\Phi_{n_2}(\theta)\Psi_{n_2}(0)\widetilde{A}_{1100},\vspace{0.1cm}\\
\dot{h}_{2n_1,1100}(0)-\mathscr{L}_0(h_{2n_1,1100}(\theta))=\frac{1}{\sqrt{2\ell\pi}}\widetilde{A}_{1100},
\end{cases}\vspace{0.2cm}\\
z_3z_4:&\begin{cases}
\dot{h}_{2n_1,0011}(\theta)=(0, 0)^T,\vspace{0.1cm}\\
\dot{h}_{2n_1,0011}(0)-\mathscr{L}_0(h_{2n_1,0011}(\theta))=(0\quad 0)^T,
\end{cases}
\end{array}
\right.
\end{equation}
%-------------------------------------------------------------------------

%----------------------------------------------------------------------------
\begin{equation}
\label{H2N2}
n=2n_2,~\left\{\begin{array}{ll}
z_3^2: &\begin{cases}
\dot{h}_{2n_2,0020}(\theta)-2i\omega_{2c} h_{2n_2,0020}(\theta)=(0\quad 0)^T,\vspace{0.1cm}\\
\dot{h}_{2n_2,0020}(0)-\mathscr{L}_0(h_{2n_2,0020}(\theta))=\frac{1}{\sqrt{2\ell\pi}}\widetilde{A}_{0020},
\end{cases}\vspace{0.2cm}\\
z_1z_2:&\begin{cases}
\dot{h}_{2n_2,1100}(\theta)=(0, 0)^T,\vspace{0.1cm}\\
\dot{h}_{2n_2,1100}(0)-\mathscr{L}_0(h_{2n_2,1100}(\theta))=(0\quad 0)^T,
\end{cases}\vspace{0.2cm}\\
z_3z_4:&\begin{cases}
\dot{h}_{2n_2,0011}(\theta)=(0\quad 0)^T,\vspace{0.1cm}\\
\dot{h}_{2n_2,0011}(0)-\mathscr{L}_0(h_{2n_2,0011}(\theta))=\frac{1}{\sqrt{2\ell\pi}}\widetilde{A}_{0011},
\end{cases}
\end{array}
\right.
\end{equation}
 %----------------------------------------------------------------------------------------------
\begin{equation}
\label{HN1PN2}
{\small
n=n_1+n_2,\left\{\begin{array}{ll}
z_1z_3: &
\begin{cases}
\dot{h}_{n_1+n_2,1010}(\theta)-i\left(\omega_{1c}+\omega_{2c}\right)h_{n_1+n_2,1010}(\theta)=(0\quad 0)^T,\vspace{0.1cm}\\
\dot{h}_{n_1+n_2,1010}(0)-\mathscr{L}_0\left(h_{n_1+n_2,1010}(\theta)\right)
=\frac{1}{\sqrt{2\ell\pi}}\widehat{A}_{1010},
\end{cases} \vspace{0.2cm}\\
z_1z_4:&
\begin{cases}
\dot{h}_{n_1+n_2,1001}(\theta)-i\left(\omega_{1c}-\omega_{2c}\right)h_{n_1+n_2,1001}(\theta)=(0\quad 0)^T,\vspace{0.1cm}\\
\dot{h}_{n_1+n_2,1001}(0)-\mathscr{L}_0\left(h_{n_1+n_2,1001}(\theta)\right)
=\frac{1}{\sqrt{2\ell\pi}}\widehat{A}_{1001},
\end{cases} \vspace{0.2cm}\\
z_2z_3:&
\begin{cases}
\dot{h}_{n_1+n_2,0110}(\theta)-i\left(\omega_{2c}-\omega_{1c}\right)h_{n_1+n_2,0110}(\theta)=(0\quad 0)^T,\vspace{0.1cm}\\
\dot{h}_{n_1+n_2,0110}(0)-\mathscr{L}_0\left(h_{n_1+n_2,0110}(\theta)\right)
=\frac{1}{\sqrt{2\ell\pi}}\widehat{A}_{0110},
\end{cases}
\end{array}
\right.}
\end{equation}
 %----------------------------------------------------------------------------------------------
\begin{equation}
\label{HN2MN11}
{\small
\begin{array}{l}
n=n_2-n_1,\vspace{0.2cm}\\
n_2\not=2n_1,
\end{array}
\left\{\begin{array}{ll}
z_1z_3: &
\begin{cases}
\dot{h}_{n_2-n_1,1010}(\theta)-i\left(\omega_{1c}+\omega_{2c}\right)h_{n_2-n_1,1010}(\theta)=(0\quad 0)^T,\vspace{0.1cm}\\
\dot{h}_{n_2-n_1,1010}(0)-\mathscr{L}_0\left(h_{n_2-n_1,1010}(\theta)\right)
=\frac{1}{\sqrt{2\ell\pi}}\widetilde{A}_{1010},
\end{cases} \vspace{0.2cm}\\
z_1z_4:&
\begin{cases}
\dot{h}_{n_2-n_1,1001}(\theta)-i\left(\omega_{1c}-\omega_{2c}\right)h_{n_2-n_1,1001}(\theta)=(0\quad 0)^T,\vspace{0.1cm}\\
\dot{h}_{n_2-n_1,1001}(0)-\mathscr{L}_0\left(h_{n_2-n_1,1001}(\theta)\right)
=\frac{1}{\sqrt{2\ell\pi}}\widetilde{A}_{1001},
\end{cases} \vspace{0.2cm}\\
z_2z_3:&
\begin{cases}
\dot{h}_{n_2-n_1,0110}(\theta)-i\left(\omega_{2c}-\omega_{1c}\right)h_{n_2-n_1,0110}(\theta)=(0\quad 0)^T,\vspace{0.1cm}\\
\dot{h}_{n_2-n_1,0110}(0)-\mathscr{L}_0\left(h_{n_2-n_1,0110}(\theta)\right)
=\frac{1}{\sqrt{2\ell\pi}}\widetilde{A}_{0110},
\end{cases}
\end{array}
\right.}
\end{equation}

 %----------------------------------------------------------------------------------------------
\begin{equation}
\label{HN2MN12}
{\small
\begin{array}{l}
n=n_2-n_1,\vspace{0.2cm}\\
n_2=2n_1
\end{array}
\left\{\begin{array}{ll}
z_1z_3: &
\begin{cases}
\dot{h}_{n_2-n_1,1010}(\theta)-i\left(\omega_{1c}+\omega_{2c}\right)h_{n_2-n_1,1010}(\theta)=\frac{1}{\sqrt{2\ell\pi}} \Phi_{n_1}(\theta)\Psi_{n_1}(0)\widetilde{A}_{1010} ,\vspace{0.1cm}\\
\dot{h}_{n_2-n_1,1010}(0)-\mathscr{L}_0\left(h_{n_2-n_1,1010}(\theta)\right)
=\frac{1}{\sqrt{2\ell\pi}}\widetilde{A}_{1010},
\end{cases} \vspace{0.2cm}\\
z_1z_4:&
\begin{cases}
\dot{h}_{n_2-n_1,1001}(\theta)-i\left(\omega_{1c}-\omega_{2c}\right)h_{n_2-n_1,1001}(\theta)=\frac{1}{\sqrt{2\ell\pi}} \Phi_{n_1}(\theta)\Psi_{n_1}(0)\widetilde{A}_{1001} ,\vspace{0.1cm}\\
\dot{h}_{n_2-n_1,1001}(0)-\mathscr{L}_0\left(h_{n_2-n_1,1001}(\theta)\right)
=\frac{1}{\sqrt{2\ell\pi}}\widetilde{A}_{1001},
\end{cases} \vspace{0.2cm}\\
z_2z_3:&
\begin{cases}
\dot{h}_{n_2-n_1,0110}(\theta)-i\left(\omega_{2c}-\omega_{1c}\right)h_{n_2-n_1,0110}(\theta)=\frac{1}{\sqrt{2\ell\pi}} \Phi_{n_1}(\theta)\Psi_{n_1}(0)\widetilde{A}_{0110} ,\vspace{0.1cm}\\
\dot{h}_{n_2-n_1,0110}(0)-\mathscr{L}_0\left(h_{n_2-n_1,0110}(\theta)\right)
=\frac{1}{\sqrt{2\ell\pi}}\widetilde{A}_{0110}.
\end{cases}
\end{array}
\right.}
\end{equation}
%-----------------------------------------------------
Solving \eqref{H0}, \eqref{H2N11}, \eqref{H2N12},\eqref{H2N2},\eqref{HN1PN2},\eqref{HN2MN11} and \eqref{HN2MN12}, we obtain
%----------------------------------------------------------------------------
\begin{equation}
\label{H0EXPR}
\left\{
\begin{array}{l}
h_{0,2000}(\theta) =\frac{1}{\sqrt{\ell\pi}}\left(\widetilde{\mathcal{M}}_0(2i\omega_{1c})\right)^{-1}A_{2000}e^{2i\omega_{1c}\theta},\vspace{0.2cm}\\
h_{0,0020}(\theta) =\frac{1}{\sqrt{\ell\pi}}\left(\widetilde{\mathcal{M}}_0(2i\omega_{2c})\right)^{-1}A_{0020}e^{2i\omega_{2c}\theta},  \vspace{0.2cm}\\
h_{0,1100}(\theta) =\frac{1}{\sqrt{\ell\pi}}\left(\widetilde{\mathcal{M}}_0(0)\right)^{-1}A_{1100}, \vspace{0.2cm}\\
h_{0,0011}(\theta) =\frac{1}{\sqrt{\ell\pi}}\left(\widetilde{\mathcal{M}}_0(0)\right)^{-1}A_{0011},
\end{array}
\right.
\end{equation}
%----------------------------------------------------------------------------
\begin{equation}
\label{H2N12EXPR1}
\left\{
\begin{array}{lll}
h_{2n_1,2000}(\theta) &=&\frac{1}{\sqrt{2\ell\pi}}\left(\widetilde{\mathcal{M}}_{2n_1}(2i\omega_{1c})\right)^{-1}\widetilde{A}_{2000}e^{2i\omega_{1c}\theta}\vspace{0.2cm}\\
h_{2n_1,1100}(\theta) &=&\frac{1}{\sqrt{2\ell\pi}}\left(\widetilde{\mathcal{M}}_{2n_1}(0)\right)^{-1}\widetilde{A}_{1100}, \vspace{0.2cm}\\
h_{2n_1,0011}(\theta) &=& (0\quad 0)^T,
\end{array} ~~\mbox{for}~~n_2\not=2n_1,
\right.
\end{equation}
%------------------------------------------------------------------------------
\begin{equation}
\label{H2N2EXPR}
\left\{
\begin{array}{l}
h_{2n_2,0020}(\theta) =\frac{1}{\sqrt{2\ell\pi}}\left(\widetilde{\mathcal{M}}_{2n_2}(2i\omega_{2c})\right)^{-1}\widetilde{A}_{0020}e^{2i\omega_{2c}\theta}\vspace{0.2cm}\\

h_{2n_2,1100}(\theta) =(0, 0)^T, \vspace{0.2cm}\\

h_{2n_2,0011}(\theta) =\frac{1}{\sqrt{2\ell\pi}}
\left(\widetilde{\mathcal{M}}_{2n_2}(0)\right)^{-1}\widetilde{A}_{0011} ,
\end{array}
\right.
\end{equation}
%--------------------------------------------------------------
\begin{equation}
\label{HN1N2EXPR}
\left\{
\begin{array}{l}
h_{n_1+n_2,1010}(\theta) =\frac{1}{\sqrt{2\ell\pi}}\left(\widetilde{\mathcal{M}}_{n_1+n_2}\left(i\left(\omega_{1c}+\omega_{2c}\right)\right)\right)^{-1}\widehat{A}_{1010}e^{i\left(\omega_{1c}+\omega_{2c}\right)\theta}\vspace{0.2cm}\\

h_{n_1+n_2,1001}(\theta) =\frac{1}{\sqrt{2\ell\pi}}\left(\widetilde{\mathcal{M}}_{n_1+n_2}\left(i\left(\omega_{1c}-\omega_{2c}\right)\right)\right)^{-1}\widehat{A}_{1001}e^{i\left(\omega_{1c}-\omega_{2c}\right)\theta}\vspace{0.2cm}\\

h_{n_1+n_2,0110}(\theta) =\frac{1}{\sqrt{2\ell\pi}}\left(\widetilde{\mathcal{M}}_{n_1+n_2}\left(i\left(\omega_{2c}-\omega_{1c}\right)\right)\right)^{-1}\widehat{A}_{0110}e^{i\left(\omega_{2c}-\omega_{1c}\right)\theta},
\end{array}
\right.
\end{equation}
%---------------------------------------
%------------------------------------------------------------------------------
$$
\left\{
\begin{array}{lll}
&&h_{2n_1,2000}(\theta)\vspace{0.2cm}\\
 &=&\frac{1}{\sqrt{2\ell\pi}}\left(\widetilde{\mathcal{M}}_{2n_1}\left(2i\omega_{1c}\right)\right)^{-1}\left( \widetilde{A}_{2000}
-C_1 \widetilde{\mathcal{M}}_{2n_1}\left(i\omega_{2c}\right)\phi_{n_2}(0)\right.\vspace{0.2cm}\\
&&
  \left.-C_2 \widetilde{\mathcal{M}}_{2n_1}\left(-i\omega_{2c}\right)\overline{\phi}_{n_2}(0) \right)e^{2i \omega_{1c}\theta}\vspace{0.2cm}\\
 &&
 +\frac{1}{\sqrt{2\ell\pi}}C_1\phi _{n_2}(\theta)+\frac{1}{\sqrt{2\ell\pi}}C_2\overline{\phi}_{n_2}(\theta),\vspace{0.2cm}\\
&&h_{2n_1,1100}(\theta)\vspace{0.2cm}\\
 &=&\frac{1}{\sqrt{2\ell\pi}}\left(\widetilde{\mathcal{M}}_{2n_1}\left(0\right)\right)^{-1}\left( \widetilde{A}_{1100}
-C_3 \widetilde{\mathcal{M}}_{2n_1}\left(i\omega_{2c}\right)\phi_{n_2}(0)\right.\vspace{0.2cm}\\
&&
  \left.-C_4 \widetilde{\mathcal{M}}_{2n_1}\left(-i\omega_{2c}\right)\overline{\phi}_{n_2}(0) \right)
 +\frac{1}{\sqrt{2\ell\pi}}C_3\phi _{n_2}(\theta)+\frac{1}{\sqrt{2\ell\pi}}C_4\overline{\phi}_{n_2}(\theta),\vspace{0.2cm}\\
&&h_{2n_1,0011}(\theta) = (0\quad 0)^T,
\end{array} ~~\mbox{for}~~n_2=2n_1,
\right.
$$
where
$$C_1=\frac{1}{i(\omega_{2c}-2\omega_{1c})} \psi_{n_2}^T(0)\widetilde{A}_{2000} , \quad C_2=-\frac{1}{i(\omega_{2c}+2\omega_{1c})} \overline{\psi}_{n_2}^T(0)\widetilde{A}_{2000}, $$
$$C_3=\frac{1}{i\omega_{2c}} \psi_{n_2}^T(0)\widetilde{A}_{1100} , \quad C_4=-\frac{1}{i\omega_{2c}} \overline{\psi}_{n_2}^T(0)\widetilde{A}_{1100}, $$
%------------------------------------------------------------------------------
$$
\left\{
\begin{array}{lll}
h_{n_2-n_1,1010}(\theta) &=&\frac{1}{\sqrt{2\ell\pi}}\left(\widetilde{\mathcal{M}}_{n_2-n_1}\left(i(\omega_{1c}+\omega_{2c})\right)\right)^{-1}\widetilde{A}_{1010}e^{i\left(\omega_{1c}+\omega_{2c}\right)\theta},\vspace{0.2cm}\\
h_{n_2-n_1,1001}(\theta) &=&\frac{1}{\sqrt{2\ell\pi}}\left(\widetilde{\mathcal{M}}_{n_2-n_1}\left(i(\omega_{1c}-\omega_{2c})\right)\right)^{-1}\widetilde{A}_{1001}e^{i\left(\omega_{1c}-\omega_{2c}\right)\theta}, \vspace{0.2cm}\\
h_{n_2-n_1,0110}(\theta) &=&\frac{1}{\sqrt{2\ell\pi}}
\left(\widetilde{\mathcal{M}}_{n_2-n_1}\left(i(\omega_{2c}-\omega_{1c})\right)\right)^{-1}\widetilde{A}_{0110}e^{i\left(\omega_{2c}-\omega_{1c}\right)\theta},
\end{array} ~~\mbox{for}~~n_2\not=2n_1,
\right.
$$
%------------------------------------------------------
and for $n_2 =2n_1,$ we have
$$
\left\{
\begin{array}{lll}
h_{n_2-n_1,1010}(\theta)
 &=&\frac{1}{\sqrt{2\ell\pi}}\left(\widetilde{\mathcal{M}}_{n_2-n_1}\left(i(\omega_{1c}+\omega_{2c})\right)\right)^{-1} \left(\widetilde{A}_{1010}-C_5 \widetilde{\mathcal{M}}_{n_2-n_1}\left(i\omega_{1c}\right)\phi_{n_1}(0)\right.\vspace{0.2cm}\\
   &&\left.
   -C_6 \widetilde{\mathcal{M}}_{n_2-n_1}\left(-i\omega_{1c}\right)\overline{\phi}_{n_1}(0) \right)e^{i\left(\omega_{1c}+\omega_{2c}\right)\theta}\vspace{0.2cm}\\
 &&
 +\frac{1}{\sqrt{2\ell\pi}}C_5\phi _{n_1}(\theta)+\frac{1}{\sqrt{2\ell\pi}}C_6\overline{\phi}_{n_1}(\theta),\vspace{0.2cm}\\
h_{n_2-n_1,1001}(\theta)
 &=&\frac{1}{\sqrt{2\ell\pi}}\left(\widetilde{\mathcal{M}}_{n_2-n_1}\left(i(\omega_{1c}-\omega_{2c})\right)\right)^{-1} \left(\widetilde{A}_{1001}-C_7 \widetilde{\mathcal{M}}_{n_2-n_1}\left(i\omega_{1c}\right)\phi_{n_1}(0)\right.\vspace{0.2cm}\\
   &&\left.
   -C_8 \widetilde{\mathcal{M}}_{n_2-n_1}\left(-i\omega_{1c}\right)\overline{\phi}_{n_1}(0) \right)e^{i\left(\omega_{1c}-\omega_{2c}\right)\theta}\vspace{0.2cm}\\
 &&
 +\frac{1}{\sqrt{2\ell\pi}}C_7\phi _{n_1}(\theta)+\frac{1}{\sqrt{2\ell\pi}}C_8\overline{\phi}_{n_1}(\theta),\vspace{0.2cm}\\
h_{n_2-n_1,0110}(\theta)
 &=&\frac{1}{\sqrt{2\ell\pi}}\left(\widetilde{\mathcal{M}}_{n_2-n_1}\left(i(\omega_{2c}-\omega_{1c})\right)\right)^{-1} \left(\widetilde{A}_{0110}-C_9 \widetilde{\mathcal{M}}_{n_2-n_1}\left(i\omega_{1c}\right)\phi_{n_1}(0)\right.\vspace{0.2cm}\\
   &&\left.
   -C_{10} \widetilde{\mathcal{M}}_{n_2-n_1}\left(-i\omega_{1c}\right)\overline{\phi}_{n_1}(0) \right)e^{i\left(\omega_{2c}-\omega_{1c}\right)\theta}\vspace{0.2cm}\\
 &&
 +\frac{1}{\sqrt{2\ell\pi}}C_9\phi _{n_1}(\theta)+\frac{1}{\sqrt{2\ell\pi}}C_{10}\overline{\phi}_{n_1}(\theta),\vspace{0.2cm}
\end{array}
\right.
$$
where
$$
\begin{array}{ll}
C_5=-\frac{1}{i\omega_{2c}} \psi_{n_1}^T(0)\widetilde{A}_{1010} , &
  C_6=-\frac{1}{i(2\omega_{1c}+\omega_{2c})} \overline{\psi}_{n_1}^T(0)\widetilde{A}_{1010}, \vspace{0.2cm}\\
C_7=\frac{1}{i\omega_{2c}} \psi_{n_1}^T(0)\widetilde{A}_{1001}, & C_8=-\frac{1}{i(2\omega_{1c}-\omega_{2c})} \overline{\psi}_{n_1}^T(0)\widetilde{A}_{1001},\vspace{0.2cm}\\
C_9=\frac{1}{i(2\omega_{1c}-\omega_{2c})} \psi_{n_1}^T(0)\widetilde{A}_{0110} , & C_{10}=-\frac{1}{i\omega_{2c}} \overline{\psi}_{n_1}^T(0)\widetilde{A}_{0110} .
\end{array}
 $$
 %----------------------------------------------------------------------------------------------
\subsection{Case 2: $n_1= n_2$. }
 %----------------------------------------------------------------------------------------------
 By \eqref{PIX0},  \eqref{TLDF2}, \eqref{F2EXP} and \eqref{F2DEXP},
for $n_2= n_1$,  we have
\begin{equation}
\label{EQF22Bn1En2}
\begin{array}{lll}
 &&
\left(\begin{array}{c}
\left[f_2^2(z,0,0),\beta_n^{(1)}\right]\vspace{0.1cm}\\
\left[f_2^2(z,0,0),\beta_n^{(2)}\right]
\end{array}\right)\vspace{0.2cm}\\
&=&\begin{cases}
\left\{\begin{array}{l}
\frac{1}{\sqrt{\ell\pi}} X_0(\theta)\left(A_{2000}z_1^2+A_{0200}z_2^2+A_{0020}z_3^2+A_{0002}z_4^2\right.\vspace{0.2cm}\\
+A_{1100}z_1z_2+A_{0011}z_3z_4
+\breve{A}_{1010}z_1z_3+\breve{A}_{1001}z_1z_4
\vspace{0.2cm}\\
\left.
+\breve{A}_{0110}z_2z_3+\breve{A}_{0101}z_2z_4
\right),
\end{array}\right.
& n=0,\vspace{0.3cm}\\
%---------------------------------
\left\{\begin{array}{l}
\frac{1}{\sqrt{2\ell\pi}}X_0(\theta)\left(\widetilde{A}_{2000}z_1^2+\widetilde{A}_{0200}z_2^2+\widetilde{A}_{1100}z_1z_2\right.\vspace{0.2cm}\\
+\widetilde{A}_{0020}z_3^2+\widetilde{A}_{0002}z_4^2+\widetilde{A}_{0011}z_3z_4+\widehat{A}_{1010}z_1z_3\vspace{0.2cm}\\
\left.+\widehat{A}_{1001}z_1z_4+\widehat{A}_{0110}z_2z_3+\widehat{A}_{0101}z_2z_4
\right),
\end{array}\right.
&n=2n_1,\vspace{0.3cm}\\
\end{cases}
\end{array}
\end{equation}
where $\widetilde{A}_{j_1j_2j_3j_4}$  and $\widehat{A}_{j_1j_2j_3j_4}$ are defined   by  \eqref{WLDAJ1}, \eqref{WLDAJ2}, 
\eqref{WLDAJ3} and \eqref{HATAJ}, and $\breve{A}_{j_1j_2j_3j_4}$ is defined by the following
%--------------------------------------
\begin{equation*}
\left\{\begin{array}{l}
\breve{A}_{j_1j_2j_3j_4}=A_{j_1j_2j_3j_4}+\frac{n_1^2}{\ell^2} \left(A_{j_1j_2j_3j_4}^{(d,1)} -  A_{j_1j_2j_3j_4}^{(d,2)}- A_{j_1j_2j_3j_4}^{(d,3)}\right) ,\vspace{0.2cm}\\
j_1, j_2, j_3, j_4=0, 1, \quad j_1+j_2=1, \quad   j_3+j_4=1.
\end{array}\right.
\end{equation*}
%---------------------------------------
%--------------------------------------------------------------------------------------------------
From  \eqref{M22HJ} and \eqref{EQF22Bn1En2}, we have
\begin{equation}
\label{H0N1EN2}
\left\{
\begin{array}{l}
h_{0,2000}(\theta) =\frac{1}{\sqrt{\ell\pi}}\left(\widetilde{\mathcal{M}}_0(2i\omega_{1c})\right)^{-1}A_{2000}e^{2i\omega_{1c}\theta},\vspace{0.2cm}\\
h_{0,0020}(\theta) =\frac{1}{\sqrt{\ell\pi}}\left(\widetilde{\mathcal{M}}_0(2i\omega_{2c})\right)^{-1}A_{0020}e^{2i\omega_{2c}\theta},  \vspace{0.2cm}\\
h_{0,1100}(\theta) =\frac{1}{\sqrt{\ell\pi}}\left(\widetilde{\mathcal{M}}_0(0)\right)^{-1}A_{1100}, \vspace{0.2cm}\\
h_{0,0011}(\theta) =\frac{1}{\sqrt{\ell\pi}}\left(\widetilde{\mathcal{M}}_0(0)\right)^{-1}A_{0011},
\vspace{0.2cm}\\
h_{0,1010}(\theta) =\frac{1}{\sqrt{\ell\pi}}\left(\widetilde{\mathcal{M}}_0\left(i\left(\omega_{1c}+\omega_{2c}\right) \right)\right)^{-1}\breve{A}_{1010}e^{i\left(\omega_{1c}+\omega_{2c}\right)\theta},
\vspace{0.2cm}\\
h_{0,1001}(\theta) =\frac{1}{\sqrt{\ell\pi}}\left(\widetilde{\mathcal{M}}_0\left(i\left(\omega_{1c}-\omega_{2c}\right) \right)\right)^{-1}\breve{A}_{1001}e^{i\left(\omega_{1c}-\omega_{2c}\right)\theta},
\vspace{0.2cm}\\
h_{0,0110}(\theta) =\frac{1}{\sqrt{\ell\pi}}\left(\widetilde{\mathcal{M}}_0\left(i\left(\omega_{2c}-\omega_{1c}\right) \right)\right)^{-1}\breve{A}_{0110}e^{i\left(\omega_{2c}-\omega_{1c}\right)\theta},
\end{array}
\right.
\end{equation}
%----------------------------------------------------------------------------
and
\begin{equation}
\label{H2N1N1EN2}
\left\{
\begin{array}{l}
h_{2n_1,2000}(\theta) =\frac{1}{\sqrt{2\ell\pi}}\left(\widetilde{\mathcal{M}}_{2n_1}(2i\omega_{1c})\right)^{-1}\widetilde{A}_{2000}e^{2i\omega_{1c}\theta},
\vspace{0.2cm}\\
h_{2n_1,0020}(\theta) =\frac{1}{\sqrt{2\ell\pi}}\left(\widetilde{\mathcal{M}}_{2n_1}(2i\omega_{2c})\right)^{-1}\widetilde{A}_{0020}e^{2i\omega_{2c}\theta}
\vspace{0.2cm}\\
h_{2n_1,1100}(\theta) =\frac{1}{\sqrt{2\ell\pi}}\left(\widetilde{\mathcal{M}}_{2n_1}(0)\right)^{-1}\widetilde{A}_{1100},
\vspace{0.2cm}\\
h_{2n_1,0011}(\theta) =\frac{1}{\sqrt{2\ell\pi}}
\left(\widetilde{\mathcal{M}}_{2n_1}(0)\right)^{-1}\widetilde{A}_{0011} ,
\vspace{0.2cm}\\
h_{2n_1,1010}(\theta) =\frac{1}{\sqrt{2\ell\pi}}\left(\widetilde{\mathcal{M}}_{2n_1}\left(i\left(\omega_{1c}+\omega_{2c}\right)\right)\right)^{-1}\widehat{A}_{1010}e^{i\left(\omega_{1c}+\omega_{2c}\right)\theta}
\vspace{0.2cm}\\
h_{2n_1,1001}(\theta) =\frac{1}{\sqrt{2\ell\pi}}\left(\widetilde{\mathcal{M}}_{2n_1}\left(i\left(\omega_{1c}-\omega_{2c}\right)\right)\right)^{-1}\widehat{A}_{1001}e^{i\left(\omega_{1c}-\omega_{2c}\right)\theta}
\vspace{0.2cm}\\
h_{2n_1,0110}(\theta) =\frac{1}{\sqrt{2\ell\pi}}\left(\widetilde{\mathcal{M}}_{2n_1}\left(i\left(\omega_{2c}-\omega_{1c}\right)\right)\right)^{-1}\widehat{A}_{0110}e^{i\left(\omega_{2c}-\omega_{1c}\right)\theta}.
\end{array}
\right.
\end{equation}
%------------------------------------------------------------------------------
Notice that for $n_1=n_2$, $2n_1=2n_2=n_1+n_2$ and $n_2-n_1=0$. Thus, substituting \eqref{H0N1EN2} and \eqref{H2N1N1EN2} into \eqref{EIJ}, \eqref{EIJD1} and \eqref{EIJD2}, we can calculate  $E_{ij}$ and $E_{ij}^d$.

 \end{appendices}
%----------------------------------------------------------------------------------------
%-----------------------------------------------------------------------------------------
\bibliographystyle{elsart-num-sort}
\bibliography{Bifurcation-ReferencesV1}
\end{document}